\documentclass[a4paper,10pt]{aart}

 \usepackage[top=4cm, bottom=2.7cm, left=3cm, right=3cm]{geometry}

\usepackage[titletoc,toc,title]{appendix}

\usepackage[centertags]{amsmath}
\usepackage{amsmath}
\usepackage[british]{babel}
\usepackage{amsfonts}
\usepackage{amssymb}
\usepackage{amsthm}

\usepackage{newlfont}
\usepackage{color}

 \usepackage{bbm}

\usepackage{stmaryrd}
\usepackage{mathrsfs}
\usepackage{pstricks,pst-node}

\usepackage{fancyhdr}
\usepackage{graphicx}
\usepackage{fancybox}
\usepackage{geometry}
\usepackage{setspace}
\usepackage{cleveref}

\newtheorem{thm}{Theorem}
\newtheorem{cor}[thm]{Corollary}
\newtheorem{lem}[thm]{Lemma}
\newtheorem{prop}[thm]{Proposition}
\newtheorem{defn}[thm]{Definition}

\theoremstyle{definition}

\theoremstyle{remark}
\newtheorem{rem}[thm]{Remark}

\newcommand{\R} {\mathbb{R}}
\newcommand{\C} {\mathbb{C}}
\newcommand{\N} {\mathbb{N}}
\newcommand{\Z} {\mathbb{Z}}

\newcommand{\E} {\mathbb{E}}
\newcommand{\T} {\mathbb{T}}
\newcommand{\p} {\mathbb{P}}

\newcommand{\im}{\mathrm{Im }\,}
\newcommand{\re}{\mathrm{Re }\,}

\newcommand{\dd}{\mathrm{d}}
\newcommand{\ii}{\mathrm{i}}

\renewcommand{\a}{\mathrm{a}}
\renewcommand{\b}{\mathrm{b}}

\newcommand{\deq}{\mathrel{\mathop:}=}

\newcommand{\beq}{ \begin{equation} }
\newcommand{\eeq}{ \end{equation} }

\newcommand{\wt}{\widetilde}

\newcommand{\caD}{{\mathcal D}}
\newcommand{\caE}{{\mathcal E}}
\newcommand{\caF}{{\mathcal F}}
\newcommand{\caG}{{\mathcal G}}

\newcommand{\caL}{{\mathcal L}}

\newcommand{\caO}{{\mathcal O}}
\newcommand{\caP}{{\mathcal P}}

\newcommand{\caR}{{\mathcal R}}

\newcommand{\caX}{{\mathcal X}}
\newcommand{\caY}{{\mathcal Y}}

\newcommand{\fb}{{\mathfrak b}}

\newcommand{\lone}{\mathbbm{1}}

\newcommand{\e}[1]{\mathrm{e}^{#1}}

\DeclareMathOperator{\Tr}{Tr}

\DeclareMathOperator{\supp}{supp}

\numberwithin{equation}{section} 
\numberwithin{thm}{section}

\begin{document}
\begin{minipage}{0.85\textwidth}
\vspace{3cm}
\end{minipage}
\begin{center}
\large\bf
Extremal Eigenvalues and Eigenvectors of Deformed Wigner Matrices
\end{center}

\vspace{1.5cm}
\begin{center}
\begin{minipage}{0.45\textwidth}
\begin{center}

Ji Oon Lee\\
\small Department of Mathematical Sciences\\
{\it Korea Advanced Institute of Science and Technology }\\
jioon.lee@kaist.edu
\end{center}
\end{minipage}
\begin{minipage}{0.45\textwidth}
\begin{center}
Kevin Schnelli\\
\small School of Mathematics\\
{\it Institute for Advanced Study }\\
kschnelli@math.ias.edu
\end{center}
\end{minipage}
\end{center}

\vspace{1.5cm}

 \begin{abstract}
We consider random matrices of the form $H = W + \lambda V$, $\lambda\in\R^+$, where $W$ is a real symmetric or complex Hermitian Wigner matrix of size $N$ and $V$ is a real bounded diagonal random matrix of size $N$ with i.i.d.\ entries that are independent of $W$. We assume subexponential decay of the distribution of the matrix entries of $W$ and we choose $\lambda \sim 1$, so that the eigenvalues of $W$ and $\lambda V$ are typically of the same order. Further, we assume that the density of the entries of $V$ is supported on a single interval and is convex near the edges of its support. In this paper we prove that there is $\lambda_+\in\R^+$ such that the largest eigenvalues of $H$ are in the limit of large $N$ determined by the order statistics of $V$ for $\lambda>\lambda_+$. In particular, the largest eigenvalue of $H$ has a Weibull distribution in the limit $N\to\infty$ if $\lambda>\lambda_+$. Moreover, for $N$ sufficiently large, we show that the eigenvectors associated to the largest eigenvalues are 
partially localized 
for $\lambda>\lambda_+$, while they are completely delocalized for $\lambda<\lambda_+$. Similar results hold for the lowest eigenvalues.

 \end{abstract}

 \vspace{15mm}
 
 {

 
 \vspace{3mm}
\textit{\today}
\vspace{3mm}
 }

\section{Introduction}

The universality of random matrices is usually divided into bulk and edge universalities. Edge universality concerns the distribution of the extreme eigenvalues. It is known that the extreme eigenvalues of a large class of Wigner matrices exhibit universal limiting behavior. The limiting distribution of the largest eigenvalue was first identified by Tracy and Widom~\cite{TW1, TW2} for the Gaussian ensembles. Edge universality for Wigner matrices has first been proved by Soshnikov~\cite{So1} (see also~\cite{SiSo1}) for real symmetric and complex Hermitian ensembles with symmetric distributions. The symmetry assumption on the entries' distribution was partially removed in~\cite{PeSo1,PeSo2}. Edge universality without any symmetry assumption
was proved in~\cite{TV2} under the condition that the distribution of the matrix elements has subexponential decay and
its first three moments match those of the Gaussian distribution. For Wigner matrices with arbitrary symmetry
class, edge universality was proven in~\cite{EKYY2} under the assumption that the entries have $12+\epsilon$ moments. Recently, a necessary and sufficient condition for the edge universality of Wigner matrices was given in~\cite{LY}.

The distribution of the largest eigenvalues of a random diagonal matrix $V$ whose entries $(v_i)$ are i.i.d.\ real random variables is given by the order statistics of $(v_i)$. The Fisher-Tippett-Gnedenko theorem (see e.g.~\cite{HF}) thus implies that the limiting distribution of the largest eigenvalue of $V$ belongs either to the Gumbel, Fr\'{e}chet or Weibull family. 

In this paper, we consider the interpolation between Wigner matrices and real diagonal random matrices. Let~$W$ be an~$N\times N$ real symmetric or complex Hermitian Wigner matrix whose centered entries have variance $N^{-1}$ and subexponential decay. Let $V$ be an $N\times N$ real diagonal random matrix whose entries are bounded i.i.d.\ random variables. For $\lambda\in\R^+$ we set
\begin{align} \label{interpolation}
H = (h_{ij}) \deq \lambda V + W\,,\qquad (1\le i,j\le N)\,.
\end{align}
The matrices $V$ and $W$ are normalized in the sense that the eigenvalues of~$V$ and~$W$ are of order one.

If $W$ belongs to the Gaussian Unitary ensemble (GUE), the model~\eqref{interpolation} is called the deformed GUE. It was shown in~\cite{J2,S1} that the edge eigenvalues of the deformed GUE are governed by the Tracy-Widom distribution for $\lambda\ll N^{-1/6}$. At $\lambda\sim N^{-1/6}$ the fluctuations of the edge eigenvalues change from the Tracy-Widom to a Gaussian distribution. More precisely, Johansson showed in~\cite{J2} that the limiting distribution of the edge eigenvalues for $\lambda={\alpha}{N^{-1/6}}$ is given by the convolution of the Tracy-Widom and the centered Gaussian distribution, with variance depending on $\alpha$. These results have not been established for the Gaussian orthogonal ensemble (GOE) or for general Wigner matrices.

In the present paper, we consider the edge behavior of the deformed model~\eqref{interpolation} in the regime $\lambda\sim 1$ with $W$ a real symmetric or complex Hermitian matrix. We show that there is, for certain $V$, yet another transition for the limiting behavior of the largest eigenvalues of $H$ as $\lambda$ varies. For simplicity, we assume that the distribution of the entries of $V$ is centered and given by the density
\begin{align}\label{babyjacobi}
\mu(v)\deq Z^{-1} (1+v)^{\a}(1-v)^{\b} d(v)\lone_{[-1, 1]}(v)\,,
\end{align}
where $-1\le \a, \b <\infty$, $d$ is a strictly positive $C^1$-function and $Z$ is a normalization constant. We primarily focus on the choices $\a,\b>1$. From our first main result, Theorem~\ref{thm:main}, it follows that there are $N$-independent constants $\lambda_+\equiv\lambda_{+}(\mu)>1$ and $L_+\equiv L_+(\mu,\lambda)>2$, such that, for $\b>1$ and $\lambda> \lambda_+$, the largest eigenvalue $\mu_1$ of $H$ satisfies
 \begin{align}\label{possible1}
  \lim_{N\to\infty}\mathbb{P}( N^{1/(\b+1)} (L_+-\mu_1) \leq x)=G_{\b+1}(x)\,,\qquad \b>1\,,\quad\lambda>\lambda_+\,,
 \end{align}
where $G_{\b+1}$ is a {\it Weibull distribution} with parameter $\b+1$; see~\eqref{Weibull}. 

However, if $\lambda<\lambda_+$, then there are $N$-independent constants $ L_+\equiv  L_+(\mu,\lambda)$ and $c\equiv c(\mu,\lambda)$, such that
\begin{align}\label{possible2}
 \lim_{N\to\infty}\mathbb{P}( N^{1/2}( L_+ -\mu_1) \leq x)=\Phi_c(x)\,,\qquad \b>1\,,\quad \lambda<\lambda_+\,,
\end{align}
where $\Phi_c$ denotes the cumulative distribution function of the centered {\it Gaussian distribution} with variance~$c$; see Appendix C. We remark that neither~\eqref{possible1} nor~\eqref{possible2} depend on the symmetry type of the Wigner matrix $W$.

The appearance of the Weibull distribution in the model~\eqref{interpolation} is indeed expected when~$\lambda$ grows sufficiently fast with $N$, since in this case the diagonal matrix dominates the spectral properties of $H$. However, it is quite surprising that the Weibull distribution already appears for $\lambda$ order one, since the local behavior of the eigenvalues in the bulk of the deformed model mainly stems from the Wigner part, and the contribution from the random diagonal part is limited to mesoscopic fluctuations of the eigenvalues; see~\cite{LS}.

Having identified two possible limiting distributions of the largest eigenvalues, it is natural to ask about the behavior of the associated eigenvectors. Before considering the deformed model, we recall that the eigenvectors of Wigner matrices with subexponential decay are completely delocalized, as was proved by Erd\H{o}s, Schlein and Yau \cite{ESY1, ESY2}. 

For deformed Wigner matrices we show that the eigenvectors of the largest eigenvalues are {\it partially localized} in the regime where the edge behavior~\eqref{possible1} holds. More precisely,  we prove that one component of the ($\ell^2$-normalized) eigenvectors associated with eigenvalues at the extreme edge carries a weight of order one, while the other components each carry a weight of order $o(1)$; see Theorem~\ref{thm:local}. If, however, the edge behavior~\eqref{possible2} holds, all eigenvectors are completely delocalized. 
Although we do not prove it explicitly, we claim that the bulk eigenvectors of the model~\eqref{interpolation} with~\eqref{babyjacobi} for the choice of $\mu$, are completely delocalized (for any choice of $\lambda\sim 1$). This can be proved with the very same methods as in~\cite{LS}. To understand the transition from partial localization to delocalization, further efforts are required.

The phenomenology described above is reminiscent of the one for heavy-tailed Wigner matrices. For instance, consider real symmetric Wigner matrices whose entries' distribution function decays as a power law, i.e.,
\begin{align} \label{alpha-stable}
\p (|h_{ij}| > x) = L(x) x^{-\alpha}\,,\qquad (1\le i,j\le N)\,,
\end{align}
for some slowly varying function $L(x)$. It was proved by Soshnikov~\cite{So} that the linear statistics of the largest eigenvalues are Poissonian for $\alpha < 2$; in particular, the largest eigenvalue has a Fr\'{e}chet limit distribution. Later, Auffinger, Ben Arous and P\'ech\'e~\cite{ABP} showed that the same conclusions hold for $2 \leq \alpha < 4$ as well. Recently, it was proved by Bordenave and Guionnet~\cite{BG} that the eigenvectors of models satisfying~\eqref{alpha-stable} are weakly delocalized for $1 < \alpha < 2$. For $0<\alpha<1$, it is conjectured~\cite{Cizeau} that there is a sharp ``metal-insulator'' transition. In~\cite{BG} it is proved that the eigenvectors of sufficiently large eigenvalues are weakly localized for $0<\alpha<2/3$.

To clarify the terminology ``partial localization'' we remark that it is quite different from the usual notion of localization for random Schr\"{o}dinger operators. The telltale signature of localization for random Schr\"{o}dinger operators is an exponential decay of off-diagonal Green function entries, which implies the absence of diffusion, spectral localization etc. For the Anderson model in dimensions $d\ge 3$ such an exponential decay was first obtained by Fr\"ohlich and Spencer~\cite{FS} using a multiscale analysis. Later, a similar bound was presented by Aizenman and Molchanov \cite{AM} using fractional moments. Due to the mean-field nature of the Wigner matrix $W$, there is no notion of distance for the deformed model~\eqref{interpolation}. Instead, our localization result states that most of the mass of the eigenvectors is concentrated on a few sites, whose locations are independent and uniformly distributed. This result agrees with the predictions of formal perturbation theory.

Yet, there are some similarities with the Anderson model in $d\ge 3$: In the Anderson model localization occurs where the density of states is (exponentially) small~\cite{FS}; this is known to happen close to the spectral edges or for large disorder. Further it is strongly believed that the Anderson model admits extended states, i.e., the generalized eigenvectors in the bulk are expected to be delocalized. Moreover, it was proven by Minami~\cite{Minami} that the local eigenvalue statistics of the Anderson model can be described by a Poisson point process in the strong localization regime. It is also conjectured that the local eigenvalue statistics in the bulk are given by the GOE statistics, respectively GUE statistics if time-reversal symmetry is broken. We remark that there are some partial results on bulk universality for the deformed model~\eqref{interpolation}; see~\cite{S1,OV}.

We also mention that the localization result we prove in this paper also differs from that for random band matrices, where all the eigenvectors are localized, even in the bulk. We refer to \cite{Sc,EK,EKYY3} for more details on the localization/delocalization for random band matrices.

Next, we outline the proofs of our main results. It was first shown by Pastur~\cite{P} that the empirical eigenvalue distribution of the deformed model~\eqref{interpolation} converges to a deterministic distribution in the limit $N\to\infty$ under some weak assumption on $\lambda V$. However, this limiting eigenvalue distribution, referred to as the {\it deformed semicircle law} in the following, is in general different from Wigner's semicircle law and depends on the limiting distribution of $\lambda V$. The deformed semicircle law can be defined in terms of a functional equation for the Stieltjes transforms of the limiting eigenvalue distributions of~$\lambda V$ and~$W$~\cite{P}. Restricting the discussion to the special case when the entries of $V$ follow the centered Jacobi distribution in~\eqref{babyjacobi} with $\b>1$, we showed in~\cite{LS} that the deformed semicircle law, henceforth denoted by $\mu_{fc}$, is 
supported on a single interval and shows either of the following behavior close to the upper edge: 
\begin{align}\label{the cases}
 \mu_{fc}(E)\sim\begin{cases} \sqrt{\kappa_E}\,,\quad&\textrm{ for }\quad\lambda<\lambda_+\,,\\
(\kappa_E)^{\b}\,,&\textrm{ for }\quad \lambda>\lambda_+\,,
\end{cases}
\end{align}
for $E\in\supp\mu_{fc}$, $E\ge 0$, where $\kappa_E$ denotes the distance from $E$ to the upper endpoint of the support of~$\mu_{fc}$; see Lemma~\ref{general case - large lambda} below. In case the square root behavior prevails, we are going to show that the largest eigenvalue of $H$ satisfies~\eqref{possible2}, whereas in case we have a ``convex decay'' with $\b>1$,~\eqref{possible1} is satisfied.

 In a first step, we derive a {\it local law} for the empirical eigenvalue density:  Under some moment conditions the convergence of the empirical eigenvalue distribution to the semicircle law also holds on very small scales. Denoting by $G(z)=(H-z)^{-1}$, $z\in\C^+$, the Green function or resolvent of $H$, convergence of the empirical eigenvalue distribution on scale $\eta$ around an energy $E\in\R$ is equivalent to the convergence of the averaged Green function $m(z)=N^{-1}\Tr G(z)$, $z=E+\ii\eta$. In a series of papers~\cite{ESY1,ESY2,ESY3} Erd\H{o}s, Schlein and Yau showed that the semicircle law for Wigner matrices also holds down to the optimal scale $1/N$, up to logarithmic corrections. In~\cite{EYY2} a ``fluctuation average lemma'' was introduced that yielded optimal bounds on the convergence of $m(z)$ for Wigner matrices in the bulk~\cite{EYY2} and up to the edge~\cite{EYY} on scales $\eta\gg N^{-1}$. Below this scale the eigenvalue density remains fluctuating even for large $N$. In~\cite{EYY} the 
Green function $G(z)$ and its average $m(z)$ have been used to prove edge universality for generalized Wigner matrices. In~\cite{LS} we derived a {\it local deformed semicircle law} for the deformed ensemble~\eqref{interpolation} under the assumption that $\mu_{fc}$ has a square root behavior at the endpoints. In the present paper, we derive a local law at the {\it extreme edge} in case $\mu_{fc}$ shows a convex decay at the edge. We propose, however, a slightly different path than the one taken in~\cite{LS}: We condition on the random variables~$(v_i)$ and show that~$m(z)$ converges, for ``typical'' realizations of~$(v_i)$, on scale~$\sim N^{-1/2}$; see Proposition~\ref{prop:step 2_4}. In particular, we show that the typical eigenvalue spacing at the extreme edge is of order $N^{-1/(\b+1)}\gg N^{-1/2}$; as is suggested by the convex decay 
in~\eqref{the cases}. Similar to the Wigner case, see e.g.,~\cite{EYY}, this is accomplished by deriving a self-consistent equation for~$m(z)$. However, the analysis of the self-consistent equation is quite different from the Wigner case, due to the absence of the usual stability bound; see~\cite{LS}.

 In a second step, we can use the self-averaging property of the Wigner matrix $W$, to show that the {\it imaginary part} of $m$ can be controlled on scales much smaller than $N^{-1/2}$: the technical input here is the ``fluctuation averaging lemma''~\cite{EYY,EKY,EKYY4}. Our proof relies on the basic strategy of \cite{EKYY4}. However, in our setup the diagonal entries of $G$ are not uniformly bounded, which requires several changes to previous arguments. To complete the proof of our first main result, Theorem~\ref{thm:local}, we note that the imaginary part of $m$ can be written as
\begin{align}
 \im m(E+\ii\eta)=\frac{1}{N}\sum_{\alpha=1}^N\frac{\eta}{(\mu_\alpha-E)^2+\eta^2}\,,\quad\qquad (E\in\R\,,\eta>0)\,,
\end{align}
where $(\mu_\alpha)$ are the eigenvalues of $H$. Thus, having control on the left side for $\eta\ll N^{-1/2}$ allows tracking the individual eigenvalues at the extreme edge, where their typical spacing much bigger than~$N^{-1/2}$.  

Finally, we point out the main steps in the proof of the partial localization of eigenvectors; see Theorem~\ref{thm:local} for precise results. It is well-known that information on the averaged Green function $m(z)$ can be translated via the Helffer-Sj\"{o}strand formula to information on the density of states; see, e.g.,~\cite{ERSY}. Since the typical eigenvalue spacing at the edge is, for the case at hand, much larger than~$N^{-1/2}$, the Helffer-Sj\"{o}strand formula also allows to translate information on the diagonal Green function entries $(G_{ii}(z))$ into information on the eigenvectors at the edge. Relying on estimates on the Green function, we can then prove ``partial localization'' of the eigenvectors at the edge.

The paper is organized as follows: In Section~\ref{Definition and Results}, we introduce the precise definition of the model and state the main results of the paper. In Section \ref{prelim}, we collect basic notations and identities for the resolvent of $H$. In Section~\ref{mfc and hat mfc}, we prove the first main result of the paper, Theorem \ref{thm:main}, using estimates on the Stieltjes transform of the deformed semicircle measure. (See also Corollary \ref{cor:main}.) In Sections \ref{location} and \ref{sec:Zlemma}, we prove important lemmas on the location of the extreme eigenvalues, including the local law, which have crucial roles in the proof of Theorem~\ref{thm:main}. In Section~\ref{sec:local}, we prove the second main result of the paper, Theorem \ref{thm:local}, on the partial localization of the eigenvectors at the edge. Proofs of some technical lemmas are collected in the Appendices A, B and C.

{\it Acknowledgements:}
We thank Horng-Tzer Yau for numerous helpful discussions and remarks. We are also grateful to Paul Bourgade, L\'{a}szl\'o Erd\H{o}s and Antti Knowles for discussions and comments. Ji Oon Lee is partially supported by the Basic Science Research Program of the National Research Foundation of Korea, Grant 2011-0013474. The stay of Kevin Schnelli at IAS is supported by The Fund For Math.

\section{Definition and Results} \label{Definition and Results}

In this section, we define our model and state our main results.

\subsection{Deformed semicircle law}
For a (probability) measure, $\omega$, on $\R$, we define its Stieltjes transform by
\begin{align}
 m_{\omega}(z)\deq\int_{\R}\frac{\dd\omega(x)}{x-z}\,,\quad\quad\quad (z\in\C^+)\,.
\end{align}
Note that $m_{\omega}(z)$ is an analytic function in the upper half plane, satisfying $\im m_{\omega}(z)\ge 0$, $z\in \C^+$.

As first shown in~\cite{P}, the Stieltjes transform of the limiting spectral distribution of the interpolating model~\eqref{interpolation} satisfies the equation
\begin{align}\label{the functional equation}
m_{fc}(z)=\int_\R\frac{\dd\mu(x)}{\lambda v-z-m_{fc}(z)}\,,\quad\quad \im m_{fc}(z)\ge 0\,,\qquad\quad (z\in\C^+)\,,
\end{align}
where $\mu$ is the distribution of the i.i.d.\ random variables $(v_i)$. Equation~\eqref{the functional equation} is often called the Pastur relation. It is shown in~\cite{P,B} that~\eqref{the functional equation} has a unique solution. Moreover, it is easy to check that $\limsup_{\eta\searrow 0} \im m_{fc}(E+\ii\eta)<\infty$, thus $m_{fc}(z)$ determines an absolutely continuous probability measure on $\R$, whose density, $\mu_{fc}$, is given by
\begin{align}\label{stieltjes inversion}
 \mu_{fc}(E)=\frac{1}{\pi}\lim_{\eta\searrow 0} \im m_{fc}(E+\ii\eta)\,,\quad\quad (E\in\R)\,.
\end{align}
The measure $\mu_{fc}$ has been studied in details in~\cite{B}; for example, it was shown that $\mu_{fc}$ is an analytic function inside its support.

\begin{rem}
 Setting $\lambda=0$, ~\eqref{the functional equation} reduces to 
\begin{align}
 m_{fc}(z)=-\frac{1}{z+m_{fc}(z)}\,,\quad\quad \im m_{fc}(z)\ge 0\,,\quad\quad (z\in\C^+)\,,
\end{align}
and one immediately checks that in this case $\mu_{fc}$ is given, as expected, by the standard semicircular measure, $\mu_{sc}$, which is characterized by the density $ \mu_{sc}(E)=\frac{1}{2\pi}\sqrt{(4-E^2)_+}$.
\end{rem}

\begin{rem}
The measure $\mu_{fc}$ is often called the {\it additive free convolution} of the semicircular law and the measure~$\mu$ (up to the scaling by $\lambda$). More generally, the additive free convolution of two (probability) measures $\omega_1$ and $\omega_2$, usually denoted by $\omega_1\boxplus\omega_2$, is defined as the distribution of the sum of two freely independent non-commutative random variables, having distributions $\omega_1$, $\omega_2$ respectively; we refer to~\cite{VDN,NS,HP,AGZ}. Similarly to~\eqref{the functional equation}, the free convolution measure $\omega_1\boxplus\omega_2$ can be described in terms of a set of functional equations for the Stieltjes transforms; see~\cite{PV,CG, BeB07}. For a discussion of regularity properties of $\omega_1\boxplus\omega_2$ we refer to~\cite{BBG}.
 
Free probability theory turned out to be a natural setting for studying global laws for such ensembles; see, e.g.,~\cite{VDN,AGZ}. For more recent treatments, including local laws, we refer to~\cite{K,CDFF,BBCF}.
\end{rem}

\subsection{Definition of the model}
\begin{defn} \label{assumption wigner}
Let $W$ be an $N\times N$ random matrix, whose entries, $(w_{ij})$, are independent, up to the symmetry constraint $w_{ij}=\overline{w_{ji}}$, centered, real (complex) random variables with variance $N^{-1}$ and with subexponential decay, i.e.,
\begin{align} \label{eq.C0}
\p \left (\sqrt{N} |w_{ij}|>x \right) \leq C_0 \e{-x^{1/\theta}},
\end{align}
for some positive constants $C_0$ and $\theta>1$. In particular, if $(w_{ij})$ are complex random variables, 
\begin{align} \label{complex wigner}
\E w_{ij}=0\,,\qquad \E|w_{ij}|^2 = \frac{1}{N}\,, \qquad \E w_{ij}^2 = 0\,, \qquad \E|w_{ij}|^p \leq C \frac{(\theta p)^{\theta p}}{N^{p/2}}\,,\quad (p \geq 3)\,;
\end{align}
if $(w_{ij})$ are real random variables,
\begin{align}
 \E w_{ij}=0\,,\qquad \E w_{ij}^2 = \frac{1+\delta_{ij}}{N}\,, \qquad \E|w_{ij}|^p \leq C \frac{(\theta p)^{\theta p}}{N^{p/2}}\,,\quad (p \geq 3)\,.
\end{align}

\end{defn}

Let $V$ be an $N\times N$ diagonal random matrix, whose entries $(v_i)$ are real, centered, i.i.d.\ random variables, independent of $W=(w_{ij})$, with law $\mu$. More assumptions on $\mu$ will be stated below. Without loss of generality, we assume that the entries of $V$ are ordered,
\begin{align}\label{the ordering}
v_1 \ge v_2 \ge \ldots \ge v_N.
\end{align}
For $\lambda \in \R^+$, we consider the random matrix
\begin{align} \label{thematrix}
H = (h_{ij}) \deq \lambda V + W\,.
\end{align}
We choose for simplicity $\mu$ as a Jacobi measure, i.e., $\mu$ is described in terms of its density
\begin{align} \label{jacobi measure}
\mu(v)= Z^{-1} (1+v)^{\a}(1-v)^{\b} d(v)\lone_{[-1, 1]}(v)\,,
\end{align}
where $\a,\b>-1$, $d\in C^1 ([-1, 1])$ such that $d(v)>0$, $v \in [-1, 1]$, and $Z$ is an appropriately chosen normalization constant. We assume, for simplicity of the arguments, that~$\mu$ is centered, but this condition can easily be relaxed. We remark that the measure~$\mu$ has support $[-1,1]$, but we observe that varying $\lambda$ is equivalent to changing the support of~$\mu$. Since~$\mu$ is absolutely continuous, we may assume that~\eqref{the ordering} holds with strict inequalities.

\subsection{Edge behavior of $\mu_{fc}$}
Properties of $\mu_{fc}$ with the special choice~\eqref{jacobi measure} for $\mu$ and with $\lambda\sim 1$, have been studied in~\cite{LS}; see also~\cite{B,S2}. For example, the support of~$\mu_{fc}$ consists of a single interval. For $E\in\R$, we denote by $\kappa_E$ the distance to the endpoints support of~$\mu_{fc}$, i.e.,
\begin{align}\label{kappa}
 \kappa_E\deq \min\{|E-L_-|,|E-L_+|\}\,,\qquad \supp\,\mu_{fc}=[L_-,L_+]\,,\quad\qquad (E\in\R)\,.
\end{align}
In the following we will often abbreviate $\kappa\equiv \kappa_E$.

In the present paper, we are mainly interested in the limiting behavior of the largest, respectively smallest, eigenvalues of the interpolating matrix~\eqref{thematrix}, for $\lambda\sim 1$. 
For concreteness, we focus on the upper edge and comment on the lower edge in Remark~\ref{remark about lower edge}. The following lemma is taken from~\cite{LS}; see also~\cite{B,S2,ON} for statement~(1).

\begin{lem} \label{general case - large lambda}
Let $\mu$ be a centered Jacobi measure defined in \eqref{jacobi measure} with $\b>1$. Define
\begin{align}\label{definition of lambda+}
\lambda_+ \deq \left( \int_{-1}^1 \frac{\mu(v) \dd v}{(1-v)^2} \right)^{1/2}, \qquad \tau_+ \deq \int_{-1}^1 \frac{\mu(v) \dd v}{1-v}\,.
\end{align}
Then, there exist $L_- < 0 <L_+$ such that the support of~$\mu_{fc}$ is $[L_-,L_+]$. Moreover,
\begin{itemize}
\item [(1)] if $\lambda<\lambda_+$, then for $0 \leq \kappa \leq L_+$,
\begin{align}
C^{-1}\sqrt{\kappa} \leq \mu_{fc}(L_+-\kappa) \leq C \sqrt\kappa\,,
\end{align}
for some $C \geq 1$;

\item[(2)] if $\lambda>\lambda_+$, then $L_+ = \lambda+(\tau_+/\lambda)$ and, for $0 \leq \kappa \leq L_+$,
\begin{align} \label{exponent beta}
C^{-1}{\kappa}^{\b} \leq \mu_{fc}(L_+-\kappa) \leq C \kappa^{\b},
\end{align}
for some $C \geq 1$. Moreover, $L_+$ satisfies $L_++m_{fc}(L_+)=\lambda$.
\end{itemize}

\end{lem}

\begin{rem}
Since
\begin{equation}
\tau_+ =\int_{-1}^1 \frac{\mu (v) \dd v}{1-v} > \int_{-1}^1 (1+v) \mu (v) \dd v = 1,
\end{equation}
we find that $L_+ > \lambda + (1/\lambda) \geq 2$. Similarly, we also have that $L_- \leq -2$.
\end{rem}

\begin{rem} \label{remark about lower edge}
For $\a>1$, the analogue statements to Lemma~\ref{general case - large lambda} hold for the lower endpoint $L_-$ of the support of~$\mu_{fc}$, with $\lambda_+$ and $\tau_+$ replaced by
\begin{align}
\lambda_- \deq \left( \int_{-1}^1 \frac{\mu(v) \dd v}{(1+v)^2} \right)^{1/2}\,,\qquad \tau_- \deq \int_{-1}^1 \frac{\mu(v) \dd v}{1+v} \,.
\end{align}

\end{rem}

\begin{rem}
When $-1 < \a, \b < 1$, there exist, for any $\lambda\in\R^+$, $L_-<0<L_+$, such that $\supp\,\mu_{fc}=[L_-,L_+]$. Moreover, for any $\lambda\in\R^+$, there exists $C\ge1$ such that
\begin{align}\label{squarerrot behavior}
 C^{-1}\sqrt{\kappa_E}\le\mu_{fc}(E)\le C\sqrt{\kappa_E}\,,\quad\quad E\in[L_-,L_+]\,.
\end{align}

If $\a < 1 < \b$ or $\b < 1 < \a$, the analogous statement to \eqref{squarerrot behavior} holds only at the lower edge or at the upper edge, respectively. These results can be proved using the methods of~\cite{S2}; see~\cite{LS} for more details.

\end{rem}

In~\cite{LS}, spectral properties of the interpolating matrix~\eqref{thematrix} have been analyzed in detail under the assumption that \eqref{squarerrot behavior} holds, i.e., it was assumed that either $\lambda$ is sufficiently small or $\a,\b\le 1$.

\subsection{Main results}
Denote by $(\mu_i)$ the ordered eigenvalues of the matrix $H=\lambda V+W$,
 
$$
\mu_1 \geq \mu_2 \geq \ldots \geq \mu_N\,.
$$
In the following, we fix some $n_0\in\N$, independent of $N$, and consider the largest eigenvalues $(\mu_i)_{i=1}^{n_0}$ of~$H$. All our results also apply mutatis mutandis to the smallest eigenvalues $(\mu_i)_{i=N-n_0}^N$ of $H$ as can readily be checked.

\subsubsection{Eigenvalue statistics}
The first main result of the paper shows that the locations of the extreme eigenvalues are determined by the order statistics of the diagonal elements $(v_i)$. Recall that we denote by $\mu$ the distribution of the (unordered) centered random variables~$(v_i)$.

\begin{thm} \label{thm:main}
Let $W$ be a real symmetric or complex Hermitian Wigner matrix, satisfying the assumptions in Definition~\ref{assumption wigner}. Assume that the distribution~$\mu$ is given by~\eqref{jacobi measure} with $b>1$ and fix some $\lambda>\lambda_+$; see~\eqref{definition of lambda+}.  Let $n_0>10$ be a fixed constant independent of $N$, denote by $\mu_i$ the $i$-th largest eigenvalue of $H=\lambda V+W$ and let $1\le k<n_0$.  Then the joint distribution function of the $k$ largest rescaled eigenvalues,
\begin{align}\label{converging expression 1}
\p \left( N^{1/(\b+1)} (L_+ - \mu_1)\le s_1,\, N^{1/(\b+1)} (L_+ - \mu_2)\le s_2,\, \ldots,\, N^{1/(\b+1)} (L_+ - \mu_k)\le s_k \right)\,,
\end{align}
converges to the joint distribution function of the $k$ largest rescaled order statistics of $(v_i)$,
\begin{align}\label{converging expression}
\p \left(C_\lambda N^{1/(\b+1)} (1-v_1)\le s_1,\,C_\lambda N^{1/(\b+1)} (1-v_2)\le s_2,\, \ldots,\,C_\lambda N^{1/(\b+1)}  (1-v_k)\le s_k \right)\,,
\end{align}
as $N \to \infty$, where $C_\lambda= \frac{\lambda^2 - \lambda_+^2}{\lambda} $. In particular, the cumulative distribution function of the rescaled largest eigenvalue $N^{1/(\b+1)} (L_+ - \mu_1)$ converges to the cumulative distribution function of the Weibull distribution,
\begin{align}\label{Weibull}
G_{\b+1}(s)\deq 1 - \exp \left( -\frac{C_{\mu} s^{\b+1}}{(\b+1)} \right)\,,
\end{align} 
where
$$
C_{\mu} \deq \left( \frac{\lambda}{\lambda^2 - \lambda_+^2} \right)^{\b+1} \lim_{v \to 1} \frac{\mu(v)}{(1-v)^{\b}}\,.
$$
\end{thm}
In Section~\ref{location} we obtain estimates on the speed of convergence of~\eqref{converging expression 1}; see Corollary~\ref{cor:main}.

\begin{rem} \label{rem:Gaussian}
 For $\lambda>\lambda_+$, the typical size of the fluctuations of the largest eigenvalues is of order $ N^{-1/(\b+1)}$ (with $\b>1$) as we can see from Theorem \ref{thm:main}. For $\lambda < \lambda_+$, on the other hand, the fluctuations for the largest eigenvalue become, in the limit $N\to\infty$, Gaussian with standard deviation of order $N^{-1/2}$. (See Appendix C for more detail.)

\end{rem}

\begin{rem}
Theorem \ref{thm:main} shows that the extreme eigenvalues of $H$ become, for $\lambda>\lambda_+$, uncorrelated in the limit $N\to\infty$. In fact, extending the methods presented in this paper (by choosing $n_0\lesssim N^{1/(\b+1)}$), one can show that the point process defined by the (unordered) rescaled extreme eigenvalues of $H$ converges in distribution to an inhomogeneous Poisson point process on $\R^+$ with intensity function determined by $\lambda$ and $\mu$.
\end{rem}

\subsubsection{Eigenvectors behavior}
Our second main result asserts that the eigenvectors associated with the largest eigenvalues are ``partially localized'' for $\lambda>\lambda_+$. We denote by $(u_k(j))_{j=1}^N$ the components of the eigenvector~$u_k$ associated to the eigenvalue $\mu_k$. All eigenvectors are normalized as $\sum_{j=1}^N|u_k(j)|^2=\|u_k\|_2^2=1$.

\begin{thm} \label{thm:local}
Let $W$ be a real symmetric or complex Hermitian Wigner matrix satisfying the assumptions in Definition~\ref{assumption wigner}. Assume that the distribution $\mu$ is given by~\eqref{jacobi measure} with $b>1$ and fix some $\lambda>\lambda_+$; see~\eqref{definition of lambda+}. Let $n_0>10$ be a fixed constant independent of $N$. Then there exist constants $\delta, \delta',\sigma>0$, depending only on $\b$, $\lambda$ and~$\mu$, such that
\begin{align} \label{mass_k}
\p\left(\left||u_k (k)|^2 - \frac{\lambda^2 - \lambda_+^2}{\lambda^2}\right|> N^{-\delta} \right)\le N^{-\sigma}\,,\quad\qquad (1\le k\le n_0-1)\,,
\end{align}
and
\begin{align} \label{mass_j}
\p\left(|u_k (j)|^2 > \frac{N^{\delta'}}{N^{\phantom{\delta'}}}\frac{1}{\lambda^2 |v_k - v_j|^2}\right)\le N^{-\sigma}\,,\quad\qquad (1\le j\le N\,,1\le k\le n_0-1\,, j\neq k)\,.
\end{align}
\end{thm}
In Section~\ref{sec:local} we obtain explicit expressions for the constants $\delta, \delta',\sigma>0$.

\begin{rem} \label{rem:bulk}
In the preceding paper~\cite{LS}, we proved that all eigenvectors are completely delocalized when $\lambda < \lambda_+$. This shows the existence of a  sharp transition from the partial localization to the complete delocalization regime. We say that an eigenvalue $\mu_i$ is in the bulk of the spectrum of $H$ if $i\in[\epsilon N,(1-\epsilon)N]$, for any (small) $\epsilon>0$ and sufficiently large $N$. Following the proof in~\cite{LS}, can prove that the eigenvectors associated to eigenvalues in the bulk are completely delocalized if $\lambda > \lambda_+$.

Assuming that $\mu$ is given by~\eqref{jacobi measure} with $\b\le 1$ (and $\a\le 1$), we showed in~\cite{LS} that all eigenvectors of $H$ are completely delocalized up to the edge. 
\end{rem}

\begin{rem}
Theorems \ref{thm:main} and \ref{thm:local} remain valid for deterministic potentials $V$, provided the entires $(v_i)$ satisfy some suitable assumptions; see Definition~\ref{v assumptions} in Section~\ref{mfc and hat mfc} for details.
\end{rem}

\begin{rem} \label{rem:perturbation}
From \eqref{mass_k} we find 
$$
\sum_{j:j \neq k}^N |u_k (j)|^2 = \frac{\lambda_+^2}{\lambda^2} + o(1)\,,\qquad (1\le k \le n_0-1)\,,
$$
which is in accordance with the fact that \eqref{mass_j} holds and that, typically,
$$
\frac{1}{N} \sum_{j:j \neq k}^N \frac{1}{\lambda^2 |v_k - v_j|^2} = \frac{\lambda_+^2}{\lambda^2} + o(1)\,,\qquad  (1\le k \le n_0-1)\,,
$$
where we used~\eqref{converging expression}. 
\end{rem}
In the remaining sections, we prove Theorems~\ref{thm:main} and~\ref{thm:local}. We state our proofs for complex Hermitian matrices. The real symmetric case can be dealt with in the same way. 
\section{Preliminaries} \label{prelim}

In this section, we collect basic notations and identities.

\subsection{Notations}
For high probability estimates we use two parameters $\xi \equiv \xi_N$ and $\varphi \equiv \varphi_N$: We let
\begin{align} \label{eq.xi}
\xi = 10 \log\log N\,, \qquad \varphi=(\log N)^C,
\end{align}
for some fixed constant $C \geq 1$.

\begin{defn}
We say an event $\Omega$ has $(\xi,\nu)$-high probability, if
$$
\p (\Omega^c) \leq \e{-\nu(\log N)^{\xi}}\,,
$$
for $N$ sufficiently large. Similarly, for a given event $\Omega_0$ we say an event $\Omega$ holds with $(\xi,\nu)$-high probability on $\Omega_0$, if
$$
\p(\Omega_0\cap\Omega^c) \leq \e{-\nu(\log N)^{\xi}}\,,
$$
for $N$ sufficiently large.
\end{defn}

For brevity, we occasionally say an event holds with high probability, when we mean with $(\xi,\nu)$-high probability. We do not keep track of the explicit value of $\nu$ in the following, allowing $\nu$ to decrease from line to line such that $\nu>0$. From our proof it becomes apparent that such reductions occur only finitely many times.

We define the resolvent, or \emph{Green function}, $G(z)$, and the averaged Green function, $m(z)$, of $H$ by
\begin{align}
G(z) = (G_{ij}(z)) \deq \frac{1}{H-z} = \frac{1}{\lambda V+W-z}\,,\qquad m(z) \deq \frac{1}{N} \Tr G(z)\,, \qquad\quad( z \in \C^{+})\,.
\end{align} 
Frequently, we abbreviate $G \equiv G(z)$, $m\equiv m(z)$, etc. We refer to $z$ as spectral parameter and often write $z=E+\ii\eta$, $E\in\R$, $\eta>0$.

We will use double brackets to denote the index set, i.e., for $n_1, n_2 \in \R$,
$$
\llbracket n_1, n_2 \rrbracket \deq [n_1, n_2] \cap \Z\,.
$$

We use the symbols $\caO(\,\cdot\,)$ and $o(\,\cdot\,)$ for the standard big-O and little-o notation. The notations $\caO$, $o$, $\ll$, $\gg$, refer to the limit $N \to \infty$ unless otherwise stated, where the notation $a \ll b$ means $a=o(b)$. We use $c$ and $C$ to denote positive constants that do not depend on $N$. Their value may change from line to line. Finally, we write $a \sim b$, if there is $C \geq 1$ such that $C^{-1}|b| \leq |a| \leq C |b|$, and, occasionally, we write for $N$-dependent quantities $a_N \lesssim b_N$, if there exist constants $C, c >0$ such that $|a_N| \leq C(\varphi_N)^{c\xi}|b_N|$.

\subsection{Minors}
Let $\T\subset \llbracket1, N \rrbracket$. Then we define $H^{(\T)}$ as the $(N-|\T|)\times(N-|\T|)$ minor of $H$ obtained by removing all columns and rows of $H$ indexed by $i\in\T$. Note that we do not change the names of the indices of $H$ when defining $H^{(\T)}$.
More specifically, we define an operation $\pi_i$, $i\in \llbracket1, N \rrbracket$, on the probability space by
\begin{align}
 (\pi_i(H))_{kl}\deq\lone(k\not=i)\lone(l\not=i)h_{kl}\,.
\end{align}
Then, for $\T\subset \llbracket1, N \rrbracket$, we set $\pi_{\T}\deq\prod_{i\in\T}\pi_i$ and define
\begin{align}
 H^{(\T)}\deq((\pi_{\T}(H)_{ij})_{i,j\not\in\T}\,.
\end{align}
The Green functions $G^{(\T)}$, are defined in an obvious way using $H^{(\T)}$.  Moreover, we use the shorthand notation
\begin{align}
 \sum_{i}^{(\T)}\deq\sum_{\substack{i=1\\i\not\in\T}}^N\,\,,\qquad \qquad\sum_{i\not=j}^{(\T)}\deq\sum_{\substack{i=1,\, j=1\\ i\not=j\,,\,i,j\not\in\T}}^N\,,
\end{align}
abbreviate $(i)=(\{i\})$, $(\T i)=(\T\cup\{i\})$ and use the convention $(\T\backslash i)= (\T\backslash \{i\})$, if $i\in\T$, $(\T\backslash i)=(\T)$, else.  In Green function entries $(G_{ij}^{(\T)})$ we refer to $\{i,j\}$ as lower indices and to $\T$ as upper indices.

Finally, we set
\begin{align}
 m^{(\T)}\deq\frac{1}{N}\sum_{i}^{(\T)}G_{ii}^{(\T)}\,.
\end{align}
Here, we use the normalization $N^{-1}$, instead $(N-|\T|)^{-1}$, since it is more convenient for our computations.

\subsection{Resolvent identities}
The next lemma collects the main identities between resolvent matrix elements of $H$ and $H^{(\T)}$.

\begin{lem}
Let $H=H^*$ be an $N\times N$ matrix. Consider the Green function $G(z)\equiv G\deq(H-z)^{-1}$, $ z\in\C^+$. Then, for $i,j,k,l\in\llbracket 1,N\rrbracket$, the following identities hold:
\begin{itemize}
 \item[-] {\it Schur complement/Feshbach formula:}\label{feshbach} For any $i$,
\begin{align}\label{schur} 
 G_{ii}=\frac{1}{h_{ii}-z-\sum_{k,l}^{(i)}{h_{ik} G_{kl}^{(i)}}h_{li}}\,.
\end{align}
\item[-] For $i\not=j$,
\begin{align}\label{twosided}
 G_{ij}=-G_{ii}G_{jj}^{(i)}\left(h_{ij}-\sum_{k,l}^{(ij)}h_{ik}G_{kl}^{(ij)}h_{lj}\right)\,.
\end{align}
\item[-] For $i\not=j$,
\begin{align}\label{onesided}
 G_{ij}=-G_{ii}\sum_{k}^{(i)}h_{ik}G_{kj}^{(i)}=-G_{jj}\sum_{k}^{(j)} G_{ik}^{(j)} h_{kj}\,.
\end{align}

\item[-] For $i,j\not=k$,
\begin{align}\label{basic resolvent}
 G_{ij}=G_{ij}^{(k)}+\frac{G_{ik}G_{kj}}{G_{kk}}\,.
\end{align}

\item[-]{\it Ward identity:} For any $i$,
\begin{align}\label{ward}
 \sum_{j=1}^N|G_{ij}|^2=\frac{1}{\eta}\im G_{ii}\,,
\end{align}
where $\eta=\im z$.
\end{itemize}
\end{lem}
For a proof we refer to, e.g.,~\cite{EKYY1}.

\begin{lem}\label{cauchy interlacing} There is a constant $C$ such that, for any $z \in \C^+$, $i\in\llbracket 1,N\rrbracket$, we have
\begin{align}\label{m minus m^a}
 |m(z)-m^{(i)}(z)|\le \frac{C}{N\eta}\,.
\end{align}
\end{lem}
 The lemma follows from Cauchy's interlacing property of eigenvalues of $H$ and its minor $H^{(i)}$. For a detailed proof we refer to~\cite{E}. For $\T\subset\llbracket 1,N\rrbracket$, with, say, $|\T|\le 10$, we obtain $|m-m^{(\T)}|\le \frac{C}{N\eta}$.
\subsection{Large deviation estimates}
We collect here some useful large deviation estimates for random variables with slowly decaying moments.
\begin{lem}\label{lemma.LDE}
 Let $(a_i)$ and $(b_i)$ be centered and independent complex random variables with variance~$\sigma^2$ and having subexponential decay
\begin{equation}
 \mathbb{P}\left(|a_i|\ge x\sigma\right)\le C_0\,\mathrm{e}^{-x^{1/\theta}}\,,\qquad\mathbb{P}\left(|b_i|\ge x\sigma\right)\le C_0\,\mathrm{e}^{-x^{1/\theta}}\,,
\end{equation}
for some positive constants $C_0$ and $\theta>1$. For $i,j\in\llbracket 1,N\rrbracket$, let $A_i\in\C$ and $B_{ij}\in\C$. Then there exists a constant $c_0$, depending only on $\theta$ and $C_0$, such that for $1<\xi\le 10\log\log N$ and $\varphi_N=(\log N)^{c_0}$ the following estimates hold.
\begin{itemize}
 \item[(1)]

\begin{align}
\mathbb{P}\left( \left|\sum_{i=1}^N A_ia_i   \right|\ge (\varphi_N)^{\xi}\sigma\left(\sum_{i=1}^N|A_i|^2\right)^{1/2}\right)&\le \mathrm{e}^{-(\log N)^{\xi}}\,,\label{LDE1} \\[1mm]
\mathbb{P}\left(\left|\sum_{i=1}^N\overline{a}_i B_{ii}a_i-\sum_{i=1}^N \sigma^2 B_{ii}\right|\ge (\varphi_N)^{\xi}\sigma^2\left(\sum_{i=1}^N|B_{ii}|^2\right)^{1/2}\right)&\le \mathrm{e}^{-(\log N)^{\xi}}\,,\label{LDE2}\\[1mm]
\mathbb{P}\left(\left|\sum_{i\not=j}^N\overline{a}_i B_{ij}a_j\right|\ge(\varphi_N)^{2\xi}\sigma^2\left(\sum_{i\not=j}|B_{ij}|^2  \right)^{1/2}\right)&\le\mathrm{e}^{-(\log N)^{\xi}}\label{LDE3}\,,
\end{align}
for $N$ sufficiently large;
\item[(2)]
\begin{align}\label{LDE4}
 \mathbb{P}\left(\left|\sum_{i,j} \overline{a}_i B_{ij}b_j\right|\ge(\varphi_N)^{2\xi}\sigma^2\left(\sum_{i,j}|B_{ij}|^2\right)^{1/2}\right)\le\mathrm{e}^{-(\log N)^{\xi}}\,,
\end{align}
for $N$ sufficiently large.
\end{itemize}

\end{lem}
For a proof we refer to~\cite{EYY1}. We now choose $C$ in~\eqref{eq.xi} such that $C\ge c_0$.

Finally, we point out the difference between the random variables $(w_{ij})$ and $(v_i)$: From~\eqref{eq.C0}, we obtain
\begin{align}\label{bound on wij}
 |w_{ij}|\le \frac{(\varphi_N)^{\xi}}{\sqrt{N}}\,,
\end{align}
with $(\xi,\nu)$-high probability, whereas $v_i\in[-1,1]$, almost surely.

\section{Proof of Theorem \ref{thm:main}} \label{mfc and hat mfc}

In this section, we outline the proof of Theorem~\ref{thm:main}. We first fix the diagonal random entries $(v_i)$ and consider~$\widehat \mu_{fc}$, the deformed semicircle measure with fixed $(v_i)$. The main tools we use in the proof are Lemma \ref{mfc estimate}, where we obtain a linear approximation of~$m_{fc}$, and Lemma \ref{hat bound}, which estimates the difference between $m_{fc}$ and~$\widehat m_{fc}$, the latter being the Stieltjes transform of~$\widehat \mu_{fc}$. Using Proposition~\ref{prop:mu_k} that estimates the eigenvalue locations in terms of $\widehat m_{fc}$, we prove Theorem \ref{thm:main}.

\subsection{Definition of $\Omega_V$}
In this subsection we define an event $\Omega_V$, on which the random variables $(v_i)$ exhibit ``typical'' behavior. For this purpose we need some more notation: Denote by $\fb$ the constant
\begin{align} \label{fb}
\fb \deq \frac{1}{2} - \frac{1}{\b +1} = \frac{\b -1}{2(\b + 1)} = \frac{\b}{\b +1} - \frac{1}{2}\,,
\end{align}
which only depends on $\b$. Fix some small $\epsilon > 0$ satisfying
\begin{align}\label{epsilon condition}
\epsilon < \left( 10 + \frac{\b +1}{\b -1} \right) \fb\,,
\end{align}
and define the domain, $\caD_{\epsilon}$, of the spectral parameter $z$ by 
\begin{align} \label{domain}
\caD_{\epsilon} \deq \{ z = E + \ii \eta\in\C^+\, :\, -3 - \lambda \leq E \leq 3 + \lambda, \; N^{-1/2 - \epsilon} \leq \eta \leq N^{-1/(\b+1) + \epsilon} \}\,.
\end{align}
Using spectral perturbation theory, we find that the following a priori bound
\begin{align} \label{mu_k a priori}
|\mu_k| \leq \| H \| \leq \| W \| + \lambda \| V \| \leq 2+\lambda+(\varphi_N)^{c\xi}N^{-2/3}\,,\quad\qquad (k\in\llbracket 1,N\rrbracket)\,,
\end{align}
holds with high probability; see, e.g., Theorem~2.1. in~\cite{EYY}.
 
Further, we define $N$-dependent constants $\kappa_0$ and $\eta_0$ by
\begin{align}\label{definition of kappa0}
\kappa_0 \deq N^{-1/(\b+1)}, \qquad\quad \eta_0 \deq \frac{N^{-\epsilon}}{\sqrt N}\,.
\end{align}
In the following, typical choices for $z \equiv L_+ - \kappa + \ii \eta$ will be such that $\kappa$ and $\eta$ satisfy $\kappa \lesssim \kappa_0$ and $\eta \geq \eta_0$. 

We are now prepared to give a definition of the ``good'' event $\Omega_V$:

\begin{defn} \label{v assumptions}
Let $n_0 > 10$ be a fixed positive integer independent of $N$. We define $\Omega_V$ to be the event on which the following conditions hold for any $k \in \llbracket 1, n_0 -1 \rrbracket$:

\begin{enumerate}
\item The $k$-th largest random variable $v_k$ satisfies, for all $j\in\llbracket 1,N\rrbracket$ with $j \neq k$,
\begin{align}\label{eq4.3}
N^{-\epsilon} \kappa_0 < |v_j - v_k| < (\log N) \kappa_0\,.
\end{align}
In addition, for $k=1$, we have
\begin{align}\label{eq4.4}
N^{-\epsilon} \kappa_0 < |1 - v_1| < (\log N) \kappa_0\,.
\end{align}

\item There exists a constant $\mathfrak{c} <1$ independent of $N$ such that, for any $z \in \caD_{\epsilon}$ satisfying
\begin{align} \label{assumption near v_k}
\min_{i \in \llbracket 1, N \rrbracket} |\re (z + m_{fc}(z)) - \lambda v_i| = |\re (z + m_{fc}(z)) - \lambda v_k|\,,
\end{align}
we have
\begin{align} \label{assumption_CLT_2}
\frac{1}{N} \sum_{i}^{(k)} \frac{1}{|\lambda v_i - z - m_{fc}(z)|^2} <\mathfrak{c} < 1\,.
\end{align}
We remark that, together with~\eqref{eq4.3} and~\eqref{eq4.4}, \eqref{assumption near v_k} implies
\begin{align}
|\re (z + m_{fc}(z)) - \lambda v_i| > \frac{N^{-\epsilon} \kappa_0}{2}\,,
\end{align}
for all $i \neq k$.

\item There exists a constant $C>0$ such that, for any $z \in \caD_{\epsilon}$, we have
\begin{align} \label{assumption_CLT_1}
\left| \frac{1}{N} \sum_{i=1}^N \frac{1}{\lambda v_i - z - m_{fc}(z)} - \int \frac{\dd \mu(v)}{\lambda v - z - m_{fc}(z)} \right| \leq \frac{C N^{3\epsilon /2}}{\sqrt N}\,.
\end{align}
\end{enumerate}
\end{defn}
\noindent In Appendix~A we show that 
\begin{align} \label{Omega_v}
\p (\Omega_V) \geq 1 - C (\log N)^{1 + 2\b} N^{-\epsilon}, 
\end{align}
thus $(\Omega_V)^c$ is indeed a rare event.

\subsection{Definition of $\widehat m_{fc}$}
Let~$\widehat \mu$ be the empirical measure defined by
\begin{align} \label{def hat mu}
\widehat \mu \deq \frac{1}{N} \sum_{i=1}^N \delta_{\lambda v_i}\,.
\end{align}
We define a random measure $\widehat\mu_{fc}$ by setting $\widehat \mu_{fc} \deq \widehat \mu \boxplus \mu_{sc}$, i.e.,~$\widehat\mu_{fc}$ is the additive free convolution of the empirical measure $\widehat \mu$ and the semicircular measure~$\mu_{sc}$. As in the case of $m_{fc}$, the Stieltjes transform~$\widehat m_{fc}$ of the measure~$\widehat \mu_{fc}$ is a solution to the equation
\begin{align} \label{eq:hat mfc}
\widehat m_{fc}(z) = \frac{1}{N} \sum_{i=1}^N \frac{1}{\lambda v_i - z - \widehat m_{fc}(z)}\,,\qquad \im\widehat m_{fc}(z)\ge 0\,,\qquad\qquad (z\in\C^+)\,,
\end{align}
and we obtain~$\widehat\mu_{fc}$ trough the Stieltjes inversion formula from $\widehat m_{fc}(z)$, c.f.,~\eqref{stieltjes inversion}.

Recall that we assume that $v_1 > v_2 > \ldots > v_N$. Assuming that $\Omega_V$ holds, i.e., $(v_i)$ are fixed and satisfy the conditions in Definition~\ref{v assumptions}, we are going to show that $m_{fc}(z)$ is a good approximation of $\widehat m_{fc}(z)$ for~$z$ in some subset of~$\caD_\epsilon$.

\subsection{Properties of $m_{fc}$ and $\widehat m_{fc}$}
Recall the definitions of $m_{fc}$ and $\widehat m_{fc}$. Let
\begin{align}\label{definition of R2 without hat}
R_2 (z) \deq \int \frac{\dd \mu(v)}{|\lambda v - z - m_{fc}(z)|^2}, \qquad \widehat R_2 (z) \deq \frac{1}{N} \sum_{i=1}^N \frac{1}{|\lambda v_i - z - \widehat m_{fc}(z)|^2}\,,\quad\qquad (z\in\C^+) \,.
\end{align}
Since
$$
\im m_{fc}(z) = \int \frac{\im z + \im m_{fc}(z)}{|\lambda v - z - m_{fc}(z)|^2}\, \dd \mu(v)\,,
$$
we have that
$$
R_2 (z) = \frac{\im m_{fc}(z)}{\im z + \im m_{fc}(z)} < 1\,,\qquad\qquad (z\in\C^+)\,.
$$
Similarly, we also find that $\widehat R_2 (z) < 1$.

The following lemma shows that $m_{fc}$ is approximately a linear function near the spectral edge.

\begin{lem} \label{mfc estimate}
Let $z = L_+ - \kappa + \ii \eta \in \caD_{\epsilon}$. Then,
\begin{align}
z + m_{fc}(z) = \lambda - \frac{\lambda^2}{\lambda^2 - \lambda_+^2} (L_+ - z) + \caO \left( (\log N) (\kappa + \eta)^{\min \{ \b, 2 \} } \right)\,.
\end{align}
Similarly, if $z, z' \in \caD_{\epsilon}$, then
\begin{align}
m_{fc}(z) - m_{fc}(z') = \frac{\lambda_+^2}{\lambda^2 - \lambda_+^2} (z-z') + \caO \left( (\log N)^2 (N^{-1/(\b+1)})^{\min \{ \b-1, 1 \} } |z-z'|  \right)\,. 
\end{align}
\end{lem}

\begin{proof}
We only prove the first part of the lemma; the second part is proved analogously. Since
$L_+ + m_{fc}(L_+) = \lambda$, see Lemma~\ref{general case - large lambda},
we can write
\begin{align} \begin{split}
m_{fc}(z) - m_{fc}(L_+) &= \int \frac{\dd \mu (v)}{\lambda v - z - m_{fc}(z)} - \int \frac{\dd \mu (v)}{\lambda v - L_+ - m_{fc}(L_+)} \\
&= \int \frac{m_{fc}(z) - m_{fc}(L_+) + (z-L_+)}{(\lambda v - z - m_{fc}(z))(\lambda v - \lambda)} \dd \mu (v)\,.
\end{split} \end{align}
Setting
\begin{align} \label{definition T}
T(z) \deq \int \frac{\dd \mu (v)}{(\lambda v - z - m_{fc}(z))(\lambda v - \lambda)}\,,
\end{align}
we find
\begin{align}\label{definition of T1}
|T(z)| \leq \left( \int \frac{\dd \mu (v)}{|\lambda v - z - m_{fc}(z)|^2} \right)^{1/2} \left( \int \frac{\dd \mu (v)}{|\lambda v - \lambda|^2} \right)^{1/2} \leq \sqrt{R_2(z)} \frac{\lambda_+}{\lambda} < \frac{\lambda_+}{\lambda} < 1\,.
\end{align}
Hence, for $z \in \caD_{\epsilon}$, we have
\begin{align}
m_{fc}(z) - m_{fc}(L_+) = \frac{T(z)}{1-T(z)} (z - L_+)\,,
\end{align}
which shows that
\begin{align} \label{Lipschitz estimate}
z + m_{fc}(z) = \lambda - \frac{1}{1-T(z)} (L_+ - z)\,.
\end{align}
We thus obtain from~\eqref{definition of T1} and~\eqref{Lipschitz estimate} that
$$
|z + m_{fc}(z) - \lambda| \leq \frac{\lambda}{\lambda - \lambda_+} |L_+ -z|\,.
$$

We now estimate the difference  $T(z)-{\lambda^2_+}/{\lambda^2}\,$: Let $\tau \deq z + m_{fc}(z)$. We have
\begin{align} \label{eq:estimate T}
T(z) - \frac{\lambda_+^2}{\lambda^2} = \int \frac{\dd \mu (v)}{(\lambda v - \tau)(\lambda v - \lambda)} - \int \frac{\dd \mu (v)}{(\lambda v - \lambda)^2} = (\tau - \lambda) \int \frac{\dd \mu (v)}{(\lambda v - \tau)(\lambda v - \lambda)^2}.
\end{align}
In order to find an upper bound on the integral on the very right side, we consider the following cases:
\begin{itemize}
\item[(1)] When $\b \geq 2$, we have
\begin{align}
\left| \int \frac{\dd \mu (v)}{(\lambda v - \tau)(\lambda v - \lambda)^2} \right| \leq C \int_{-1}^1 \frac{\dd v}{|\lambda v - \tau|} \leq C \log N\,.
\end{align}

\item[(2)] When $\b < 2$, define a set $B \subset [-1, 1]$ by 
$$
B \deq \{ v \in [-1, 1] : \lambda v < -\lambda + 2 \, \re \tau \}\,,
$$
and $B^c\equiv [-1, 1] \backslash B$. Estimating the integral in \eqref{eq:estimate T} on $B$ we find
\begin{align}
\left| \int_{B} \frac{\dd \mu (v)}{(\lambda v - \tau)(\lambda v - \lambda)^2} \right| \leq C \int_{B} \frac{\dd \mu (v)}{|\lambda v - \lambda|^3} \leq C |\lambda - \tau|^{\b-2}\,,
\end{align}
where we have used that, for $v \in B$,
$$
|\lambda v - \tau| > |\re \tau - \lambda v| > \frac{1}{2} (\lambda - \lambda v)\,.
$$
On the set $B^c$, we have
\begin{align}\label{estimate T1}
\left| \int_{B^c} \frac{\dd \mu (v)}{(\lambda v - \tau)(\lambda v - \lambda)} \right| \leq C \int_{B^c} \frac{|\lambda - \lambda v|^{\b-1}}{|\lambda v - \tau|} \dd v \leq C |\lambda - \tau|^{\b -1} \log N\,, 
\end{align}
where we have used that, for $v \in B^c$,
$$
|\lambda - \lambda v| \leq 2 (\lambda - \re \tau) \leq 2 |\lambda - \tau|\,.
$$
We also have
\begin{align}\label{estimate T2}
\left| \int_{B^c} \frac{\dd \mu (v)}{(\lambda v - \lambda)^2} \right| \leq C \int_{B^c} |\lambda v - \lambda|^{\b -2} \dd v \leq C |\lambda - \tau|^{\b -1}\,.
\end{align}
Thus, we obtain from~\eqref{eq:estimate T}, ~\eqref{estimate T1} and~\eqref{estimate T2} that
\begin{align}
\left| \int \frac{\dd \mu (v)}{(\lambda v - \tau)(\lambda v - \lambda)^2} \right| \leq C |\lambda - \tau|^{\b-2} \log N\,.
\end{align}
\end{itemize}
Since $T(z)$ is continuous and $\caD_\epsilon$ is compact, we can choose the constants uniform in $z$. We thus have proved that
\begin{align}
T(z) = \frac{\lambda_+^2}{\lambda^2} + \caO\left((\log N)|L_+ - z|^{\min \{ \b-1, 1 \} } \right)\,,
\end{align}
which, combined with \eqref{Lipschitz estimate}, proves the desired lemma.
\end{proof}

\begin{rem} \label{rem:kappa}
Choosing in Lemma \ref{mfc estimate} $z=z_k$, where $z_k \deq L_+ - \kappa_k + \ii \eta \in \caD_{\epsilon}$ with 
$$
\kappa_k = \frac{\lambda^2 - \lambda_+^2}{\lambda} (1-v_k)\,,
$$
 we obtain
\begin{align}
z_k + m_{fc}(z_k) = \lambda v_k + \frac{\lambda^2}{\lambda^2 - \lambda_+^2} \eta + \caO\left((\log N) N^{-\min \{ \b, 2 \} / (\b+1) + 2 \epsilon} \right)\,.
\end{align}
\end{rem}

To estimate the difference $|\widehat m_{fc} - m_{fc}|$, we consider the following subset of $\caD_{\epsilon}$.
\begin{defn}
Let $A\deq\llbracket n_0,N\rrbracket$. We define the domain $\caD_{\epsilon}'$ of the spectral parameter $z$ as
\begin{align}\label{a index assumption}
 \caD_{\epsilon}'=\left\{ z\in\caD_{\epsilon}\,:\, |\lambda v_a-z-{m}_{fc}(z)| > \frac{1}{2} N^{-1 / (\b + 1)-\epsilon},\,\forall a\in A\right\}\,.
\end{align}
\end{defn}
Eventually, we are going to show that $\mu_k + \ii \eta_0 \in \caD_{\epsilon}'$, $k\in\llbracket 1,n_0-1\rrbracket$, with high probability on $\Omega_V$; see Remark~\ref{rem:mu_k}. We now prove an a priori bound on the difference $|\widehat m_{fc} - m_{fc}|$ on $\caD_{\epsilon}'$.

\begin{lem} \label{hat bound}
For any $z \in \caD_{\epsilon}'$, we have on $\Omega_V$ that
\begin{align} \label{eq:hat bound}
|\widehat m_{fc} (z)-m_{fc}(z) | \leq \frac{N^{2\epsilon}}{\sqrt N}\,.
\end{align}
\end{lem}

\begin{proof}
Assume that $\Omega_V$ holds. For given $z \in \caD_{\epsilon}'$, choose $k \in \llbracket 1, n_0 -1 \rrbracket$ satisfying \eqref{assumption near v_k}, i.e., among $(\lambda v_i)$,  $\lambda v_k$ is closest to $\re (z + m_{fc}(z))$. Suppose that \eqref{eq:hat bound} does not hold. Using the definitions of $m_{fc}$ and $\widehat{m}_{fc}$, we obtain the following self-consistent equation for $(\widehat m_{fc} - m_{fc})$:
\begin{align} \begin{split} \label{mfc difference}
&\widehat m_{fc} - m_{fc} = \frac{1}{N} \sum_{i=1}^N \left( \frac{1}{\lambda v_i - z - \widehat m_{fc}} - m_{fc} \right) \\
&= \frac{1}{N} \sum_{i=1}^N \left( \frac{1}{\lambda v_i - z - \widehat m_{fc}} - \frac{1}{\lambda v_i - z - m_{fc}} \right) + \left( \frac{1}{N} \sum_{i=1}^N \frac{1}{\lambda v_i - z - m_{fc}} - \int \frac{\dd \mu(v)}{\lambda v - z - m_{fc}} \right) \\
&= \frac{1}{N} \sum_{i=1}^N \frac{\widehat m_{fc} - m_{fc}}{(\lambda v_i - z - \widehat m_{fc})(\lambda v_i - z - m_{fc})} + \left( \frac{1}{N} \sum_{i=1}^N \frac{1}{\lambda v_i - z - m_{fc}} - \int \frac{\dd \mu(v)}{\lambda v - z - m_{fc}} \right).
\end{split} \end{align}
From the assumption \eqref{assumption_CLT_1}, we find that the second term in the right hand side of \eqref{mfc difference} is bounded by $N^{-1/2 + 3\epsilon /2}$.

Next, we estimate the first term in the right hand side of \eqref{mfc difference}. For $i=k$, we have
$$
|\lambda v_k - z - \widehat m_{fc}| + |\lambda v_k - z - m_{fc}| \geq |\widehat m_{fc} (z) - m_{fc} (z)| > \frac{N^{2\epsilon}}{\sqrt N}\,,
$$
which shows that either
$$
|\lambda v_k - z - \widehat m_{fc}| \geq \frac{N^{2\epsilon}}{2\sqrt N} \qquad \text{or} \qquad |\lambda v_k - z - m_{fc}| \geq \frac{N^{2\epsilon}}{2\sqrt N}\,.
$$
In either case, by considering the imaginary part, we find
$$
\frac{1}{N} \left| \frac{1}{(\lambda v_k - z - \widehat m_{fc})(\lambda v_k - z - m_{fc})} \right| \leq \frac{1}{N} \frac{2\sqrt N}{N^{2\epsilon}} \frac{1}{\eta} \leq C N^{-\epsilon}\,,\qquad\quad (z\in\caD_{\epsilon}')\,.
$$
For the other terms, we use
\begin{align}
\frac{1}{N} \left| \sum_{i}^{(k)} \frac{1}{(\lambda v_i - z - \widehat m_{fc})(\lambda v_i - z - m_{fc})} \right| \leq \frac{1}{2N} \sum_{i}^{(k)} \left( \frac{1}{|\lambda v_i - z - \widehat m_{fc}|^2} + \frac{1}{|\lambda v_i - z - m_{fc}|^2} \right)\,.
\end{align}
From \eqref{eq:hat mfc}, we have that
\begin{align}\label{R2 hat less than 1}
\frac{1}{N} \sum_{i=1}^N \frac{1}{|\lambda v_i - z - \widehat m_{fc}|^2} = \frac{\im \widehat m_{fc}}{\eta + \im \widehat m_{fc}} < 1\,.
\end{align}
We also assume in the assumption \eqref{assumption_CLT_2} that
\begin{align}
\frac{1}{N} \sum_{i}^{(k)} \frac{1}{|\lambda v_i - z - m_{fc}|^2} < \mathfrak{c} < 1\,,
\end{align}
for some constant $c$. Thus, we get
\begin{align}
|\widehat m_{fc} (z) - m_{fc} (z)| < \frac{1+\mathfrak{c}}{2} |\widehat m_{fc} (z) - m_{fc} (z)| + N^{-1/2 + 3\epsilon /2}\,,\quad\qquad (z\in\caD_{\epsilon}')\,,
\end{align}
which implies that
$$
|\widehat m_{fc} (z) - m_{fc} (z)| < C N^{-1/2 + 3\epsilon /2}\,,\qquad\quad (z\in\caD_{\epsilon}')\,.
$$
Since this contradicts the assumption that \eqref{eq:hat bound} does not hold, it proves the desired lemma.
\end{proof}

\subsection{Proof of Theorem \ref{thm:main}}
The main result of this subsection is Proposition~\ref{prop:main}, which will imply Theorem~\ref{thm:main}. The key ingredient of the proof of Proposition~\ref{prop:main} is an implicit equation for the largest eigenvalues $(\mu_k)$ of $H$. This equation, Equation~\eqref{implicit equation} in Proposition~\ref{prop:mu_k} below, involves the Stieltjes transform $\widehat m_{fc}$ and the random variables $(v_k)$. Using the information on $\widehat m_{fc}$ gathered in the previous subsections, we can solve Equation~\eqref{implicit equation} approximately for $(\mu_k)$. The proof of Proposition~\ref{prop:mu_k} is postponed to Section~\ref{location}.

\begin{prop} \label{prop:mu_k}
Let $n_0>10$ be a fixed integer independent of $N$. Let $\mu_k$ be the $k$-th largest eigenvalue of $H$, $k\in\llbracket 1, n_0-1\rrbracket$. Suppose that the assumptions in Theorem~\ref{thm:main} hold. Then, the following holds with $(\xi-2,\nu)$-high probability on $\Omega_V$:
\begin{align}\label{implicit equation}
\mu_k + \re \widehat m_{fc} (\mu_k + \ii \eta_0) = \lambda v_k + \caO (N^{-1/2 + 3\epsilon})\,,
\end{align}
where $\eta_0$ is defined in~\eqref{definition of kappa0}.

\end{prop}

\begin{rem} \label{rem:mu_k}
Since $|\lambda v_i - \lambda v_k| \geq N^{-\epsilon} \kappa_0 \gg N^{-1/2 + 3\epsilon}$, for all $i \neq k$, on $\Omega_V$, we obtain from Proposition~\ref{prop:mu_k} that 
$$
|\mu_k + \ii \eta_0 + \re \widehat m_{fc}(\mu_k + \ii \eta_0) - \lambda v_i| \geq |\lambda v_i - \lambda v_k| - |\mu_k + \ii \eta_0 + \re \widehat m_{fc}(\mu_k + \ii \eta_0) - \lambda v_k| \geq \frac{N^{-\epsilon} \kappa_0}{2}\,,
$$
on $\Omega_V$. Hence, we find that $\mu_k + \ii \eta_0 \in \caD_{\epsilon}'$, $k\in\llbracket 1,n_0-1\rrbracket$, with high probability on~$\Omega_V$.
\end{rem}

Combining the tools developed in the previous subsection, we now prove the main result on the eigenvalue locations. 

\begin{prop} \label{prop:main}
Let $n_0>10$ be a fixed integer independent of $N$. Let $\mu_k$ be the $k$-th largest eigenvalue of $H=\lambda V+W $, where $k \in \llbracket 1, n_0 -1 \rrbracket$. Then, there exist constants $C$ and $\nu>0$ such that we have
\begin{align}
\left| \mu_k - \left( L_+ - \frac{\lambda^2 - \lambda_+^2}{\lambda} (1-v_k) \right) \right| \leq C \frac{1}{N^{1/(\b+1)}} \left(\frac{N^{3\epsilon}}{ N^{\fb}}  + \frac{(\log N)^2}{N^{1/(\b+1)}}  \right)\,,
\end{align}
with $(\xi-2,\nu)$-high probability on $\Omega_V$.
\end{prop}

\begin{proof}[Proof of Theorem \ref{thm:main} and Proposition~\ref{prop:main}]
It suffices to prove Proposition \ref{prop:main}. Let $k\in\llbracket 1,n_0-1\rrbracket$. From Lemma~\ref{hat bound} and Proposition~\ref{prop:mu_k}, we find that, with high probability on~$\Omega_V$,
\begin{align}
\mu_k + \re m_{fc} (\mu_k + \ii \eta_0) = \lambda v_k + \caO (N^{-1/2 + 3\epsilon})\,.
\end{align}
In Lemma~\ref{mfc estimate}, we showed that
\begin{align}
\mu_k + \ii \eta_0 + m_{fc} (\mu_k + \ii \eta_0) = \lambda - \frac{\lambda^2}{\lambda^2 - \lambda_+^2} (L_+ - \mu_k) + \ii C \eta_0 + \caO \left( \kappa_0^{\min \{ \b, 2 \} } (\log N)^2 \right)\,.
\end{align}
Thus, we obtain
\begin{align}
\mu_k + \re m_{fc} (\mu_k + \ii \eta_0) = \lambda - \frac{\lambda^2}{\lambda^2 - \lambda_+^2} (L_+ - \mu_k) + \caO \left( \kappa_0^{\min \{ \b, 2 \} } (\log N)^2 \right)\,.
\end{align}
Therefore, we have with high probability on $\Omega_V$ that
\begin{align}
\mu_k = L_+ - \frac{\lambda^2 - \lambda_+^2}{\lambda} (1-v_k) + \caO \left( \kappa_0^{\min \{ \b, 2 \} } (\log N)^2 \right) + \caO (N^{-1/2 + 3\epsilon})\,,
\end{align}
completing the proof of Proposition~\ref{prop:main}.
\end{proof}

Recalling that $\p (\Omega_V) \geq 1 - C (\log N)^{1 + 2\b} N^{-\epsilon}$, we obtain from Proposition \ref{prop:main} the following corollary. 

\begin{cor} \label{cor:main}
Let $n_0$ be a fixed constant independent of $N$. Let $\mu_k$ be the $k$-th largest eigenvalue of $H=\lambda V+W$, where $1 \leq k < n_0$. Then, there exists a constant $C_1 > 0$ such that for $s\in\R^+$ we have
\begin{align} \begin{split}
&\p \left( N^{1/(\b+1)} \frac{\lambda^2 - \lambda_+^2}{\lambda} (1-v_k) \leq s - C_1 \left( \frac{N^{3\epsilon}}{N^{\fb}} + \frac{ (\log N)^2}{N^{1/(\b+1)}} \right) \right) - C_1 \frac{(\log N)^{1+2\b}}{ N^{\epsilon}} \\
&\quad \leq \p \left( N^{1/(\b+1)} (L_+ - \mu_k) \leq s \right) \\
&\quad \leq \p \left( N^{1/(\b+1)} \frac{\lambda^2 - \lambda_+^2}{\lambda} (1-v_k) \leq s + C_1 \left( \frac{N^{3\epsilon}}{N^{\fb}} + \frac{ (\log N)^2}{N^{1/(\b+1)}} \right) \right) + C_1 \frac{(\log N)^{1+2\b}}{ N^{\epsilon}}\,,
\end{split} \end{align}
for $N$ sufficiently large.
\end{cor}

\begin{rem}
The constants in Proposition~\ref{prop:main} and Corollary~\ref{cor:main} depend only on $\lambda$, the distribution~$\mu$ and the constants~$C_0$ and~$\theta$ in~\eqref{eq.C0}, but are otherwise independent of the detailed structure of the Wigner matrix $W$. 
\end{rem}

\section{Estimates on the Location of the Eigenvalues} \label{location}
In this section, we prove Proposition \ref{prop:mu_k}. Recall the definition of $\eta_0$ in~\eqref{definition of kappa0}. For $k\in\llbracket 1,n_0-1\rrbracket$, let $\widehat E_k\in\R$ be a solution $E=\widehat E_k$ to the equation
\begin{align}\label{definition of hatzk}
 E + \re\widehat m_{fc} (E+\ii\eta_0 )  = \lambda v_k\,,\qquad\quad(E\in\R)\,,
\end{align}
and set $\widehat z_k \deq \widehat E_k+\ii\eta_0$. The existence of such $\widehat E_k$ is easy to see from Lemma \ref{mfc estimate} and Lemma \ref{hat bound}. If there are two or more solutions to~\eqref{definition of hatzk}, we choose $\widehat E_k$ to be the largest one among these solutions. The key observation used in the proof of Proposition~\ref{prop:mu_k} is that $\im m(z)$, the imaginary part of the averaged Green function, has a sharp peak if and only if the imaginary part of
\begin{align}\label{the expression}
 g_{k}(z)\deq  \frac{1}{\lambda v_k - z - \widehat m_{fc} (z)}\,,\qquad\quad( z\in\C^+)\,,
\end{align}
becomes sufficiently large for some $k\in\llbracket 1,n_0-1\rrbracket$. Since $z\mapsto\im g_k(z)$ has a sharp peak near $\widehat z_k$, we can then conclude that $z\mapsto\im m(z)$ also has a peak near $\widehat{z}_k$. From the spectral decomposition
$$
\im m(E + \ii \eta_0) = \frac{1}{N} \sum_{\alpha=1}^N \frac{\eta_0}{(\mu_{\alpha} - E)^2 + \eta_0^2}\,,
$$
we also observe that the positions of the peaks of $\im m(z)$ correspond to the locations of the eigenvalues. This will enable us to estimate the location of the $k$-th largest eigenvalue in terms of $v_k$, yielding a proof of Proposition~\ref{prop:mu_k}.

This section is organized as follows. In Subsection~\ref{Properties of widehat m_{fc} and m}, we establish a local law for $m(z)$ with $z\in\caD_{\epsilon}'$, i.e., for $z$ close to the upper edge; see Proposition~\ref{prop:step 2_4} below. In the Subsections~\ref{aux estimate 1} and~\ref{aux estimate 2} we establish further estimates that will be used in the proof of Proposition~\ref{prop:mu_k}. The estimates in Subsection~\ref{aux estimate 2} are rather straightforward, while the estimates of Subsection~\ref{aux estimate 2} rely on the ``fluctuation average lemma'' whose proof is postponed to Section~\ref{sec:Zlemma}. The proof of Proposition~\ref{prop:mu_k} is then completed in Subsection~\ref{Proof of Proposition 2_4}.

\subsection{Properties of $\widehat m_{fc}$ and $m$}\label{Properties of widehat m_{fc} and m}
In the proof of Proposition~\ref{prop:mu_k}, we will use the following local law as an a priori estimate. Recall the constant $\epsilon>0$ in~\eqref{epsilon condition} and the definition of the domain $\caD_{\epsilon}'$ in~\eqref{a index assumption}.
\begin{prop}{\emph{[Local law near the edge]}} \label{prop:step 2_4}
We have with $(\xi,\nu)$-high probability on $\Omega_V$ that
\begin{align}
|m(z)-\widehat m_{fc}(z)| \leq \frac{N^{2\epsilon}}{\sqrt N}\,,
\end{align}
for all $z \in \caD_{\epsilon}'$.
\end{prop}
The proof of Proposition~\ref{prop:step 2_4} is the content of the rest of this subsection.

Recall the definitions of $(\widehat z_k)$ in~\eqref{definition of hatzk}. We begin by deriving a basic property of $\widehat m_{fc}(z)$ near $(\widehat z_k)$. Recall the definition of $\eta_0$ in~\eqref{definition of kappa0}.
\begin{lem} \label{lem:step 1}
 For $z = E + \ii \eta_0 \in \caD_{\epsilon}'$, the following hold on $\Omega_V$:
\begin{enumerate}
\item[$(1)$]  if $|z- \widehat z_j| \geq N^{-1/2 + 3\epsilon}$ for all $j \in\llbracket 1,n_0-1\rrbracket$, then there exists a constant $C >1$ such that
$$
C^{-1} \eta_0 \leq \im \widehat m_{fc} (z) \leq C \eta_0\,;
$$
\item[$(2)$] if $z= \widehat z_k$  for some $k \in\llbracket 1,n_0-1\rrbracket$, then there exists a constant $C >1$ such that
$$
C^{-1} N^{-1/2} \leq \im \widehat m_{fc} (z) \leq C N^{-1/2}\,.
$$
\end{enumerate}
\end{lem}

\begin{proof}
Recall that
\begin{align}
\widehat R_2 (z) = \frac{\im \widehat m_{fc}(z)}{\eta_0 + \im \widehat m_{fc}(z)} = \frac{1}{N} \sum_{i=1}^N \frac{1}{|\lambda v_i - z - \widehat m_{fc}(z)|^2}<1\,,\qquad\quad( z\in\C^+)\,,
\end{align}
c.f.,~\eqref{definition of R2 without hat}. For given $z \in \caD_{\epsilon}'$ with $\im z=\eta_0$, choose $k \in \llbracket 1, n_0 -1 \rrbracket$ such that~\eqref{assumption near v_k} is satisfied. In the first case, where $|z- \widehat z_k| \gg N^{-1/2 + 2\epsilon}$, we find from Lemma \ref{mfc estimate} and Lemma \ref{hat bound} that 
\begin{align} \label{eq:step 1_1}
|\lambda v_k - \re(z + \widehat m_{fc}(z))| \gg N^{-1/2 + 2\epsilon}.
\end{align}
Since $z=E+\ii\eta_0$ satisfies \eqref{assumption near v_k}, we also find that 
\begin{align}
\widehat R_2^{(k)}(z) \deq \frac{1}{N} \sum_{i}^{(k)} \frac{1}{|\lambda v_i - z - \widehat m_{fc}(z)|^2} = \frac{1}{N} \sum_{i}^{(k)} \frac{1}{|\lambda v_i - z - m_{fc}(z)|^2} + o(1) < c < 1\,,
\end{align}
for some constant $c$. Thus,
\begin{align}
\widehat R_2(z) = \frac{1}{N} \frac{1}{|\lambda v_k - z - \widehat m_{fc}(z)|^2} + \frac{1}{N} \sum_{i}^{(k)} \frac{1}{|\lambda v_i - z - \widehat m_{fc}(z)|^2} < c' < 1\,,
\end{align}
for some constant $c'$. Recalling that
$$
\im \widehat m_{fc}(z) = \frac{\widehat R_2(z)}{1 - \widehat R_2(z)} \eta_0\,,
$$
statement $(1)$ of the lemma follows.

Next, we consider the second case: $z=\widehat z_k=\widehat E_k+\ii\eta_0$, for some $k\in\llbracket 1,n_0-1\rrbracket$. We have
\begin{align} \label{mfc quadratic}
\im \widehat m_{fc}(\widehat z_k) = \frac{1}{N} \sum_{i=1}^N \frac{\eta_0 + \im \widehat m_{fc}(\widehat z_k)}{|\lambda v_i - \widehat z_k - \widehat m_{fc}(\widehat z_k)|^2} = \frac{1}{N} \frac{1}{\eta_0 + \im \widehat m_{fc}(\widehat z_k)} + \frac{1}{N} \sum_{i}^{(k)} \frac{\eta_0 + \im \widehat m_{fc}(\widehat z_k)}{|\lambda v_i - \widehat z_k - \widehat m_{fc}(\widehat z_k)|^2}\,,
\end{align}
hence
$$
(1 - \widehat R_2^{(k)}(\widehat z_k )) (\im \widehat m_{fc}(\widehat z_k))^2 + (1 - 2\widehat R_2^{(k)}(\widehat z_k )) \eta_0 \,\im \widehat m_{fc}(\widehat z_k) = \frac{1}{N} + \widehat R_2^{(k)}(\widehat z_k) \eta_0^2\,.
$$
Solving the quadratic equation above for $\im\widehat m_{fc}(\widehat z_k)$, we find
$$
C^{-1} N^{-1/2} \leq \im \widehat m_{fc} (\widehat z_k) \leq C N^{-1/2}\,,
$$
completing the proof of the lemma.
\end{proof}

\begin{rem} \label{im mfc upper bound}
For any $z = E + \ii \eta_0 \in \caD_{\epsilon}'$, we have, similarly to \eqref{mfc quadratic}, that
\begin{align}
\im \widehat m_{fc}(z) \leq \frac{1}{N} \frac{1}{\eta_0 + \im \widehat m_{fc}(z)} + \frac{1}{N} \sum_{i}^{(k)} \frac{\eta_0 + \im \widehat m_{fc}(z)}{|\lambda v_i - z - \widehat m_{fc}(z)|^2}\,.
\end{align}
Solving this inequality for $\im \widehat m_{fc}(z)$, we find that $\im \widehat m_{fc} (z) \leq C N^{-1/2}$.
\end{rem}

Recall that by Schur's complement formula we have, for all $i\in\llbracket 1,N\rrbracket$,
$$
G_{ii} = \frac{1}{\lambda v_i + w_{ii} - z - \sum_{s, t}^{(i)} {h_{is} G_{st}^{(i)} h_{ti}}}\,;
$$
see~\eqref{schur}. Define $\E_i$ to be the partial expectation with respect to the $i$-th column/row of $W$ and set
\begin{align} \label{Z_i}
Z_i \deq (\lone-\E_i) \sum_{s, t}^{(i)} {h_{is} G_{st}^{(i)} h_{ti}} = \sum_{s}^{(i)} (|w_{is}|^2 - \frac{1}{N}) G_{ss}^{(i)} + \sum_{s \neq t}^{(i)} w_{is} G_{st}^{(i)} w_{ti}\,.
\end{align}
Using $Z_i$, we can rewrite $G_{ii}$ as
\begin{align}
G_{ii} = \frac{1}{\lambda v_i + w_{ii} - z - m^{(i)} - Z_i}\,.
\end{align}

The following lemma states an a priori bound on $\im m$, the imaginary part of $m=N^{-1}\Tr G$.

\begin{lem} \label{lem:step 2_1}
We have with $(\xi,\nu)$-high probability on $\Omega_V$ that, for all $z = E + \ii \eta_0 \in \caD_{\epsilon}'$, 
\begin{align}
\im m(z) \leq \frac{N^{2\epsilon}}{\sqrt N}\,.
\end{align}
\end{lem}

\begin{proof}
Fix $\eta = \eta_0$. For given $z=E+\ii\eta_0 \in \caD_{\epsilon}'$, choose $k \in \llbracket 1, n_0 -1 \rrbracket$ such that~\eqref{assumption near v_k} is satisfied. Suppose that $\im m(z) > N^{-1/2 + 5\epsilon /3}$. Reasoning as in the proof of Lemma \ref{hat bound}, we find the following equation for $(m -\widehat m_{fc})$:
\begin{align} \begin{split} \label{m difference}
m - \widehat m_{fc} &= \frac{1}{N} \sum_{i=1}^N \left( \frac{1}{\lambda v_i + w_{ii} - z - m^{(i)} - Z_i} - \frac{1}{\lambda v_i - z - \widehat m_{fc}} \right) \\
&= \frac{1}{N} \sum_{i=1}^N \frac{m^{(i)} - \widehat m_{fc} + Z_i - w_{ii}}{(\lambda v_i + w_{ii} - z - m^{(i)} - Z_i) (\lambda v_i - z - \widehat m_{fc})}\,.
\end{split} \end{align}
Abbreviate
\begin{align} \label{T_m}
T_m\equiv T_m(z) \deq \frac{1}{N} \sum_{i=1}^N \left| \frac{1}{(\lambda v_i + w_{ii} - z - m^{(i)} - Z_i) (\lambda v_i - z - \widehat m_{fc})} \right|\,.
\end{align}
We are going to show that $T_m<c<1$:
We define events $\Omega_Z$ and $\Omega_W$ by
\begin{align}
 \Omega_Z(z)\equiv\Omega_Z\deq \bigcap_{i=1}^N\left\{|Z_i|\le  (\varphi_N)^{\xi} \sqrt{ \frac{\im m^{(i)}}{N \eta} } \right\}\,, \qquad \Omega^W \deq \bigcap_{i=1}^N\left \{ |w_{ii}| \le \frac{(\varphi_N)^{\xi}}{\sqrt N} \right \}\,.
\end{align}
We notice that, by the large deviation estimates in Lemma \ref{lemma.LDE} and the subexponential decay of $|w_{ii}|$, $\Omega_Z$ and $\Omega_W$ both hold with high probability. Suppose now that $\Omega_Z$ and $\Omega_W$ hold. Then, we have
$$
|Z_i| \leq (\varphi_N)^{\xi} \sqrt{ \frac{\im m^{(i)}}{N \eta} } \leq (\varphi_N)^{-\xi} \im m^{(i)} + C \frac{(\varphi_N)^{2\xi}}{N \eta} \ll \im m\,,
$$
where we have used $|m - m^{(i)}| \leq C (N \eta)^{-1} \ll \im m$. We also have from $\eta \ll \im m$ that
$$
\im m^{(i)} + \eta + \im Z_i = (1 + o(1)) \im m\,.
$$
Thus,
\begin{align} \begin{split}
\frac{1}{N} \sum_{i=1}^N \frac{1}{|\lambda v_i + w_{ii} - z - m^{(i)} - Z_i|^2}& = \frac{1}{N} \sum_{i=1}^N \frac{\im G_{ii}}{\im m^{(i)} + \eta + \im Z_i} = \frac{1}{N} \sum_{i=1}^N \frac{ \im G_{ii}}{\im m}(1 + o(1)) \\
&= 1 + o(1)\,.
\end{split} \end{align}
We get from Lemma \ref{hat bound} that, on $\Omega_V$,
$$
\frac{1}{N} \sum_i^{(k)} \frac{1}{|\lambda v_i - z - \widehat m_{fc}|^2} = \frac{1}{N} \sum_i^{(k)} \frac{1 + o(1)}{|\lambda v_i - z - m_{fc}|^2} < c < 1\,,
$$
for some constant $c > 0$, and
$$
\frac{1}{N} \left| \frac{1}{(\lambda v_k + w_{kk} - z - m^{(k)} - Z_k) (\lambda v_k - z - \widehat m_{fc})} \right| \leq C \frac{1}{N} \frac{1}{N^{-1/2 + 5\epsilon /3} \eta} \leq N^{-2\epsilon /3}\,.
$$
Hence, we find that $T_m < c' < 1$ for some constant $c'$. Notice that the assumption $\im m > N^{-1/2 + 5\epsilon /3}$ also implies that 
$$
|m - \widehat m_{fc}| \geq |\im m - \im \widehat m_{fc}| > C N^{-1/2 + 5\epsilon /3}\,,
$$
as we can see from Remark \ref{im mfc upper bound}. Now, if we let
$$
M \deq\max_i |m^{(i)} - m + Z_i - w_{ii}|\,,
$$
then $M \ll |m - \widehat m_{fc}|$. Thus, taking absolute value on both sides of \eqref{m difference}, we get
$$
|m - \widehat m_{fc}| \leq T_m (|m - \widehat m_{fc}| + M) = \left( T_m + o(1) \right) |m - \widehat m_{fc}|\,,
$$
contradicting $T_m < c' < 1$.

We have thus shown that for fixed $z\in\caD_\epsilon'$,
$$
\im m(z) \le N^{-1/2 + 5\epsilon /3} \,,
$$
with high probability on $\Omega_V$.

In order to prove that the desired bound holds uniformly on $z$, we consider a lattice $\caL$ such that, for any $z$ satisfying the assumption of the lemma, there exists $z' = E' + \ii \eta_0 \in \caL$ with $|z - z'| \leq N^{-3}$. We have already seen that the uniform bound holds for all points in $\caL$. For a point $z \notin \caL$, we have $|m(z) - m(z')| \leq \eta_0^2 |z-z'| \leq N^{-1}$, for $z' \in \caL$ with $|z - z'| \leq N^{-3}$. This proves the desired lemma.
\end{proof}
As a corollary of Lemma~\ref{lem:step 2_1} we obtain:
\begin{cor} \label{cor:step 2_1}
We have with $(\xi,\nu)$-high probability on $\Omega_V$ that, for all $z = E + \ii \eta_0 \in \caD_{\epsilon}'$,
\begin{align}
\max_i |Z_i(z)| \leq \frac{N^{2\epsilon}}{\sqrt N}\,, \qquad \max_i |Z_i^{(k)}(z)| \leq \frac{N^{2\epsilon}}{\sqrt N}\,,\qquad\quad (k\in\llbracket 1,N\rrbracket)\,.
\end{align}
\end{cor}

Next, we prove an estimate for the difference $\Lambda(z)\deq |m(z)-\widehat m_{fc}(z)|$. We first show the bound on $\Lambda(z)$ in Proposition~\ref{prop:step 2_4} holds for large $\eta$; see Lemma~\ref{lem:step 2_2} below. Then, using the self-consistent equation~\eqref{m difference}, we show that if, for some $z\in\caD_{\epsilon}'$, $\Lambda(z)\le N^{-1/2+3\epsilon}$, then we must have $\Lambda(z)\le N^{-1/2+2\epsilon}$ with high probability; see Lemma~\ref{lem:step 2_3}. Thus, using the Lipschitz continuity of the Green function $G$ and of the Stieltjes transform $\widehat m_{fc}$, we can conclude that if $\Lambda(z)\le N^{-1/2+2\epsilon}$, we also have $\Lambda(z')\le N^{-1/2+2\epsilon}$, with high probability, for $z'$ in a sufficiently small neighborhood of $z$. Repeated use this argument yields a proof of Proposition~\ref{prop:step 2_4} at the end of this subsection.

Recall that we have set $\kappa_0=N^{-1/(\b+1)}$; see~\eqref{definition of kappa0}.
\begin{lem} \label{lem:step 2_2}
We have with high probability on $\Omega_V$ that, for all $z = E + \ii \eta \in \caD_{\epsilon}'$ with $N^{-1/2 + \epsilon} \leq \eta \leq N^{\epsilon} \kappa_0$,
\begin{align} \label{eq:boot}
| m(z)-\widehat m_{fc}(z) | \leq \frac{N^{2\epsilon}}{\sqrt N}\,.
\end{align}
\end{lem}

\begin{proof}
 The proof closely follows the proof of Lemma \ref{lem:step 2_1}. Fix $z\in\caD_{\epsilon}'$. Suppose that $|m(z)-\widehat m_{fc}(z) | > N^{-1/2 + 5\epsilon /3}$. Consider the self-consistent equation \eqref{m difference} and define $T_m$ as in \eqref{T_m}.

Since $\im m(E+\ii\eta)\ge C\eta$, for $z\in\caD_{\epsilon}'$, with high probability on $\Omega_V$, we obtain that
$$
\im m^{(i)} + \eta + \im Z_i = (1 + o(1)) \im m\,,
$$
with high probability on $\Omega_V$, as in the proof of Lemma \ref{lem:step 2_1}. This implies that $T_m < c < 1$. If we let
$$
M = \max_i |m^{(i)} - m + Z_i - w_{ii}|\,,
$$
it then follows that $M \ll |m - \widehat m_{fc}|$ with high probability on $\Omega_V$. Taking absolute values on both sides of~\eqref{m difference}, we obtain a contradiction to the assumption $|m(z)-\widehat m_{fc}(z)| > N^{-1/2 + 5\epsilon /3}$. In order to attain a uniform bound, we again use the lattice argument as in the proof of Lemma \ref{lem:step 2_1}. This completes the proof of the lemma.
\end{proof}

\begin{lem} \label{lem:step 2_3}
Let $z \in \caD_{\epsilon}'$. If $|m(z)-\widehat m_{fc}(z)| \leq N^{-1/2 + 3\epsilon}$, then we have with $(\xi,\nu)$-high probability on $\Omega_V$ that $|m(z)-\widehat m_{fc}(z) | \leq N^{-1/2 + 2\epsilon}$.
\end{lem}

\begin{proof}
Since the proof closely follows the proof of Lemma \ref{lem:step 2_1}, we only check the main steps here. Fix $z\in\caD_{\epsilon}'$ and choose $k \in \llbracket 1, n_0 -1 \rrbracket$ such that \eqref{assumption near v_k} is satisfied. Assume that $N^{-1/2 + 5\epsilon /3} < |m(z)-\widehat m_{fc}(z)| \leq N^{-1/2 + 3\epsilon}$. We consider the self-consistent equation \eqref{m difference} and define $T_m$ as in \eqref{T_m}. We now estimate $T_m$. For $i \neq k$, $i\in\llbracket 1,N\rrbracket$, we have
$$
\frac{1}{(\lambda v_i + w_{ii} - z - m^{(i)} - Z_i) (\lambda v_i - z - \widehat m_{fc})} = \frac{1}{(\lambda v_i - z - \widehat m_{fc})^2} + o(1)\,,
$$
where we have used that
\begin{align} \begin{split}
|w_{ii} - m^{(i)} - Z_i + \widehat m_{fc}| &\leq |w_{ii}| + |m - m^{(i)}| + |m - \widehat m_{fc}| + |Z_i| \\
&\leq \frac{N^{\epsilon}}{\sqrt N} + \frac{C}{N \eta} + N^{-1/2 + 3\epsilon} + C (\varphi_N)^{\xi} \sqrt{ \frac{\im m^{(i)}}{N \eta} } \\
&\ll |v_i - v_k|\,,
\end{split} \end{align}
which holds with high probability on $\Omega_V$. For $i=k$, we have
\begin{align} \begin{split}
|\lambda v_k + w_{kk}& - z - m^{(k)} - Z_k| + |\lambda v_k - z - \widehat m_{fc}| \\
&\geq |m - \widehat m_{fc}| - |w_{kk}| - |m - m^{(k)}| - |Z_k|\\
 &\geq \frac{1}{2} N^{-1/2 + 2\epsilon}\,,
\end{split} \end{align}
thus, as in the proofs of Lemma \ref{hat bound} and Lemma \ref{lem:step 2_1},
$$
\frac{1}{N} \left| \frac{1}{(\lambda v_k + w_{kk} - z - m^{(k)} - Z_k) (\lambda v_k - z - \widehat m_{fc})} \right| \leq C N^{-2\epsilon /3}\,,
$$
where we used that $|G_{kk}|\,,|g_k|\le \eta^{-1}$.

We now have that
\begin{align}
T_m = \widehat R_2^{(k)} + o(1) = R_2 + o(1)\,,
\end{align}
and, in particular, $T_m < c < 1$, with high probability on $\Omega_V$. We again let $M \deq \max_i |m^{(i)} - m + Z_i - w_{ii}|$ and find that $M \ll |m - \widehat m_{fc}|$ with high probability on $\Omega_V$, which contradicts the assumption. Therefore, if $|m(z)-\widehat m_{fc}(z) | \leq N^{-1/2 + 3\epsilon}$, then $|m(z)-\widehat m_{fc}(z)| \leq N^{-1/2 + 5\epsilon /3}$ with high probability on $\Omega_V$. In order to attain a uniform bound, we use the lattice argument as in the proof of Lemma \ref{lem:step 2_1}. This proves the desired lemma.
\end{proof}

We now prove Proposition \ref{prop:step 2_4} using a discrete continuity argument.

\begin{proof}[Proof of Proposition \ref{prop:step 2_4}]
Fix $E$ such that $z=E+\ii\eta_0\in\caD_{\epsilon}'$. Consider a sequence $(\eta_j)$ defined by $\eta_{j=0} = \eta_0$ and $\eta_j = \eta_{j-1} + N^{-2}$. Let $K$ be the smallest positive integer such that $\eta_K \geq N^{-1/2 + \epsilon}$. We prove by induction that, for $z_j = E + \ii \eta_j$, we have with high probability on $\Omega_V$ that
\begin{align}
|m(z_j)-\widehat m_{fc}(z_j)  | \leq \frac{N^{2\epsilon}}{\sqrt N}\,.
\end{align}
The case $j=K$ is already proved in Lemma \ref{lem:step 2_2}. For any $z = E + \ii \eta$, with $\eta_{j-1} \leq \eta \leq \eta_j$, we have
$$
|m(z_j) - m(z)| \leq \frac{|z_j - z|}{\eta_{j-1}^2} \leq \frac{N^{2 \epsilon}}{N}\,, \qquad\quad |\widehat m_{fc}(z_j) - \widehat m_{fc}(z)| \leq \frac{|z_j - z|}{\eta_{j-1}^2} \leq \frac{N^{2 \epsilon}}{N}\,.
$$
Thus, we find that if $|\widehat m_{fc}(z_j) - m(z_j)| \leq N^{-1/2 + 2\epsilon}$ then
$$
|m(z)-\widehat m_{fc}(z) | \leq N^{-1/2 + 2\epsilon} + \frac{2 N^{2 \epsilon}}{N} \ll N^{-1/2 + 3\epsilon}\,.
$$
We now invoke Lemma \ref{lem:step 2_3} to obtain that $|m(z)-\widehat m_{fc}(z) | \leq N^{-1/2 + 2\epsilon}$. This proves the desired lemma for any $z = E + \ii \eta$, with $\eta_{j-1} \leq \eta \leq \eta_j$. The desired lemma can now be proved by induction on $j$. Uniformity can now be obtained using a lattice argument.
\end{proof}

\subsection{Estimates on $|m - m^{(i)}|$}\label{aux estimate 1}
In order to derive a more accurate estimate on the difference $|\im m(z) - \im\widehat m_{fc}(z)|$, as the one obtained in~Proposition~\ref{prop:step 2_4}, we establish detailed estimates on $|m - m^{(i)}|$ and $N^{-1}\sum Z_i$. We first prove the following bound on the difference $|m - m^{(i)}|$.

\begin{lem} \label{lem:step 3}
There exists a constant $C > 1$ such that the following bound holds with $(\xi,\nu)$-high probability on $\Omega_V$ for all $z = E + \ii \eta_0 \in \caD_{\epsilon}'$: For given $z$, choose $k \in \llbracket 1, n_0 -1 \rrbracket$ such that~\eqref{assumption near v_k} is satisfied. Then, for any $i \neq k$, $i\in\llbracket 1,N\rrbracket$,
\begin{align}
|m(z) - m^{(i)}(z)| \leq C N^{1/(\b+1)} \frac{N^{4\epsilon}}{N}
\end{align}
and
\begin{align}
|m^{(k)}(z) - m^{(ki)}(z)| \leq C N^{1/(\b+1)} \frac{N^{4\epsilon}}{N}\,.
\end{align}
\end{lem}

\begin{proof}
Let $\eta = \eta_0$. Since
$$
G_{ij} = - G_{ii} \left( \sum_s^{(i)} w_{is} G_{sj}^{(i)} \right)\,,
$$
we find from the large deviation estimates in Lemma~\ref{lemma.LDE} and the Ward identity~\eqref{ward} that
$$
|G_{jj} - G_{jj}^{(i)}| = \left| \frac{G_{ij} G_{ji}}{G_{ii}} \right| \leq (\varphi_N)^{2\xi} |G_{ii}| \frac{\im G_{jj}^{(i)}}{N \eta}\,,
$$
with high probability on $\Omega_V$. For $i \neq k$, we have
$$
|G_{ii}| = \frac{1}{|\lambda v_i + w_{ii} - z - m^{(i)} - Z_i|} \leq \frac{C}{|\lambda v_i - z - \widehat m_{fc}|} \leq C N^{\epsilon} \kappa_0^{-1}\,,
$$
with high probability on $\Omega_V$. Thus, we obtain
\begin{align} \begin{split}
|m(z) - m^{(i)}(z)| &\leq \frac{|G_{ii}|}{N} + \frac{1}{N} \sum_{j}^{(i)} |G_{jj} - G_{jj}^{(i)}| \leq \frac{|G_{ii}|}{N} + C (\varphi_N)^{2\xi} \frac{N^{\epsilon} \kappa_0^{-1}}{N} \sum_{j}^{(i)} \frac{\im G_{jj}^{(i)}}{N \eta} \\
&\leq \frac{N^{\epsilon} \kappa_0^{-1}}{N} + C (\varphi_N)^{2\xi} \frac{N^{\epsilon} \kappa_0^{-1}}{N \eta} \im m^{(i)} \leq C N^{1/(\b+1)} \frac{N^{4\epsilon}}{N}\,,
\end{split} \end{align}
with high probability on $\Omega_V$. Together with the usual lattice argument, this proves the first part of the lemma. The second part of the lemma can be proved in a similar manner.
\end{proof}

\subsection{Estimates on $N^{-1}\sum Z_i$}\label{aux estimate 2}
Recall that $n_0>10$ is an integer independent of $N$. In the next lemma, we control the fluctuation average
$
\frac{1}{N} \sum_{i=n_0}^N Z_i
$
and other related quantities. Here, we aim to use cancellations in the averaging over $i$, but note that $(Z_i)$ are not independent.

\begin{lem} \label{lem:step 4}
There is a constant $c$, such that, for all  $z\in\caD_\epsilon'$, the following bounds hold with $(\xi-2,\nu)$-high probability on $\Omega_V$:
\begin{align}\label{statement Zlemma 1}
\left| \frac{1}{N} \sum_{i=n_0}^N Z_i(z) \right| \le (\varphi_N)^{c\xi}N^{-1/2-\fb/2+4\epsilon}\,,
\end{align}
and, for $k\in\llbracket 1,n_0-1\rrbracket$,
\begin{align}\label{statement Zlemma 2}
\left| \frac{1}{N} \sum_{\substack{i=n_0\\ i\not=k}}^N Z_i^{(k)}(z) \right| \le (\varphi_N)^{c\xi}N^{-1/2-\fb/2+4\epsilon}\,.
\end{align}
\end{lem}

\begin{cor} \label{cor:step 4}
There is a constant $c$, such that, for all $z\in\caD_\epsilon'$, the following bounds hold with $(\xi-2,\nu)$-high probability on $\Omega_V$:
\begin{align}
\left| \frac{1}{N} \sum_{i=n_0}^N \frac{w_{ii} - Z_i(z)}{(\lambda v_i - z - \widehat m_{fc}(z))^2} \right|  \le (\varphi_N)^{c\xi}N^{-1/2-\fb/2+4\epsilon}\,,
\end{align}
and, for $k\in\llbracket 1,n_0-1\rrbracket$,
\begin{align}
\left| \frac{1}{N} \sum_{\substack{i=n_0\\ i\not=k}}^N \frac{w_{ii} - Z_i^{(k)}(z)}{(\lambda v_i - z - \widehat m_{fc}(z))^2} \right| \le (\varphi_N)^{c\xi}N^{-1/2-\fb/2+4\epsilon}\,.
\end{align}

\end{cor}

\begin{rem}
The bounds we obtained in Lemma \ref{lem:step 3}, Lemma \ref{lem:step 4}, and Corollary \ref{cor:step 4} are~$o(\eta)$. This will be used on several occasions in the next subsection.
\end{rem}

Lemma \ref{lem:step 4} and Corollary \ref{cor:step 4} are proved in Section~\ref{sec:Zlemma}.

\subsection{Proof of Proposition~\ref{prop:step 2_4}}\label{Proof of Proposition 2_4}
Recall the definition of $(\widehat z_k)$ in~\eqref{definition of hatzk}. We first estimate $\im m(z)$ for $z = E + \ii \eta_0$ satisfying $|z - \widehat z_k| \geq N^{-1/2 + 3\epsilon}$, for all $k \in \llbracket 1, n_0 -1 \rrbracket$.

\begin{lem} \label{lem:step 5}
There exists a constant $C > 1$ such that the following bound holds with $(\xi-2,\nu$)-high probability on $\Omega_V$: For any $z = E + \ii \eta_0 \in \caD_{\epsilon}'$, satisfying $|z - \widehat z_k| \geq N^{-1/2 + 3\epsilon}$ for all $k \in \llbracket 1, n_0 -1 \rrbracket$, we have
\begin{align}
C^{-1} \eta \leq \im m(z) \leq C \eta\,.
\end{align}
\end{lem}

\begin{proof}
Let $z\in\caD_{\epsilon}'$ with $\eta=\eta_0$ and choose $k \in \llbracket 1, n_0 -1 \rrbracket$ such that~\eqref{assumption near v_k} is satisfied. Consider
\begin{align} \label{eq:step 5_1}
m = \frac{G_{kk}}{N} + \frac{1}{N} \sum_{i}^{(k)} \frac{1}{\lambda v_i + w_{ii} - z - m^{(i)} - Z_i}\,.
\end{align}
From the assumption in~\eqref{assumption near v_k}, Corollary \ref{cor:step 2_1}, and Proposition \ref{prop:step 2_4}, we find that, with high probability on~$\Omega_V$,
\begin{align} \begin{split} \label{eq:step 5_2}
&\left| \frac{1}{N} \sum_{i}^{(k)} \left( \frac{1}{\lambda v_i + w_{ii} - z - m^{(i)} - Z_i} - \frac{1}{\lambda v_i - z - \widehat m_{fc}} - \frac{m^{(i)} - \widehat m_{fc} + Z_i - w_{ii}}{(\lambda v_i - z - \widehat m_{fc})^2} \right) \right| \\
&\leq \frac{C}{N} \sum_{i}^{(k)} \frac{N^{-1 + 4\epsilon}}{|\lambda v_i - z - \widehat m_{fc}|^3} \leq C \frac{N^{4 \epsilon}}{N} N^{\epsilon} N^{1/(\b+1)} \frac{1}{N} \sum_{i}^{(k)} \frac{1}{|\lambda v_i - z - \widehat m_{fc}|^2} \ll \eta\,.
\end{split} \end{align}
We also observe that
$$
\left| \frac{1}{N} \sum_{ \substack{i=1\\i \neq k}}^{n_0} \frac{w_{ii} - Z_i}{(\lambda v_i - z - \widehat m_{fc}(z))^2} \right| \leq C N^{-1} N^{-1/2 + 2\epsilon} N^{1/(\b+1)} \ll N^{-1} \ll \eta\,.
$$
Thus, from Lemma \ref{lem:step 3} and Corollary \ref{cor:step 4}, we find with high probability on $\Omega_V$ that
\begin{align} \label{eq:step 5_3}
\frac{1}{N} \sum_{i}^{(k)} \frac{m^{(i)} - \widehat m_{fc} + Z_i - w_{ii}}{(\lambda v_i - z - \widehat m_{fc})^2} = \frac{1}{N} \sum_{i}^{(k)} \frac{m - \widehat m_{fc}}{(\lambda v_i - z - \widehat m_{fc})^2} + o(\eta)\,.
\end{align}
Recalling~\eqref{eq:step 1_1}, i.e.,
$$
|\lambda v_k - \re(z + \widehat m_{fc}(z))| \gg N^{-1/2 + 2\epsilon}\,,
$$
we get $|G_{kk}| \leq N^{1/2 - 2\epsilon}$. We thus obtain from \eqref{eq:step 5_1}, \eqref{eq:step 5_2}, and \eqref{eq:step 5_3} that, with high probability on~$\Omega_V$,
\begin{align}
m = \frac{1}{N} \sum_{i}^{(k)} \left( \frac{1}{\lambda v_i - z - \widehat m_{fc}} + \frac{m - \widehat m_{fc}}{(\lambda v_i - z - \widehat m_{fc})^2} \right) + o(\eta)\,.
\end{align}
We also notice that, with high probability on $\Omega_V$,
\begin{align}
\frac{1}{N} \sum_{i}^{(k)} \frac{1}{\lambda v_i - z - m} = \frac{1}{N} \sum_{i}^{(k)} \left( \frac{1}{\lambda v_i - z - \widehat m_{fc}} + \frac{m - \widehat m_{fc}}{(\lambda v_i - z - \widehat m_{fc})^2} \right) + \caO \left( \frac{1}{N} \sum_{i}^{(k)} \frac{N^{-1 + 4\epsilon}}{|\lambda v_i - z - \widehat m_{fc}|^3} \right)\,,
\end{align}
and following the estimate in \eqref{eq:step 5_2}, we find that
\begin{align}
\frac{1}{N} \sum_{i}^{(k)} \frac{1}{\lambda v_i - z - m} = \frac{1}{N} \sum_{i}^{(k)} \left( \frac{1}{\lambda v_i - z - \widehat m_{fc}} + \frac{m - \widehat m_{fc}}{(\lambda v_i - z - \widehat m_{fc})^2} \right) + o(\eta) = m + o(\eta)\,.
\end{align}
Taking imaginary parts, we get
$$
\im m = \frac{1}{N} \sum_{i}^{(k)} \frac{\eta + \im m}{|\lambda v_i - z - m|^2} + o(\eta)\,.
$$
Since
$$
\frac{1}{N} \sum_{i}^{(k)} \frac{1}{|\lambda v_i - z - m|^2} = \frac{1}{N} \sum_{i}^{(k)} \frac{1}{|\lambda v_i - z - m_{fc}|^2} + o(1) < c < 1\,,
$$
for some constant $c$, we can conclude that $C^{-1} \eta \leq \im m \leq C \eta$ with high probability for some $C > 1$. This proves the desired lemma.
\end{proof}

As a next step, we prove that there exists $\widetilde z_k = \widetilde E_k + \ii \eta_0$ near $\widehat z_k$ such that $\im m(\widetilde z_k) \gg \eta$. Before proving this, we first show that $\im m^{(k)} (z) \sim \eta$ even if $z$ is near $\widehat z_k$.

\begin{lem} \label{lem:step 6_1}
There exists a constant $C > 1$ such that the following bound holds with $(\xi-2,\nu)$-high probability on $\Omega_V$, for all  $z = E + \ii \eta_0 \in \caD_{\epsilon}'$: For given $z$, choose $k \in \llbracket 1, n_0 -1 \rrbracket$ such that~\eqref{assumption near v_k} is satisfied. Then, we have
\begin{align}
C^{-1} \eta_0 \leq \im m^{(k)}(z) \leq C \eta_0\,.
\end{align}
\end{lem}

\begin{proof}
Reasoning as in the proof of Lemma \ref{lem:step 5}, we find from Proposition \ref{prop:step 2_4}, Corollary \ref{cor:step 2_1}, Lemma \ref{lem:step 3}, and Corollary \ref{cor:step 4} that, with high probability on $\Omega_V$,
\begin{align}
m^{(k)} = \frac{1}{N} \sum_i^{(k)} \left( \frac{1}{\lambda v_i - z - \widehat m_{fc}} + \frac{m^{(k)} - \widehat m_{fc}}{(\lambda v_i - z - \widehat m_{fc})^2} \right) + o(\eta_0) = \frac{1}{N} \sum_i^{(k)} \frac{1}{\lambda v_i - z -m^{(k)}} + o(\eta_0)\,.
\end{align}
Considering the imaginary part, we can prove the desired lemma as in the proof of Lemma \ref{lem:step 5}.
\end{proof}

\begin{cor} \label{cor:step 6_1}
There exists a constant $C > 1$ such that the following bound holds with $(\xi-2,\nu)$-high probability on $\Omega_V$, for all $z = E + \ii \eta_0 \in \caD_{\epsilon}'$:
For given $z$, choose $k \in \llbracket 1, n_0 -1 \rrbracket$ such that~\eqref{assumption near v_k} is satisfied. Then, we have
\begin{align}
|Z_k| \leq C \frac{(\varphi_N)^{\xi}}{\sqrt N}\,.
\end{align}
\end{cor}

We are now ready to locate the points $z\in\caD_{\epsilon}'$ for which $\im m(z) \gg \eta_0$.

\begin{lem} \label{lem:step 6_2}
For any $k \in \llbracket 1, n_0-1 \rrbracket$, there exists $\widetilde E_k\in\R$ such that the following holds with $(\xi-2,\nu)$-high probability on $\Omega_V$: If we let $\widetilde z_k\deq\widetilde E_k + \ii \eta_0$, then $|\widetilde z_k - \widehat z_k| \leq N^{-1/2 + 3\epsilon}$ and $\im m(\widetilde z_k) \gg \eta_0$.
\end{lem}

\begin{proof}
We first notice that the condition $|z - \widehat z_k| \geq N^{-1/2 + 3\epsilon}$ has not been used in the derivation of~\eqref{eq:step 5_2} and~\eqref{eq:step 5_3}, hence, even though $|z - \widehat z_k| \leq N^{-1/2 + 3\epsilon}$, we still attain that
\begin{align} \label{eq:step 6_2}
m = \frac{G_{kk}}{N} + \frac{1}{N} \sum_i^{(k)} \frac{1}{\lambda v_i + w_{ii} - z - m^{(i)} - Z_i} = \frac{G_{kk}}{N} + \frac{1}{N} \sum_i^{(k)} \frac{1}{\lambda v_i - z - m} + o(\eta_0)\,,
\end{align}
with high probability on $\Omega_V$. Consider
$$
\frac{1}{G_{kk}} = \lambda v_k + w_{kk} - z - m^{(k)} - Z_k\,.
$$
Setting $z_k^+ \deq \widehat z_k + N^{-1/2 + 3\epsilon}$, Lemma \ref{mfc estimate} shows that
$$
\re(z_k^+ + m_{fc}(z_k^+)) - \re(\widehat z_k + m_{fc}(\widehat z_k)) \geq C N^{-1/2 + 3\epsilon}\,,
$$
on $\Omega_V$. Thus, from Lemma~\ref{hat bound} and the definition of $\widehat z_k$, we find that
$$
\lambda v_k - \re ( z_k^+ + \widehat m_{fc}(z_k^+) ) \leq -C N^{-1/2 + 3\epsilon}\,,
$$
on $\Omega_V$. Similarly, if we let $z_k^- \deq \widehat z_k - N^{-1/2 + 3\epsilon}$, we have that
$$
\lambda v_k - \re ( z_k^- + \widehat m_{fc}(z_k^-) ) \geq C N^{-1/2 + 3\epsilon}\,,
$$
on $\Omega_V$. Since
$$
|w_{kk} + \widehat m_{fc} - m^{(k)} - Z_k| \leq |w_{kk}| + |m-\widehat m_{fc}| + |m^{(k)}-m| + |Z_k| \ll N^{-1/2 + 3\epsilon}\,,
$$
with high probability on $\Omega_V$, we find that there exists $\widetilde z_k = \widetilde E_k + \ii \eta_0$, with $\widetilde E_k \in (\widehat E_k - N^{-1/2 + 3\epsilon}, \widehat E_k + N^{-1/2 + 3\epsilon})$, such that $\re G_{kk}(\widetilde z_k) = 0$. When $z = \widetilde z_k$, we have from Lemma \ref{lem:step 6_1} and Corollary \ref{cor:step 6_1} that, with high probability on $\Omega_V$, 
\begin{align}
 |\im G_{kk}(\widetilde z_k)| = \frac{1}{|\im \widetilde z_k + \im m^{(k)}(\widetilde z_k) + \im Z_k(\widetilde z_k)|} \geq C (\varphi_N)^{-\xi} N^{1/2}\,,\qquad\quad \re G_{kk}(\widetilde z_k)=0\,.
\end{align}
From~\eqref{eq:step 6_2}, we obtain that
\begin{align}
\im m(\widetilde z_k) = \frac{ {\im G_{kk}(\widetilde z_k)} }{N} + \frac{1}{N} \sum_i^{(k)} \frac{\eta_0 + \im m(\widetilde z_k)}{|\lambda v_i - \widetilde z_k - m(\widetilde z_k)|^2} + o(\eta_0)\,.
\end{align}
Since
$$
\sum_i^{(k)} \frac{1}{|\lambda v_i - \widetilde z_k - m(\widetilde z_k)|^2} < c < 1\,,
$$
with high probability on $\Omega_V$, for some constant $c$, we get
\begin{align}
\im m(\widetilde z_k) \geq C (\varphi_N)^{-\xi} N^{-1/2} + C \eta_0 \gg \eta_0 \,, 
\end{align} 
with high probability on $\Omega_V$, which was to be proved.
\end{proof}

We now turn to the proof of Proposition~\ref{prop:mu_k}. Recall that we denote by $\mu_k$ the $k$-th largest eigenvalue of $H$, $k\in\llbracket 1,n_0-1\rrbracket$.  Also recall that $\kappa_0=N^{-1/(\b+1)}$; see~\eqref{definition of kappa0}.
\begin{proof}[Proof of Proposition \ref{prop:mu_k}]
We first consider the case $k=1$. From the spectral decomposition of $H$, we have
\begin{align}
\im m(E + \ii \eta_0) = \frac{1}{N} \sum_{\alpha=1}^N \frac{\eta_0}{(\mu_{\alpha} - E)^2 + \eta_0^2}\,,
\end{align}
and in particular, $\im m(\mu_1 + \ii \eta_0) \geq (N \eta_0)^{-1} \gg \eta_0$. Recall that $\mu_1 \leq 3 + \lambda$ with high probability as discussed in~\eqref{mu_k a priori}. Recall the definition of $\widehat z_1=\widehat E_1+\ii\eta_0$ in~\eqref{definition of hatzk}. Since, with high probability on $\Omega_V$, $\im m(z) \sim \eta_0$ for $z \in \caD_{\epsilon}'$ satisfying $|z - \widehat z_1| \geq N^{-1/2 + 3\epsilon}$, as we proved in Lemma \ref{lem:step 6_2}, we obtain that $\mu_1 < \widehat E_1 + N^{-1/2 + 3\epsilon}$.

Recall the definitions for $\widehat z_1$ and $z_1^-$ in the proof of Lemma \ref{lem:step 6_2}. Assume that $\mu_1 < \widehat E_1 - N^{-1/2 + 3\epsilon}$. Then, on the interval $(\widehat E_1 - N^{-1/2 + 3\epsilon}, \widehat E_1 + N^{-1/2 + 3\epsilon})$, $\im m(E + \ii \eta_0)$ is a decreasing function of $E$. However, we already showed in Lemma \ref{lem:step 5} and Lemma \ref{lem:step 6_2} that, with $(\xi-2,\nu)$-high probability, $\im m(\widetilde z_1) \gg \eta_0$, $\im m(z_1^-) \sim \eta_0$, and $\re \widetilde z_1 > \re z_1^-$. Thus, $\mu_1 \geq \widehat E_1 - N^{-1/2 + 3\epsilon}$. We now use Lemma~\ref{mfc estimate} and Lemma~\ref{hat bound}, together with Remark~\ref{im mfc upper bound}, to conclude that
\begin{align}
\mu_1 + \ii \eta_0 + \widehat m_{fc} (\mu_1 + \ii \eta_0) = \widehat z_1 + \widehat m_{fc} (\widehat z_1) + \caO (N^{-1/2 + 3\epsilon}) = \lambda v_1 + \caO (N^{-1/2 + 3\epsilon})\,,
\end{align}
which proves the proposition for the special choice $k=1$.

Next, we consider the case $k=2$; the general case can be proved in a similar manner by induction. Consider~$H^{(1)}$, the minor of $H$ obtained by removing the first column and the first row. If we denote by $\mu_1^{(1)}$ the largest eigenvalue of $H^{(1)}$, then the Cauchy interlacing property yields $\mu_2 \leq \mu_1^{(1)}$. We now follow the first part of the proof to estimate~$\mu_1^{(1)}$, which gives us
\begin{align}
\widehat E_2 - N^{-1/2 + 3\epsilon} \leq \mu_1^{(1)} \leq \widehat E_2 + N^{-1/2 + 3\epsilon}\,,
\end{align}
where we let $\widehat z_2 = \widehat E_2 + \ii \eta_0$ be a solution to the equation
$$
\re \left( \widehat z_2 + \widehat m_{fc} (\widehat z_2) \right) = \lambda v_2\,.
$$
This, in particular, shows that
\begin{align}
\mu_2 \leq \widehat E_2 + N^{-1/2 + 3\epsilon}\,.
\end{align}

To prove the lower bound, we may follow the arguments we used in the first part of the proof. Recall that we have proved in Lemma \ref{lem:step 5} and Lemma \ref{lem:step 6_2} that, with $(\xi-2,\nu)$-high probability on $\Omega_V$,
\begin{enumerate}
\item[(1)] for $z = \widehat z_2 - N^{-1/2 + 3\epsilon}$, we have $\im m(z) \leq C \eta_0$;
\item[(2)] there exists $\widetilde z_2 = \widetilde E_2 + \ii \eta_0$, satisfying $|\widetilde z_2 - \widehat z_2| \leq N^{-1/2 + 3\epsilon}$, such that $\im m(\widetilde z_2) \gg \eta_0$.
\end{enumerate}
If $\mu_2 < \widehat E_2 - N^{-1/2 + 3\epsilon}$, then 
$$
\im m(E + \ii \eta_0) - \frac{1}{N} \frac{\eta_0}{(\mu_1 - E)^2 + \eta_0^2} = \frac{1}{N} \sum_{\alpha=2}^N \frac{\eta_0}{(\mu_{\alpha} - E)^2 + \eta_0^2}
$$
is a decreasing function of $E$. Since we know that, with $(\xi-2,\nu)$-high probability on $\Omega_V$,
$$
\frac{1}{N} \frac{\eta_0}{(\mu_1 - \widehat E_2)^2 + \eta_0^2} \leq \frac{1}{N}\frac{C \eta_0}{  N^{-2\epsilon} \kappa_0^2} \ll \eta_0\,,
$$
we have that $\im m(\widetilde z_2) \leq C \eta_0$, which contradicts the definition of $\widetilde z_2$. Thus, we find that $\mu_2 \geq \widehat E_2 - N^{-1/2 + 3\epsilon}$, with $(\xi-2,\nu)$-high probability on $\Omega_V$.

We now proceed as above to conclude that, with $(\xi-2,\nu)$-high probability on $\Omega_V$,
\begin{align}
\mu_2 + \ii \eta_0 + \widehat m_{fc} (\mu_2 + \ii \eta_0) = \widehat z_2 + \widehat m_{fc} (\widehat z_2) + \caO (N^{-1/2 + 3\epsilon}) = \lambda v_2 + \caO (N^{-1/2 + 3\epsilon})\,,
\end{align}
which proves the proposition for $k=2$. The general case is proven in the same way.
\end{proof}

\section{Fluctuation Average Lemma} \label{sec:Zlemma}
In this section we prove Lemma~\ref{lem:step 4} and Corollary~\ref{cor:step 4}. Recall that we denote by $\E_i$ the partial expectation with respect to the $i$-th column/row of $W$. Set $Q_i\deq\lone-\E_i$. 

We are interested in bounding the fluctuation average
\begin{align}\label{donkey}
 \frac{1}{N}\sum_{a=n_0}^NZ_a(z)\,,
\end{align}
where $n_0$ is a $N$-independent fixed integer. We first note that, using Schur's complement formula, we can write
\begin{align}
\frac{1}{N}\sum_{a=n_0}^{N}Q_a\left(\frac{1}{G_{aa}}\right)&=\frac{1}{N}\sum_{a=n_0}^Nw_{aa}-\frac{1}{N}\sum_{a=n_0}^NQ_a\sum_{k,l}^{(a)}w_{ak}G_{kl}^{(a)}w_{la}\nonumber\\
&=-\frac{1}{N}\sum_{a=n_0}^N Z_a+\caO\left(\frac{(\varphi_N)^{c\xi}}{N}\right)\,,\label{donkey 2}
\end{align}
with high probability, where we used the large deviation estimate~\eqref{LDE1}. The main result of this section, Lemma~\ref{the Zlemma}, asserts that
\begin{align}\label{donkey 1}
 \left|\frac{1}{N}\sum_{a=n_0}^{N}Q_a\left(\frac{1}{G_{aa}}\right)\right|\le (\varphi_N)^{c\xi} N^{-1/2-\fb/2+4\epsilon}\,, 
\end{align}
with $(\xi-2,\nu)$-high probability, provided that $z$ satisfies $|\lambda v_a-\re {m}_{fc}(z)-\re z|\ge\frac{1}{2} N^{-1/(b+1)+\epsilon}$, for all $a\ge n_0$. (Note that we reduced here $\xi$ to $\xi-2$).

Fluctuation averages of the form~\eqref{donkey} (with $n_0=1$) and more general fluctuation averages have been studied in~\cite{EKY}, see also~\cite{EKYY4}, for generalized Wigner ensembles and random band matrices. In~\cite{LS}, these ideas have been applied to the deformed Wigner ensemble, under the assumptions that the limiting eigenvalue distribution has a square root behavior at the spectral edge. In these studies, it was assumed that there is a deterministic control parameter $\Lambda_o\equiv\Lambda_o(z)$, such that $G_{ij}(z)$ and $Z_i(z)$ satisfy $\max_{i,j} |G_{ij}|\le C\Lambda_o+C\delta_{ij}$, $|G_{ii}(z)|\ge c$, and $\max_i|Z_i|\le C\Lambda_o$, with high probability, and $\Lambda_o$ satisfies $\Lambda_o\ll 1$, for $\im z\gg N^{-1}$.

Under the assumption of Lemma~\ref{lem:step 4}, the Green function entries $(G_{ij}(z))$ can become large, i.e., $|G_{ij}(z)|\gg 1$, $\im \eta \sim N^{-1/2}$, for certain choices of the spectral parameter $z$ (close to the spectral edge) and certain choice of indices $i,j$. However, resolvent fractions of the form $G_{ab}(z)/G_{bb}(z)$ and $G_{ab}(z)/G_{aa}(z)G_{bb}(z)$ ($a,b\ge n_0$), are small (see Lemma~\ref{lemma for assumption on xi} below for a precise statement). Using this observation, we adapt the methods of~\cite{EKYY4} to control the fluctuation average~\eqref{donkey}.

\subsection{Preliminaries}
Let $a,b\in\llbracket 1,N\rrbracket$ and $\T,\T'\subset\llbracket 1,N\rrbracket$, with $a,b\not\in\T$, $b\not\in \T'$, $a\not=b$,  then we set
\begin{align}
F_{ab}^{(\T,\T')}(z)\deq\frac{G_{ab}^{(\T)}(z)}{G_{bb}^{(\T')}(z)}\,,\qquad \quad(z\in\C^+)\,,
\end{align}
and we often abbreviate $F_{ab}^{(\T,\T')}\equiv F_{ab}^{(\T,\T')}(z)$. In case $\T=\T'=\emptyset$, we simply write $F_{ab}\equiv F_{ab}^{(\T,\T')}$. Below we will always implicitly assume that $\{a,b\}$ and $\T,\T'$ are compatible in the sense that $a\not=b$, $a,b\not\in\T$, $b\not\in\T'$.

Starting from~\eqref{basic resolvent}, simple algebra yields the following relations among the $\lbrace F_{ab}^{(\T,\T')}\rbrace$.
\begin{lem}
 Let $a,b,c\in\llbracket1,N\rrbracket$, all distinct, and let $\T,\T'\subset\llbracket 1,N\rrbracket$. Then,
\begin{itemize}
 \item[(1)] for $c\not\in\T\cup \T'$,
\begin{align}\label{Zlemma expand 1}
 F_{ab}^{(\T,\T')}=F_{ab}^{(\T c,\T')}+F_{ac}^{(\T,\T')}F_{cb}^{(\T,\T')}\,;
\end{align}
\item[(2)] for $c\not\in\T\cup \T'$,
\begin{align}\label{Zlemma expand 2}
 F_{ab}^{(\T,\T')}=F_{ab}^{(\T,\T'c)}- F_{ab}^{(\T,\T'c)}F_{bc}^{(\T,\T')}F_{cb}^{(\T,\T')}\,;
\end{align}
\item[(3)] for $c\not\in\T$,
\begin{align}\label{Zlemma expand basic}
 \frac{1}{G_{aa}^{(\T)}}=\frac{1}{G_{aa}^{(\T c)}}\left(1-F_{ac}^{(\T,\T)}F_{ca}^{(\T,\T)}\right)\,.
\end{align}

\end{itemize}

\end{lem}

\subsection{The fluctuation average lemma}
Recall the definition of the domain $\caD_{\epsilon}'$ of the spectral parameter in~\eqref{a index assumption} and of the constant $\fb>0$ in~\eqref{fb}. Set $A\deq\llbracket n_0,N\rrbracket$. To start with, we bound $F_{ab}$ and $F_{ab}^{(\emptyset,a)}/G_{aa}$ on the domain~$\caD_{\epsilon}'$.  
\begin{lem}\label{lemma for assumption on xi}
Assume that, for all $z\in\caD_{\epsilon}'$, the estimates
\begin{align}\label{weaklocallaw}
 |m(z)-\widehat{m}_{fc}(z)|\le N^{-1/2+2\epsilon}\,,\qquad \im m(z)\le N^{-1/2+2\epsilon}\,,
\end{align}
hold with $(\xi,\nu)$-high probability on $\Omega_V$.

Then, there exists a constant $c$, such that for all $z\in\caD_{\epsilon}'$,
\begin{align}\label{F bound 1}
 \max_{\substack{a,b\in A\\ a\not=b}}|F_{ab}(z)|\le (\varphi_N)^{c\xi} N^{-\fb/2}N^{\epsilon}\,,\qquad\quad( z\in\caD_{\epsilon}')\,,
\end{align}
 and
\begin{align}\label{F bound 2}
 \max_{\substack{a,b\in A\\ a\not=b}}\left|\frac{F_{ab}^{(\emptyset,a)}(z)}{G_{aa}(z)}\right|\le (\varphi_N)^{c\xi} N^{-1/2}N^{2\epsilon}\,,\qquad\quad( z\in\caD_{\epsilon}')\,,
\end{align}
with $(\xi,\nu)$-high probability on $\Omega_V$.
\end{lem}
\begin{proof}

Dropping the $z$-dependence from the notation, we first note that by Schur's complement formula~\eqref{schur} and Inequality~\eqref{weaklocallaw}, we have with high probability on $\Omega_V$, for $z\in\caD_{\epsilon}'$,
\begin{align}
 \frac{1}{G^{(b)}_{aa}}=\lambda v_a-z-\widehat{m}_{fc}+\caO(|m-m^{(ab)}|)+\caO(N^{-1/2+2\epsilon})\,,
\end{align}
for all $a \in A$, $b\in\llbracket 1,N\rrbracket$, $a\not=b$. Thus, for $z\in\caD_{\epsilon}'$, Lemma~\ref{cauchy interlacing} yields
\begin{align}\label{consequence of index assumption}
 |G^{(b)}_{aa}|\le C(\varphi_N)^{\xi}N^{1/{(\b+1)}}N^{\epsilon}\,,
\end{align}
with high probability on $\Omega_V$. Further, from the resolvent formula~\eqref{onesided} we obtain
\begin{align}
 F_{ab}=-\sum_{k}^{(b)} G_{ak}^{( b)}h_{kb}\,,
\end{align}
for $a,b\in A$, $a\not= b$.
 From the large deviation estimate~\eqref{LDE1} we infer that
\begin{align}
 \left|\sum_{k}^{(b)} G_{ak}^{(b)}h_{kb}\right|\le (\varphi_N)^{\xi}\left(\frac{\im G^{( b)}_{aa}}{N\eta}\right)^{1/2}\,,
\end{align}
with high probability, and hence conclude by~\eqref{consequence of index assumption} that
\begin{align}
 |F_{ab}|\le C(\varphi_N)^{2\xi}N^{-\fb/2}N^{\epsilon}\,,
\end{align}
with high probability on $\Omega_V$.

To prove the second claim, we recall that, for $a\not=b$, the resolvent formula~\eqref{twosided} gives
\begin{align}
 \frac{F_{ab}^{(\emptyset, a)}}{G_{aa}}= -h_{ab}+\sum_{k,l}^{( ab) }h_{ak}G_{kl}^{(ab)}h_{lb}\,,
\end{align}
and we conclude from the large deviation estimates~\eqref{LDE4} and~\eqref{bound on wij} that
\begin{align}
 \left|\frac{F_{ab}^{(\emptyset, a)}}{G_{aa}}\right|\le \frac{(\varphi_N)^{\xi}}{\sqrt{N}}+(\varphi_N)^{\xi}\sqrt{\frac{\im m^{(ab)}}{N\eta}}\,,
\end{align}
with high probability. Since $|m-m^{(ab)}|\le CN^{-1/2+\epsilon}$ on $\caD_{\epsilon}'$, by Lemma~\ref{cauchy interlacing}, we obtain using~\eqref{weaklocallaw}
\begin{align}
 \left|\frac{F_{ab}^{(\emptyset,a)}}{G_{aa}}\right|\le (\varphi_N)^{c\xi}N^{-1/2}N^{2\epsilon}\,,
\end{align}
with high probability on $\Omega_V$.

\end{proof}
\begin{defn}\label{definition of xi}
 Let $\Xi$ be an event defined by requiring that the following holds on it:  $(1)$ there exists a constant~$c$, such that for all $z\in\caD_{\epsilon}'$,~\eqref{weaklocallaw},~\eqref{F bound 1} and~\eqref{F bound 2} hold; $(2)$ there exists a constant  $c$ such that, for all~$z\in\caD_{\epsilon}'$ and~$a\in A$,
\begin{align}
\left|Q_{a}\left(\frac{1}{G_{aa}}\right)\right|\le (\varphi_N)^{c\xi}N^{-1/2+2\epsilon}\,;
\end{align}
and $(3)$, for all $i,j\in\llbracket 1,N\rrbracket$,
\begin{align}
 \max_{ij}|w_{ij}|\le \frac{(\varphi_N)^{c\xi}}{\sqrt{N}}\,.
\end{align}
\end{defn}
By Lemma~\ref{lemma for assumption on xi}, Lemma~\ref{lem:step 2_3}, Corollary~\ref{cor:step 2_1}, Lemma~\ref{lem:step 2_1} and Inequality~\eqref{bound on wij}, we know that $\Xi$ holds with $(\xi,\nu)$-high probability on~$\Omega_V$.

\begin{cor}\label{Zlemma 1}

For $p\le (\log N)^{\xi}$, there exists a constant $c$, such that the following holds. For all $\T,\T',\T''\subset A$, with $|\T|\,,|\T'|\,,|\T''|\le p$, for all $a,b\in A $, $a\not=b$, and, for all $z\in\caD_{\epsilon}'$, we have
\begin{align}\label{Zlemma bound 1}
 \lone(\Xi)\left|{F^{(\T,\T')}_{ab}(z)}\right|\le (\varphi_N)^{c\xi} N^{-\fb/2}N^{\epsilon}\,,
\end{align}
\begin{align}\label{Zlemma bound 2}
 \lone(\Xi)\left|\frac{F_{ab}^{(\T',\T'')}(z)}{G_{aa}^{(\T)}(z)} \right|\le(\varphi_N)^{c\xi} N^{-1/2}N^{2\epsilon} \,,
\end{align}
and
\begin{align}\label{initial estimate on Q_a}
 \lone(\Xi)\left|Q_{a}\left(\frac{1}{G_{aa}^{(\T)} }\right) \right|\le(\varphi_N)^{c\xi} N^{-1/2}N^{2\epsilon} \,,
\end{align}
on $\Omega_V$, for $N$ sufficiently large.

\end{cor}
The proof of this corollary is given in Appendix B.

Before we state the next lemma, we remark that in this section we use the symbol $\E^W$ for the partial expectation with respect to the random variables $(w_{ij})$ with $(v_i)$ kept fixed, i.e., $\E^W[\,\cdot\,]\equiv\E[\,\cdot\,| (v_i)]$.
\begin{lem}\label{Jensen lemma}
 Let $p\in\N$ satisfy $p\le (\log N)^{\xi-3/2}$. Let $q\in\llbracket 0,p\rrbracket$ and consider random variables $(\caX_{i})\equiv(\caX_{i}(H))$  and $(\caY_{i})\equiv(\caY_{i}(H))$, $i\in\llbracket 1,p\rrbracket$, satisfying
\begin{align}\label{estimates on good event}
 \lone(\Xi)|\caX_{i}|\le C(\varphi_N)^{c\xi d_{i}}N^{-1/2+2\epsilon}N^{-(d_{i}-1)(\fb/2-\epsilon)}\,,\qquad \lone(\Xi)|Q_{i}\caY_{i}|\le C(\varphi_N)^{c\xi}N^{-1/2}N^{2\epsilon}\,,
\end{align}
     where $d_{i}\in\N_0$ satisfy $0\le s=\sum_{i=1}^q (d_{i}-1)\le p+2$. Assume moreover that there are constants $C'$ and $K$, such that 
\begin{align}\label{estimates on expectation}
\E^W |\caX_{i}|^{r}\le (C'N)^{K(d_{i}+1) r}\,,\qquad\E^W|\caY_{i}|^{r}\le (C'N)^{Kr}\,,\end{align}
for any $r\in\N$, with $r\le 10 p$. Finally, assume that the event $\Xi$ holds with $(\xi,\nu)$-high probability.

Then, there is a constant $c_0$, depending only on $C,C',c,K$ in~\eqref{estimates on good event} and~\eqref{estimates on expectation}, and  on $\xi$ and $\nu$, such that
\begin{align}\label{jensen lemma bound}
 \left|\E^W\prod_{i=1}^q Q_{i}(\caX_{i})\prod_{i=q+1}^p Q_{i}(\caY_{i})\right|\le (\varphi_N)^{c_0\xi p}N^{-p/2-s\fb/2}N^{(2p+s)\epsilon}\,.
\end{align}
(Here, we use the convention that, for $q=0$, the first product is set to one, and, similarly, for $q=p$, the second product is set to one.)
\end{lem}
Lemma~\ref{Jensen lemma} is proved in Appendix~B.

Next, we state the main result of this section:

\begin{lem}\label{the Zlemma}\emph{[Fluctuation Average Lemma]}
Let $A\deq\llbracket n_0,N\rrbracket$. Recall the definition of the domain $\caD_{\epsilon}'$ in~\eqref{a index assumption}. Let $\Xi$ denote the event in Definition~\ref{definition of xi} and assume it has $(\xi,\nu)$-high probability. Then there exist constants $C$, $c$, $c_0$, such that for $p=2r$, $r\in\N$, $p\le (\log N)^{\xi-3/2}$, we have
\begin{align}\label{estimate in Zlemma}
 \E^W\left|\frac{1}{N}\sum_{a\in A}Q_{a}\left(\frac{1}{G_{aa}(z)}\right) \right|^{p}\le (Cp)^{cp}(\varphi_N)^{c_0p\xi} N^{-p/2-p\fb/2}N^{3p\epsilon}\,,
\end{align}
for all $z\in\caD_{\epsilon}'$, on $\Omega_V$.
\end{lem}

\begin{proof}
Fix $z\in\caD_{\epsilon}'$. For simplicity we drop the $z$-dependence from our notation and we always work on $\Omega_V$. We explain the idea of the proof for the simple case $p=2$. First, we note that
\begin{align}
 \frac{1}{N^2}\sum_{a\in A}\E^W\,\left|Q_{a}\left(\frac{1}{G_{aa}}\right)\right|^2\le C(\varphi_N)^{2c_0\xi}\frac{N^{4\epsilon}}{N^2}\,,
\end{align}
where we used~\eqref{jensen lemma bound} (with $p=2$, $q=0$, $\caY_{1}=(\overline{G_{aa}})^{-1}$, $\caY_{2}={({G_{aa}})^{-1}}$). Here, we also used that $\E^W |G_{aa}|^{-r}\le (CN)^{Kr}$, for some $K>0$, $r\in \N$, (see Remark~\ref{remark about easy estimate} below) to ensure~\eqref{estimates on expectation}. It thus suffices to consider
\begin{align}
 \frac{1}{N^2}\sum_{\substack{a_1\in A,a_2\in A\\a_1\not=a_2}}\E^W\, Q_{a_1}\overline{\left(\frac{1}{G_{a_1a_1}}\right)} Q_{a_2}\left(\frac{1}{G_{a_2a_2}} \right)\,.
\end{align}
Applying the formula
\begin{align}\label{expand 1}
 \frac{1}{G_{aa}}=\frac{1}{G_{aa}^{(b)}}\left(1-F_{ab}F_{ba}\right)\,, \qquad (a\not=b)\,,
\end{align}
twice, we obtain
\begin{align}\label{Zlemma example}
  \E^W\, Q_{a_1}\overline{\left(\frac{1}{G_{a_1a_1}}\right)} Q_{a_2}\left(\frac{1}{G_{a_2a_2}} \right)&=\E^W\, Q_{a_1}\overline{\left({\frac{1}{G_{a_1a_1}^{(a_2)}}\left(1-F_{a_1a_2}F_{a_2a_1} \right)}\right)}\,Q_{a_2}\left({\frac{1}{G_{a_2a_2}^{(a_1)}}\left(1-F_{a_2a_1}F_{a_1a_2} \right)}\right)\nonumber\\
 &=\E^W\, Q_{a_1}\overline{\left({\frac{1}{G_{a_1a_1}^{(a_2)}}F_{a_1a_2}F_{a_2a_1}}\right)}\,Q_{a_2}\left({\frac{1}{G_{a_2a_2}^{(a_1)}}F_{a_2a_1}F_{a_1a_2} }\right)\,,
\end{align}
where we used that $G_{aa}^{(b)}$, $a\not=b$, is independent of the entries in the $b$-th column/row of $W$, and that, for general random variables $A=A(W)$ and $B=B(W)$, $\E^W[ (Q_{b}A) B]=\E^W[B \E_{b}Q_{b} A]=0$ if $B$ is independent of the variables in the ${b}$-th column/row of $W$.

By Corollary~\ref{Zlemma 1}, we have
\begin{align}
\lone(\Xi)\left|\frac{F_{a_2a_1}}{G_{a_1a_1}^{(a_2)}}F_{a_1a_2}\right|\le C N^{-1/2-\fb/2}N^{3\epsilon}\,,
\end{align}
 with high probability, and a similar bound holds for the second term in~\eqref{Zlemma example}. Since 
\begin{align}
 \E^W \left|\frac{F_{a_2a_1}}{G_{a_1a_1}^{(a_2)}}F_{a_1a_2}\right|^r\le (CN)^{4Kr}\,,
\end{align}
for $r\le 10p$, for some $K$, as can be easily checked (c.f., Remark~\ref{remark about easy estimate} below), Lemma~\ref{Jensen lemma} implies that
\begin{align}
 \left|\frac{1}{N^2}\sum_{a_1\in A, a_2\in A}\E^W\, Q_{a_1}\overline{\left(\frac{1}{G_{a_1a_1}}\right)} Q_{a_2}\left(\frac{1}{G_{a_2a_2}} \right)\right|\le C(\varphi_N)^{2c_0\xi} N^{-1-\fb}N^{6\epsilon}\,,
 \end{align}
and we obtain the claim in the sense of second moment bounds. 

To deal with higher moments, we abbreviate $\mathbf{a}\equiv(a_1,\ldots,a_p)$ and write, for $p=2r$ satisfying $p\le (\log N)^{\xi-3/2}$,

\begin{align}\label{guinea pig}
 \E^W\left| \frac{1}{N}\sum_{a\in A} Q_a\left(\frac{1}{G_{aa}}\right)\right|^{2r}&=\frac{1}{N^{2r}}\sum_{a_1\in A,\ldots, a_{2r}\in A}\E^W \overline{X}_{a_1}\cdots\overline{X}_{a_r}X_{a_{r+1}}\cdots X_{a_{2r}}\nonumber\\
&=\frac{1}{N^{2r}}\sum_{a_1\in A,\ldots, a_{2r}\in A}\E^W{X}_\mathbf{a}
\end{align}
where $X_{a}\deq Q_{a}(G_{aa})^{-1}$ and $X_\mathbf{a}\deq\overline{X}_{a_1}\cdots\overline{X}_{a_r}X_{a_{r+1}}\cdots X_{a_{2r}}$.

 To cope with~\eqref{guinea pig} efficiently, we need some more notation. Let $\mathbf{L},\mathbf{U}\subset\llbracket 1,N\rrbracket$. We denote by $\caF(\mathbf L, \mathbf U)$ the set of all  ``off-diagonal resolvent fractions'' $F_{ab}^{(\T,\T')}$, $a\not=b$, $a,b\in \mathbf L$ and $\T,\T'\subset \mathbf{U}$, with $a,b\not\in\T$, $b\not\in \T'$. Further, we denote by $\caG(\mathbf  L, \mathbf U)$ the set of ``diagonal resolvent entries'' having lower indices in $\mathbf  L$ and upper indices in $ \mathbf U$; more precisely, $\caG( \mathbf L, \mathbf U)\deq\{G_{aa}^{(\T)}\,:\, a\in \mathbf L, \T\subset \mathbf U\,, a\not\in\T\}$. Finally, we denote by $\widetilde\caF( \mathbf L, \mathbf U)$ the set of monomials of the form $\C G_d^{-1}F_1F_2\ldots F_n$, with $G_d\in\caG(\mathbf  L, \mathbf U)$, $F_i\in\caF( \mathbf L, \mathbf U)$, $n\in\N_0$.

Following~\cite{EKYY4}, we call $F_{ab}^{(\T,\T')}\in\caF(\mathbf L,\mathbf U)$,  maximally expanded (in $\mathbf U$), if $\{a,b\}\cup \T\supset \mathbf L$ and $\{b\}\cup\T'\supset\ \mathbf L$. Similarly, $G_{aa}^{(\T)}$, is maximally expanded if $\{a\}\cup\T\supset \mathbf L$, and we call a monomial in $\widetilde\caF(\mathbf L,\mathbf U)$ maximally expanded if all its factors are maximally expanded. The degree of $F\in\widetilde{\caF}(\mathbf L,\mathbf U)$ is defined as the number of factors of $F_{ab}^{(\T,\T')}$  in~$F$ (with $a\not=b$).

Next, we define a recursive procedure that, given $F\in\widetilde\caF(\mathbf L,\mathbf U)$, successively adds upper indices from the set~$\mathbf U$ to~$F$, at the expense of generating terms of higher degree. This procedure is iterated until either all generated terms are maximally expanded or their degree is sufficiently large, so that they can be neglected.

Given $F\in\widetilde\caF(\mathbf L,\mathbf U)$ the recursive procedure is as follows:
\begin{itemize}
 \item [$(A)$]{\it Stopping rules:} If  the degree of $F$ is bigger equal $p+1$ or if $F$ is maximally expanded in $\mathbf U$, we stop the procedure.
\item[$(B)$]{\it Iteration:} Else, we choose an arbitrary non-maximally expanded factor $F_{ab}^{(\T,\T')}$ or ${1/G_{aa}^{(\T)}}$ of $F$ and split it using, for the former choice, either 
\begin{align}\label{iteration1}
 F_{ab}^{(\T,\T' )}&={F_{ab}^{(\T c ,\T' )}}+{F_{ac}^{(\T ,\T' )}F_{cb}^{(\T,\T')}}\,,
\end{align}
 for the smallest $c\in \mathbf U\backslash(\T ab)$, or
\begin{align}\label{iteration2}
 F_{ab}^{(\T,\T')}=F_{ab}^{(\T,\T'c)}- F_{ab}^{(\T,\T'c)}F_{bc}^{(\T,\T')}F_{cb}^{(\T,\T')}\,,
\end{align}
for the smallest $c\in \mathbf U\backslash(\T' b)$, respectively,
\begin{align}\label{iteration3}
 \frac{1}{G_{aa}^{(\T)}}=\frac{1}{G_{aa}^{(\T c)}}\left(1-F_{ac}^{(\T,\T)}F_{ca}^{(\T,\T)}\right)\,, 
\end{align}
for the smallest $c\in \mathbf U\backslash(\T a)$,  for the latter choice.

Then, we start over with $F$ being one of the generated monomials.

\end{itemize}
We remark that a similar algorithm (expanding the resolvent entries $G_{ij}^{(\T)}$ instead of $F_{ij}^{(\T,\T')}$) appeared first in~\cite{EKY}; see also~\cite{EKYY4}.
\begin{rem}
  This recursive procedure contains some arbitrariness, e.g., in the specific choice of $F_{ab}^{(\T,\T')}$ in~\eqref{iteration1},~\eqref{iteration2} or~\eqref{iteration3}. This arbitrariness does not affect the argument and we thus choose not to remove it.
\end{rem}
We now use the above algorithm to expand the right side of~\eqref{guinea pig}: Denote by $\caP_{2r}$ the set of partitions of $\llbracket 1,2r\rrbracket$ and let $\Gamma(\mathbf{a})\in\caP_{2r}$ be the partition induced by the equivalence relation $i\sim j$, if and only if $a_i=a_j$, $i,j\in\llbracket 1,2r\rrbracket$. Then we can write
\begin{align}\label{without restrictions}
   \E^W\left|\frac{1}{N}\sum_{a\in A}  Q_a\left(\frac{1}{G_{aa}}\right)\right|^{2r}=\frac{1}{N^{2r}}\sum_{\Gamma\in\caP_{2r}}\sum_{a_1\in A,\ldots,a_{2r}\in A}\lone(\Gamma=\Gamma(\mathbf{a}))\,\E^W X_{\mathbf{a}}\,.
  \end{align}
Given $\mathbf{a}=(a_i)$, denote by $\Gamma\deq\Gamma(\mathbf{a})$, the partition induced by the equivalence relation $\sim$. For a label $i\in\llbracket1,2r\rrbracket$, we denote by $[i]$ the block of $i$ in $\Gamma$. Let $S(\Gamma)\deq\{ i\,:\ |[i]|=1\}\subset\llbracket 1,2r\rrbracket$ denote the set of singletons or single labels and abbreviate by $s\deq|S(\Gamma)|$ its cardinality. We denote by $\mathbf{A}_S\equiv\mathbf{A}_{S(\Gamma)}\deq \{a_i\}_{i\in S}$, the summation indices associated with single labels. Notice that if $i$ is a single label (for some $\Gamma$), then there is exactly one $Q_{a_i}$ on the right side of~\eqref{without restrictions}. However, if $i$ is not a single label (for some $\Gamma$), $Q_{a_i}$ appears more than once on the right side of~\eqref{without restrictions}. 

We now choose $\mathbf L=\mathbf{A}_\Gamma\equiv\{a_i \}_{i}$ and $\mathbf U=\mathbf{A}_{S(\Gamma)}\equiv \mathbf{A}_S$ and apply the algorithm $(A)$-$(B)$ to $X_{a_i}$ to obtain
\begin{align}\label{Zlemma compact}
 X_{a_i}&=X_{a_i}^{(\mathbf{A}_S\backslash a_i)}+Q_{a_i}M_{a_i}(\mathbf{A}_\Gamma,\mathbf{A}_S)+Q_{a_i}R_{a_i}(\mathbf{A}_\Gamma,\mathbf{A}_S)\,,
\end{align}
where $X_{a_i}^{(\mathbf{A}_S\backslash a_i)}\deq Q_{a_i}(G_{a_ia_i}^{(\mathbf{A}_S\backslash a_i)})^{-1}$, and where $M_{a_i}(\mathbf{A}_\Gamma,\mathbf{A}_S)$ and $R_{a_i}(\mathbf{A}_\Gamma,\mathbf{A}_S)$ are sums of elements in $\widetilde{\caF}(\mathbf{A}_{\Gamma},\mathbf{A}_{S})$, such that each term in $M_{a_i}(\mathbf{A}_\Gamma,\mathbf{A}_S)$ is maximally expanded in $\mathbf{A}_S$ and each summand in $R_{a_i}(\mathbf{A}_\Gamma,\mathbf{A}_S)$ is of degree $p+1$ or higher.

\begin{rem}\label{remark about number of terms}
 The stopping rules ensure that the procedure stops after a finite number of steps when applied to a~$X_{a_i}$. More precisely, one checks that the number of terms on the ride side of~\eqref{Zlemma compact} is bounded by $(Cp)^{cp}$, for some constants $C,c$; see, e.g.,~\cite{EKYY4}, page 54.
\end{rem}

\begin{rem}\label{remark about indices}
Ignoring for the moment the upper indices, we observe that each summand in $M_{a_i}(\mathbf{A}_\Gamma,\mathbf{A}_S)$ or $R_{a_i}(\mathbf{A}_\Gamma,\mathbf{A}_S)$ is of the form
\begin{align}\label{structure of indices}
 \frac{F^{\#}_{a_ib_1}}{G_{a_ia_i}^{\#}}\cdot F^{\#}_{b_1b_2}F^{\#}_{b_2b_3}\cdots F^{\#}_{b_na_i}\,,
\end{align}
for some $(b_1,\ldots, b_n)\in\mathbf{A}_S^n$, with $1\le n\le p+1$, satisfying $ b_1\not=a_i$, $b_k\not=b_{k+1}$, ($k\in\llbracket 1,n-1\rrbracket$), $b_n\not=a_1$.  Here $\#$ stands for some appropriate $(\T,\T')$, with $|\T|\le p-2$, $|\T'|\le p-1$. Indeed, when applying the recursive procedure to $1/G_{a_ia_i}$, the first step yields
\begin{align}\label{Zlemma first step}
 \frac{1}{G_{a_ia_i}}=\frac{1}{G_{a_ia_i}^{(b_1)}}\left(1-F_{a_ib_1}F_{b_1a_i}\right)\,,
\end{align}
 for some $b_1\in\ \mathbf{A}_S\backslash\{a_i\}$, which is of the claimed form. Applying~\eqref{iteration1},~\eqref{iteration2} or~\eqref{iteration3} to a factor in~\eqref{Zlemma first step} yields again two terms of the form~\eqref{structure of indices}, etc. The claim then follows by noticing that the stopping rules ensure that each summand is of degree less equal $p+2$.

Using~\eqref{Zlemma bound 1} $n$ times and~\eqref{Zlemma bound 2} once, we obtain
\begin{align}\label{good bound}
 \lone(\Xi) \left|\frac{F^{\#}_{a_ib_1}}{G_{a_ia_i}^{\#}}\cdot F^{\#}_{b_1b_2}F^{\#}_{b_2b_3}\cdots F^{\#}_{b_na_i}\right|\le C^{n+1}(\varphi_N)^{c\xi (n+1)} N^{-1/2} N^{-n\fb/2}N^{(n+2)\epsilon}\,,
\end{align}
irrespective of the particular choice of the upper indices (recall that $\#$ stands for some $(\T,\T')$ satisfying $|\T|,|\T'|\le (\log N)^{\xi}$). For $n=0$, the analogous bound to~\eqref{good bound} reads
\begin{align}
 \lone(\Xi)\left|Q_{a_i}\left(\frac{1}{G_{a_ia_i}^{\#}}\right)\right|\le C(\varphi_N)^{c\xi} N^{-1/2}N^{2\epsilon}\,;
\end{align}
see~\eqref{initial estimate on Q_a}.

\end{rem}

\begin{rem}\label{remark about easy estimate}
Still ignoring  the upper indices, we remark that in order to apply Lemma~\ref{Jensen lemma}, we need to check that
\begin{align}\label{the easy estimate}
 \E^W\left|\frac{F^{\#}_{a_ib_1}}{G_{a_ia_i}^{\#}}\cdot F^{\#}_{b_1b_2}F^{\#}_{b_2b_3}\cdots F^{\#}_{b_na_i}\right|^r\le (CN)^{Kr(n+1)}\,,
\end{align}
for some constants $C$ and $K$, for all $r\le 10 p$. Starting from Schur's formula
\begin{align*}
 \frac{1}{G_{aa}^{(\T)}}=\lambda v_a+w_{aa}-z-\sum_{k,l}^{(\T a)}w_{ak}G^{(a\T)}_{kl}w_{la}\,,\quad\qquad (a\not\in\T)\,,
\end{align*}
and the trivial bounds $|G_{aa}^{(\T)}|\le\eta^{-1}\le N$, $\E^W |w_{ij}|^{q}\le C (\theta q)^{\theta q}N^{-q/2}$ and $ |\lambda v_i|^q\le C^q$, and the boundedness of~$\caD_\epsilon'$, one checks~\eqref{the easy estimate} by inspection.
\end{rem}
Combining Remarks~\ref{remark about number of terms} and~\ref{remark about indices} we obtain a bound on the remainder term $R_{a_i}(\mathbf{A}_\Gamma,\mathbf{A}_S)$,
\begin{align}\label{Zlemma bound on R}
 \lone(\Xi)|R_{a_i}(\mathbf{A}_\Gamma,\mathbf{A}_S)|\le (Cp)^{2p}(\varphi_N)^{c\xi p} N^{-1/2}  N^{-p\fb/2}N^{(p+2)\epsilon}\,.
 \end{align}
To condense the notation slightly, we abbreviate $M_{a_i}\equiv M_{a_i}(\mathbf{A}_\Gamma,\mathbf{A}_S)$, $R_{a_i}\equiv R_{a_i}(\mathbf{A}_\Gamma,\mathbf{A}_S)$ and set
\begin{align}
 M'_{a_i}\deq (G_{a_ia_i}^{(\mathbf{A}_S\backslash a_i)})^{-1}+M_{a_i}\,.
\end{align}
Returning to~\eqref{without restrictions} and choosing $i=1$, we can write
\begin{align}
 \E^W X_{\mathbf{a}}=\E^W Q_{a_1}{(\overline{M'}_{a_1}})\,\overline{X}_{a_2}\cdots X_{a_{2r}}+\caR_{a_1}\,,
\end{align}
where we have set $ \caR_{a_1}\deq\E^W Q_{a_1}{\left(\overline{R}_{a_1}\right)}\,\overline{X}_{a_2} \cdots X_{a_p}$. By the bound~\eqref{Zlemma bound on R} and the bounds in Remark~\ref{remark about easy estimate}, Lemma~\ref{Jensen lemma} yields
\begin{align}
 |\caR_{a_1}|\le (Cp)^{cp}(\varphi_N)^{c_0\xi p}N^{-p/2-p\fb/2+3p\epsilon}\,.
\end{align}

Before we expand in a next step $X_{a_2}$ in~\eqref{without restrictions}, we make the following observation.

\begin{rem}\label{remark about maximally expanded terms}
Let $F\in\widetilde{F}(\mathbf{A}_\Gamma,\mathbf{A}_S)$ be a maximally expanded monomial of degree $d=n+1$. From Remark~\ref{remark about indices}, we know that $F$ is of the form~\eqref{structure of indices}. Since $F$ is maximally expanded, the lower indices in each factor in~\eqref{structure of indices} also determine its upper indices. Thus, a summand $F\in\widetilde{\caF}(\mathbf{A}_\Gamma,\mathbf{A}_S)$ in~$M_{a_i}$ can be labeled by a sequence $\mathbf{b}=(b_1,\ldots,b_n)\in\mathbf{A}_S^n$, satisfying $b_1\not=a_i$, $b_k\not=b_{k+1}$, $(k\in\llbracket 1,n-1\rrbracket$), $b_n\not=a_i$. Here $d=n+1$, $n\ge1$, is the degree of~$F$. In the following we write $F\equiv F_{a_i,\mathbf{b}}\in\widetilde{\caF}(\mathbf{A}_\Gamma,\mathbf{A}_S)$. Denoting by $\mathbf{B}(a_i,n)$ the set of such labeling sequences of length $n$, we can write
\begin{align}
 M_{a_i}=\sum_{n=1}^{p}\sum_{\mathbf{b}\in\mathbf{C}(a_i,n)}F_{a_i,\mathbf{b}}\,,
\end{align}
where $\mathbf{C}(a_i,n)$ is some subset of $\mathbf{B}(a_i,n)$.
For later purposes, we note the following crude upper bound on the cardinality of $\mathbf{B}(a_i,n)$: 
\begin{align}\label{cardinality of B}
 |\mathbf{B}(a_i,n)|\le |\mathbf{A}_S|^n\le p^n\le p^p\,,\qquad (n\ge 1)\,.
\end{align}
For $n=0$, we set $F_{a_i,(0)}\deq (G_{a_ia_i}^{(\mathbf{A}_S\backslash a_i)})^{-1}$ and $ \mathbf{B}(a_i,0)\deq\{(0)\}$.
\end{rem}
Using Remark~\ref{remark about maximally expanded terms}, we can write, for $X_\mathbf{a}$ as in~\eqref{without restrictions},
\begin{align}
 \E^W X_{\mathbf{a}}=\sum_{n_1=0}^p \sum_{\mathbf{b_1}\in\mathbf{C}(a_1,n_1)}\E^W\, Q_{a_1}{(\overline{{F}}_{a_1,\mathbf{b_1}})}\,\overline{X}_{a_2}\cdots {X}_{a_{p}}+\caR_{a_1}\,,
\end{align}
where the $n_1=0$ term in the sum is understood as $\E^W \, Q_{a_1}{(\overline{F}_{a_1,(0)})}\,\overline{X}_{a_2}\cdots {X}_{a_{p}}$.

Next, we expand $X_{a_2}$ in each summand
\begin{align}
 \E^W\, Q_{a_1}{(\overline{F}_{a_1,\mathbf{b_1}})}\,{\overline{X}_{a_2}}\cdots{X}_{a_{p}}\,,
\end{align}
using the algorithm $(A)$-$(B)$, however, we stop expanding a term (generated from $X_{a_2}$), if its degree is bigger than $p+1-n_1$, or if it is maximally expanded. Note that we do not expand the rest term~$\caR_{a_1}$ any further. With this modified stopping rule, we arrive at
\begin{align}\label{equus}
 \E^W X_{\mathbf{a}}= \sum_{n_1=0}^p\sum_{n_2=0}^{p} \sum_{\mathbf{b_1}\in\mathbf{C}(a_1,n_1)}\sum_{\mathbf{b_2}\in\mathbf{C}(a_2,n_2)}\lone(n_1+n_2\le p)&
\E^W\, Q_{a_1}{(\overline{F}_{a_1,\mathbf{b_1}})}Q_{a_2}{(\overline{F}_{a_2,\mathbf{b}_2})}\cdots {X}_{a_{p}}\nonumber\\+\caR_{a_1}+\caR_{a_2}\,,
\end{align}
for some sets of labeling sequences $\mathbf{C}(a_1,n_1)$ and $\mathbf{C}(a_2,n_2)$. It is easy to check that the rest term $\caR_{a_2}$ satisfies the same bound (with possibly slightly larger constants) as $R_{a_1}$. This is checked in the same way as before. To estimate the number of summands in~$R_{a_2}$, we note that we apply the algorithm $|\mathbf{C}(a_1,n_1)|$ times and each application yields no more than $(Cp)^{cp}$ terms; see Remark~\ref{remark about number of terms}. Thus the number of summands in~$\caR_{a_2}$ is bounded by $(Cp)^{c'p}$, $c'>c$.

We continue expanding the remaining $X_{a_k}$, $k\ge 3$, in~\eqref{equus} using the algorithm $(A)$-$(B)$, but while expanding~$X_{a_k}$, we stop the expansion as soon as the degree of a generated term exceeds $p+1-\sum_{i=1}^{k-1} n_i$. This leads to, 
\begin{align}\label{almost there}
 \E^W X_{\mathbf{a}}= \sum_{n_1,\ldots,n_p=0}^p\,\,&\sum_{\mathbf{b_1}\in\mathbf{C}(a_1,n_1)}\cdots\sum_{\mathbf{b_p}\in\mathbf{C}(a_p,n_p)}\lone(\sum_{i=1}^pn_i\le p)\,\E^WY(\mathbf{a},\mathbf{b_1},\dots,\mathbf{b_p})+\sum_{i=1}^p\caR_{a_i}\,,
\end{align}
where $(\mathbf{C}(a_i,n_i))$ are subsets of $(\mathbf{B}(a_i,n_i))$, and where we have abbreviated
\begin{align}\label{almost there 2}
 Y(\mathbf{a},\mathbf{b_1},\dots,\mathbf{b_p})\deq\prod_{i=1}^r Q_{a_i}{(\overline{F}_{a_i,\mathbf{b_i}})}\,\prod_{i=r+1}^{2r} Q_{a_i}(F_{a_i,\mathbf{b}_i})\,.
\end{align}
The remainder terms $(\caR_{a_i})$ clearly satisfy
\begin{align}
 |\caR_{a_i}|\le   (Cp)^{cp}(\varphi_N)^{c_0\xi p}N^{-p/2-p\fb/2}N^{3p}\,,
\end{align}
for all $i\in\llbracket 1,p\rrbracket$. It thus suffices to bound the first term on the right side of~\eqref{almost there}.

 Following~\cite{EKYY4}, we pick a term $Y\equiv Y(\mathbf{a},\mathbf{b_1},\dots,\mathbf{b_p}) $ of the form~\eqref{almost there 2} that has a non-vanishing expectation, $\E^W Y\not=0$.  Considering a single label $i\in S\equiv S(\Gamma)$, we know that there exists a label $j\in S\backslash\{i\}$, such that the monomial $F_{a_j,\mathbf{b_j}}\in\widetilde{\caF}(\mathbf{A}_\Gamma,\mathbf{A}_S)$ in~\eqref{almost there 2} contains a factor $F\in \caF(\mathbf{A}_\Gamma,\mathbf{A}_{S})$, having $a_i$ as a lower index (otherwise the expectation of $Y$ has to vanish due to the presence of the $Q_{a_i}$). It follows from Remark~\ref{remark about indices} that $F_{a_j,\mathbf{b_j}}$ is of the form~\eqref{structure of indices}, with $n\ge 1$ and $b_k=a_i$, for some $k\in \llbracket 1,n\rrbracket$. Note that we use here that all indices in $\mathbf{A}_S$ are distinct.

We write $j=\mathfrak{l}(i)$, if a label $j$ is linked to a label $i$ in the sense of the previous paragraph. Denoting by \mbox{$l_j\deq|\mathfrak{l}^{(-1)}(\{j\})|$}, the number of times the label $j$ has been chosen to be linked in the above sense to some other label, we obtain
\begin{align}
\lone(\Xi)|F_{a_j,\mathbf{b}_j}|\le C^p(\varphi_N)^{c\xi p}N^{-1/2}N^{-l_j\fb/2}N^{2p\epsilon+l_j\epsilon}\,,
\end{align}
as follows from~\eqref{good bound}.

Finally, using $\sum_{j\in S}l_j\ge |S|=s$, we get
\begin{align}
\lone(\Xi)| Y|\le  C^p(\varphi_N)^{c\xi p}  N^{-p/2}N^{-s\fb/2}N^{2p\epsilon+s\epsilon}\,,
\end{align}
for $Y$ as in~\eqref{almost there 2}, with $\E^W Y\not=0$. Lemma~\ref{Jensen lemma} thus gives
\begin{align}
 |\E^WY|\le C^p(\varphi_N)^{c_0\xi p} N^{-p/2-s\fb/2}N^{2p\epsilon+s\epsilon}\,.
\end{align}
To bound the right side of~\eqref{almost there}, it remains to bound the number of summands in the first term. Using~\eqref{cardinality of B}, we obtain
\begin{align}
 \sum_{n_1,\ldots, n_p=0}^p\,\, \sum_{\mathbf{b_1}\in\mathbf{C}(a_1,n_1)}\cdots\sum_{\mathbf{b_p}\in\mathbf{C}(a_p,n_p)}\lone(\sum_{i=1}^pn_i\le p)&\le \sum_{n_1,\ldots, n_p=0}^p\lone(\sum_{i=1}^pn_i\le p) \prod_{i=1}^p p^{n_i}\nonumber\\
&\le \sum_{n_1,\ldots, n_p=0}^p p^{p}
\le (Cp)^{2p}\,,
\end{align}
and we can bound~\eqref{almost there} by
\begin{align}\label{bound for later on in corollary}
 |\E^WX_{\mathbf{a}}|\le (Cp)^{cp}(\varphi_N)^{c_0\xi p} N^{-p/2-s\fb/2+3p\epsilon+s\epsilon}\,.
\end{align}

We now return to~\eqref{without restrictions}. We perform the summation by first fixing a partition $\Gamma\in\caP_{2r}$. Then we observe that
\begin{align}\label{sum over non label}
 \frac{1}{N^{2r}}\sum_{\mathbf{a}}\lone(\Gamma=\Gamma(\mathbf{a}))\le \left(\frac{1}{N}\right)^{2r-|\Gamma|}\le \left(\frac{1}{\sqrt{N}}\right)^{2r-s}\,,
\end{align}
since any block in the partition $\Gamma$ that is not associated to a single label consists of at least two elements. Thus $|\Gamma|\le (2r+s)/2=r+s/2$. Using $N^{-1/2}\ll N^{-\fb}$ we find
\begin{align}
 \E^W\left|\frac{1}{N}\sum_{a\in A}Q_{a}\left(\frac{1}{G_{aa}} \right)\right|^{2r}\le (Cp)^{cp}(\varphi_N)^{c\xi p}\sum_{\Gamma\in\caP_{2r}}N^{-p/2-p\fb/2}N^{3p\epsilon}\,.
\end{align}
Recalling that the number of partitions of $p$ elements is bounded by $(Cp)^{2p}$, we thus get
\begin{align}
 \E^W\left|\frac{1}{N}\sum_{a\in A}Q_{a}\left(\frac{1}{G_{aa}(z)} \right)\right|^{2r}\le (Cp)^{cp}(\varphi_N)^{c_0\xi p}N^{-p/2-p\fb/2}N^{3p\epsilon}\,,
\end{align}
for some constants $C$ and $c$. Finally, we note that the constants can be chosen uniformly in $z\in\caD_\epsilon'$, since all estimates used are uniform in $z$.

\end{proof}

Next, we prove Lemma~\ref{lem:step 4}.
\begin{proof}[Proof of Lemma~\ref{lem:step 4}]
Recalling the remark after Definition~\ref{definition of xi}, we know that the event $\Xi$ has $(\xi,\nu)$-high probability. Choosing $p$ as the largest integer smaller than $\nu (\log N)^{\xi-2}$, Markov's inequality and the moment bound~\eqref{estimate in Zlemma} yield, for some constant $c$,
\begin{align}\
 \frac{1}{N}\left|\sum_{a=n_0}^N Z_{a}(z)\right|\le (\varphi_N)^{c\xi}N^{-1/2-\fb/2+4\epsilon}\,,
\end{align}
with $(\xi-2,\nu)$-high probability on $\Omega_V$, uniformly for $z\in\caD_\epsilon'$. This finishes the proof of~\eqref{statement Zlemma 1}. 

To prove~\eqref{statement Zlemma 2}, it suffices to replace $(G_{aa})$ in~\eqref{donkey 2}, by $(G_{aa}^{(k)}\,:\,a\not= k)$. Using~\eqref{Zlemma expand basic} and the bounds~\eqref{Zlemma bound 1} and~\eqref{Zlemma bound 2}, the claim~\eqref{statement Zlemma 2}, follows from~\eqref{statement Zlemma 1}. We leave the details aside.
\end{proof}
Next, we prove Corollary~\ref{cor:step 4}. Define, for $a\in A$,
\begin{align}
 g_a(z)\deq\frac{1}{\lambda v_a-z-\widehat{m}_{fc}{(z)}}\,,\qquad\quad( z\in\C^{+})\,.
\end{align}
Note that for $z\in\caD_{\epsilon}'$, we have $|g_a(z)|\le C N^{1/(\b+1)}N^{\epsilon}$ on $\Omega_V$.
\begin{lem}\label{corollary to Zlemma}
Let $A=\llbracket  n_0,N\rrbracket$. Recall the definition of the domain $\caD_{\epsilon}'$ in~\eqref{a index assumption}. Let $\Xi$ denote the event in Definition~\ref{definition of xi} and assume it has $(\xi,\nu)$-high probability. Then there exist constants $C$, $c$, $c_0$, such that for $p=2r$, $r\in\N$, $p\le (\log N)^{\xi-3/2}$, we have
\begin{align}\label{estimate in Zcorollary}
 \E^W\left|\frac{1}{N}\sum_{a\in A}(g_{a}(z))^2Q_{a}\left(\frac{1}{G_{aa}(z)}\right) \right|^{p}\le (Cp)^{cp}(\varphi_N)^{c_0\xi p} N^{-p/2-p\fb/2}N^{3p\epsilon}\,,
\end{align}
on $\Omega_V$, for all $z\in\caD_\epsilon'$.
\end{lem}

\begin{proof}
First, we note that $g_{a}Q_{b}=Q_{b} g_{a}$, $a,b\in\llbracket 1,N\rrbracket$, since the $(g_a)$ are independent of the random variables $(w_{ij})$. Similar to~\eqref{guinea pig}, we are led to consider
\begin{align}\label{guinea pig 2}
\E^W\left| \frac{1}{N}\sum_{a\in A}g_a^2 Q_a\left(\frac{1}{G_{aa}}\right)\right|^{2r}&
=\frac{1}{N^{2r}}\sum_{a_1\in A,\ldots, a_{2r}\in A}g_{\mathbf{a}}^2\,\E^W{X}_\mathbf{a}\,,
\end{align}
where we have set $g_\mathbf{a}\deq\overline{g}_{a_1}\cdots\overline{g}_{a_r}g_{a_{r+1}}\cdots g_{a_{2r}}$.

Following the lines of the proof of Lemma~\ref{the Zlemma}, we write~\eqref{guinea pig 2} as
\begin{align}\label{without restrictions 2}
   \E^W\left|\frac{1}{N}\sum_{a\in A} g_a^2 Q_a\left(\frac{1}{G_{aa}}\right)\right|^{2r}=\frac{1}{N^{2r}}\sum_{\Gamma\in\caP_{2r}}\sum_{a_1\in A,\ldots,a_{2r}\in A}\lone(\Gamma=\Gamma(\mathbf{a}))g^2_{\mathbf{a}}\,\E^W X_{\mathbf{a}}\,.
  \end{align} 
As in the proof of Lemma~\ref{the Zlemma}, we fix $\mathbf{a}$ and denote by $\Gamma\deq\Gamma(\mathbf{a})$, the partition induced by the equivalence relation~$\sim$. Since $\E^W X_{\mathbf{a}}$ with $\mathbf{a}\equiv\mathbf{a}_\Gamma$ has already been bounded in the proof of Lemma~\ref{the Zlemma}, with a bound that only depends on $s\equiv |S(\Gamma)|$ (see~\eqref{bound for later on in corollary}), it suffices to control
\begin{align}
 \frac{1}{N^{2r}}\sum_{a_1\in A,\ldots,a_{2r}\in A}\lone(\Gamma=\Gamma(\mathbf{a}))g^2_{\mathbf{a}}\,.
\end{align}
Recall that
\begin{align}\label{Zlemma sumrule}
 \frac{1}{N}\sum_{i=1}^N\frac{1}{|\lambda v_i-z-\widehat{m}_{fc}(z)|^2}=\widehat{R}_2(z)<1\,,\qquad \quad( z\in\caD_{\epsilon}') \,;
\end{align}
see, e.g.,~\eqref{R2 hat less than 1}. Applying~\eqref{Zlemma sumrule} $|\Gamma|$ times, we obtain
\begin{align}
 \frac{1}{N^{2r}}\sum_{a_1\in A,\ldots,a_{2r}\in A}\lone(\Gamma=\Gamma(\mathbf{a}))|g_{\mathbf{a}}|^2\le C^{2r}\left(\frac{N^{2/(\b+1)}N^{2\epsilon}}{N}\right)^{2r-|\Gamma|}\,.
\end{align}
Here $|\Gamma|$ denotes the number of blocks of the partition $\Gamma$, and $s$ denotes the number of single labels in~$\mathbf{a}$. Since every block of $\Gamma$ that is not associated with a single label consists of at least two elements, we have $|\Gamma|\le (2r+s)/2=r+s/2$, thus $2r-|\Gamma|\ge r-s/2$, and we obtain,
\begin{align}
 \frac{1}{N^{2r}}\sum_{a_1\in A,\ldots,a_{2r}\in A}\lone(\Gamma=\Gamma(\mathbf{a}))|g_{\mathbf{a}}|^2&\le C^{2r}\left(\frac{N^{2/(\b+1)}N^{2\epsilon}}{N}\right)^{r-s/2}\nonumber\\ &=C^{2r}\left(\frac{N^{1/(\b+1)}N^{\epsilon}}{\sqrt{N}} \right)^{2r-s}\nonumber\\
 &\le C^{2r}\left(N^{-\fb}\right)^{2r-s}N^{(2r-s)\epsilon}\,. 
\end{align}
The proof of Lemma~\ref{corollary to Zlemma} is now completed as the proof of Lemma~\ref{the Zlemma}.
\end{proof}
\begin{proof}[Proof of Corollary~\ref{cor:step 4}]
 Recalling~\eqref{donkey 2}, Corollary~\ref{cor:step 4} can be proven in the same way as Lemma~\ref{lem:step 4} above.
\end{proof}

\section{Proof of Theorem~\ref{thm:local}} \label{sec:local}
In this section, we prove Theorem~\ref{thm:local}. The proof of~\eqref{mass_k} is based on an analysis of $(G_{jj}(z))$, for~$z$ close to the eigenvalues $\mu_k$, $k\in\llbracket 1, n_0-1\rrbracket$, using the Helffer-Sj\"ostrand formula~\eqref{helffer sjoestrand} below. The Helffer-Sj\"ostrand calculus has been applied in~\cite{ERSY} to $m(z)$, the averaged Green function, respectively to the empirical eigenvalue counting function, to obtain rigidity estimates on the eigenvalue locations. In Subsection~\ref{subsection mass_k}, we apply the Helffer-Sj\"ostrand formula to $(G_{jj}(z))$, respectively, to a weighted empirical eigenvalue counting function; see~\eqref{weighed counting measure}. 

In Subsection~\ref{subsection mass_j}, we prove~\eqref{mass_j} following the argument given in~\cite{EYY1} for generalized Wigner matrices.

\subsection{Preliminaries}
To start with, we claim that, for $z$ close to the spectral edge, $z+\widehat m_{fc}(z)$ is well approximated by a linear function.
\begin{lem}\label{lem: linear approx}
On $\Omega_V$, we have for all $z\in\caD_{\epsilon}'$ that
\begin{align}\label{eq: linear approx}
 z+\widehat{m}_{fc}(z)=\lambda-\frac{\lambda^2}{\lambda_+^2-\lambda^2}(z-L_+)+\caO\left((\varphi_N)^{\xi}(\kappa+\eta)^{\min\{\b,2\}}\right)+\caO(N^{-1/2+2\epsilon})\,.
\end{align}
\end{lem}
\begin{proof}
The lemma follows readily by combining Lemma~\ref{mfc estimate} and Lemma~\ref{hat bound}.
\end{proof}	
\begin{rem}
Combining Lemma~\ref{lem: linear approx} and Proposition~\ref{prop:main}, we obtain, for $j\le n_0-1$, 
\begin{align}\label{equation in remark}
 \lambda v_j-z-\widehat{m}_{fc}(z)=\frac{\lambda^2}{\lambda^2-\lambda_+^2}(\mu_j-z)&+\caO\left((\varphi_N)^{2\xi}N^{-2/(\b+1)}\right)\nonumber\\
&\qquad+\caO\left(N^{-1/2+3\epsilon}\right)+\caO\left((\varphi_N)^{\xi}(\kappa+\eta)^{\min\{\b,2\}}\right)\,,
\end{align}
with $(\xi,\nu)$-high probability, for all  $z\in\caD_{\epsilon}'$.

\end{rem}
To state the next lemma, it is convenient to abbreviate
\begin{align}\label{definition g0}
g_j^\circ(z)\deq\frac{\lambda^2-\lambda_+^2}{\lambda^2}\frac{1}{\mu_j-z}\,, \qquad
 G^{\Delta}_{jj}(z)\deq G_{jj}(z)-g_j^\circ(z)\,,\qquad\quad( z\in\C^+)\,.
\end{align}
We have the following estimate on $G^{\Delta}_{jj}(z)$.
\begin{lem}
There are constants $C$ and $c$, such that, for all $z\in\caD_{\epsilon}'$, and all $j\le n_0-1$,
\begin{align}\label{bound G11}
 |G^{\Delta}_{jj}(z)|\le C(\varphi_N)^{2\xi} |g^\circ_j(z)|\left(\frac{1}{N^{2/(\b+1)}\eta}+\frac{N^{3\epsilon}}{N^{1/2}\eta}+\frac{(\kappa+\eta)^{\min\{\b,2\}}}{\eta}\right)\,,
\end{align}
with $(\xi,\nu)$-high probability on $\Omega_V$.

\end{lem}

\begin{proof}
For $z\in\caD_{\epsilon}'$, we have $(G_{jj}(z))^{-1}={\lambda v_j-z-\widehat{m}_{fc}(z)+\caO(N^{-1/2+2\epsilon})}$, with high probability on~$\Omega_V$; see Lemma~\ref{prop:step 2_4}.
We thus obtain from~\eqref{equation in remark} that
\begin{align}
 G_{jj}(z)&=g_j^\circ(z)+g_j^\circ(z)G_{jj}(z)\caO\left((\varphi_N)^{2\xi}\frac{1}{N^{2/(\b+1)}}+\frac{N^{3\epsilon}}{N^{1/2}}+(\varphi_N)^{\xi}(\kappa+\eta)^{\min\{\b,2\}}\right)\,,
\end{align}
with high probability on $\Omega_V$. Applying the trivial bound $|G_{jj}(z)|\le \eta^{-1}$, Inequality~\eqref{bound G11} follows.
\end{proof}

\subsection{Proof of Theorem~\ref{thm:local}: Inequality~\eqref{mass_k}}\label{subsection mass_k}
For $j\in\llbracket 1,N\rrbracket$, we define a weighted empirical eigenvalue counting measure, $\rho_j$, by
\begin{align}\label{weighed counting measure}
 \rho_j\deq \sum_{k=1}^N |u_{k}(j)|^2\delta_{\mu_{k}}\,,
\end{align}
where $(\mu_{k})$ are the eigenvalues of $H$ and $(u_{k}(j))$ the components of the associated eigenvectors. For $k\in\llbracket 1,n_0-1\rrbracket$, let $f_k(x)\ge0$ be a smooth test function satisfying
\begin{align}\label{definition of f_k}
f_k(x)=\begin{cases} 1\,, & \textrm{if }|x-\mu_k|\le N^{-1/(\b+1)-\epsilon'}\\
0\,, & \textrm{if }|x-\mu_k|\ge 2 N^{-1/(\b+1)-\epsilon'}
     \end{cases}\,,\quad |f_k'(x)|\le C N^{1/(\mathrm{b}+1)}N^{\epsilon'}\,,\quad|f_k''(x)|\le C N^{2/(\mathrm{b}+1)}N^{2\epsilon'}\,,
\end{align}
for some $\epsilon'>\epsilon$. By Proposition~\ref{prop:main}, we know that the eigenvalues $(\mu_k)$ satisfy
\begin{align}
 \min_{1 \le k\le n_0-1}( \mu_{k+1}-\mu_k)\ge C N^{-1/(\b+1)-\epsilon}\,,
\end{align}
for some constant $C$, with high probability on~$\Omega_V$. Since we chose $\epsilon'>\epsilon$, we conclude that the following formula holds with high probability on $\Omega_V$, for all $k\le n_0-1$,
\begin{align}\label{eigenvector components formula}
 |u_{k}(j)|^2=\int f_k(w)\dd \rho_j(w)\,.
\end{align}

Using the Helffer-Sj\"ostrand formula, we may represent $f_k(w)$ as
\begin{align}\label{helffer sjoestrand} 
 f_k(w)=\frac{1}{2\pi} \int_{\R^2}\dd x\,\dd y\,\frac{\widetilde{f_k}(x+\ii y)}{w-x-\ii y}\,,
\end{align}
where 
\begin{align}\label{widetilde f definition}
\widetilde{f_k}(x+\ii y)\deq\ii y f_k''(x)\chi(y)+\ii(f_k(x)+\ii y f_k'(x))\chi'(y)\,,\qquad\quad( x,y\in\R)\,,
\end{align}
with $\chi\ge 0$ a smooth test function satisfying

\begin{align}\label{5.42}
 \chi(y)=\begin{cases} 1\,, & \textrm{ if }y\in [-\caE,\caE]\\0\,, &\textrm{ if } y\in[-2\caE,2\caE]^{c}
\end{cases}\,,\qquad\qquad |\chi'(y)|\le\frac{C}{\caE}\,,
\end{align}
where we have set $\caE\deq N^{-1/(\mathrm{b}+1)}$. 

Combining~\eqref{eigenvector components formula} with~\eqref{helffer sjoestrand}, we obtain the following representation for $|u_{k}(j)|$, with $k\in\llbracket1,n_0-1\rrbracket$, $j\in\llbracket 1,N\rrbracket$,
\begin{align}\label{representation for uki}
|u_{k}(j)|^2=\frac{1}{2\pi}\int_{\R^2}\dd x\,\dd y\,\widetilde{f_k}(x+\ii y) g_j^\circ(x+\ii y)+\frac{1}{2\pi}\int_{\R^2}\dd x\,\dd y\,\widetilde{f_k}(x+\ii y)G^{\Delta}_{jj}(x+\ii y)\,,
\end{align}
which holds with high probability on $\Omega_V$. The first term on the right side of~\eqref{representation for uki} can be computed explicitly as
\begin{align}
 \frac{1}{2\pi}\int_{\R^2}\dd x\,\dd y\,\widetilde{f_k}(x+\ii y) g_j^\circ(x+\ii y)&=\frac{\lambda^2-\lambda_+^2}{2\pi\lambda^2} \int_{\R^2}\frac{\widetilde{f_k}(x+\ii y)}{\mu_j-(x+\ii y)}
=\frac{\lambda^2-\lambda_+^2}{\lambda^2}f_k(\mu_j)\,.
\end{align}
Recalling~\eqref{definition of f_k}, we conclude by Proposition~\ref{prop:main} that, for $j,k\le n_0-1$, with high probability on $\Omega_V$, 
\begin{align}\label{almost  finished proof of mass_k}
 \frac{1}{2\pi}\int_{\R^2}\dd x\,\dd y\,\widetilde{f_k}(x+\ii y) g_j^\circ(x+\ii y)=\frac{\lambda^2-\lambda_+^2}{\lambda^2}\delta_{jk}\,.
\end{align}

To complete the proof of~\eqref{mass_k}, it hence suffices to show that the second term on the right side of~\eqref{representation for uki} is negligible compared to~\eqref{almost  finished proof of mass_k}, provided $\epsilon'>\epsilon$. For concreteness we set $\epsilon'=2\epsilon$ in the following.

 \begin{lem}\label{lem: HS}
Let $\widetilde{f}_k$ be defined as in~\eqref{widetilde f definition} and set $\epsilon'=2\epsilon$. Then there are constants $C$ and $c$ such that, for $j,k\le n_0-1$,
\begin{align}\label{bound on HSterm}
\left|\int_{\R^2}\dd x\,\dd y\,\widetilde{f_k}(x+\ii y) G^{\Delta}_{jj}(x+\ii y)\right|\le C(\varphi_N)^{c\xi}\left(N^{  -1/(\b+1)+3\epsilon }+N^{-\fb+5\epsilon} \right)\,,
\end{align}
with $(\xi,\nu$)-high probability on $\Omega_V$.
\end{lem}

\begin{proof}
 We set $\eta\deq\frac{1}{\sqrt{N}}$. Using $G^{\Delta}_{jj}(\overline{z}) =\overline{G^{\Delta}_{jj}(z)}$, $(z\in\C^+)$, we obtain,
\begin{align}\label{HS to be bounded}
 \left|\int_{\R^2}\dd x\,\dd y\,\widetilde{f_k}(x+\ii y) G^{\Delta}_{jj}(x+\ii y)\right|&\le C\int \dd x\int_0^{\infty}  {\dd y}|f_k(x)|\,|\chi'(y)|\, |G^{\Delta}_{jj}(x+\ii y)|\nonumber\\
&\quad+ C\int \dd x\int_0^{\infty}  {\dd y}\,y\, |f_k'(x)|\,|\chi'(y)|\, |G^{\Delta}_{jj}(x+\ii y)|\nonumber\\
&\quad+C\left| \int \dd x\int_0^{\eta}\dd y f_k''(x)\chi(y) y\,\im G^{\Delta}_{jj}(x+\ii y) \right|\nonumber\\ &\quad+C\left|\int \dd x\int_{\eta}^{\infty}\dd y f_k''(x)\chi(y) y\,\im G^{\Delta}_{jj}(x+\ii y) \right|\,.
\end{align}
The first term on the right side of~\eqref{HS to be bounded} can be bounded by using~\eqref{bound G11}: Bounding in~\eqref{bound G11} $\kappa\equiv \kappa_x$  by $\kappa\le N^{-1/(\b+1)+\epsilon/2}$ and $g_j^\circ$ by $|g_{j}^\circ(x+\ii y)|\le Cy^{-1}$, we obtain
\begin{align}\label{bound HS1}
\int \dd x\int_0^{\infty}  \dd y \,&|f_k(x)||\chi'(y)| |G^{\Delta}_{jj}(x+\ii y)|
\le \frac{C}{\caE}\int \dd x\,|f_k(x)|\,\int_{\caE}^{2\caE}\dd y\, |G^{\Delta}_{jj}(x+\ii y)|\nonumber\\
&\le C(\varphi_N)^{2\xi}\frac{N^{-1/(\b+1)-\epsilon'}}{\caE^2}\left(N^{-2/(\b+1)}+N^{-1/2+3\epsilon}+(N^{-1/(\b+1)+\epsilon/2})^{\min\{\b,2\}}\right)\nonumber\\
&\le C(\varphi_N)^{2\xi}\left(N^{-1/(\b+1)-\epsilon'+\epsilon}+N^{-\fb-\epsilon'+3\epsilon}\right)\,,
\end{align}
with high probability on $\Omega_V$.

Similarly, we bound the second term on the right side of~\eqref{HS to be bounded}:
\begin{align}
  \int \dd x\int_0^{\infty}  \dd y \,&y\, |f_k'(x)|\,|\chi'(y)| |G^{\Delta}_{jj}(x+\ii y)|\\
 &\le\frac{C}{\caE}\,\int\dd x\,|f_k'(x)|\,\int_{\caE}^{2\caE}\dd y\,{y}|G^{\Delta}_{jj}(x+\ii y)|\nonumber\\
 &\le\frac{C(\varphi_N)^{2\xi}}{\caE}|\log\caE|\left(N^{-2/(\b+1)}+N^{-1/2+3\epsilon}+(N^{-1/(\b+1)+\epsilon/2})^{\min\{\b,2\}}\right)\nonumber\\
 &\le C(\varphi_N)^{c\xi}\left(N^{-1/(\b+1)+\epsilon}+N^{-\fb+3\epsilon}\right)\,.
\end{align}
To bound the third term on the right side of~\eqref{HS to be bounded}, we use that
 $|y\,	\im G^{\Delta}_{jj}(x+\ii y)|\le C$, for all $y>0$, and we obtain
\begin{align}\label{bound HS2}
 \left| \int \dd x\int_0^{\eta}\dd y f_k''(x)\chi(y) y\,\im G^{\Delta}_{jj}(x+\ii y) \right|&\le C\eta \int\dd x\, |f_k''(x)|
\le C N^{-\fb+\epsilon'}\,.
\end{align}
To bound the fourth term in~\eqref{HS to be bounded}, we integrate by parts, first in $x$ then in $y$  to find the bound
\begin{align}\label{HS the forth term}
 C\left|\int\dd xf_k'(x)\eta\, \re G^{\Delta}_{jj}(x+\ii \eta)  \right|&+C\left|\int\dd x\int _{\eta}^{\infty}\dd y f_k'(x)\chi'(y) { y}\, \re G^{\Delta}_{jj} (x+\ii y)\right|\nonumber\\
&\quad\quad + C\left|\int\dd x \int_{\eta}^{\infty}\dd y f_k'(x)\chi(y)\,\re G^{\Delta}_{jj}(x+\ii y)\right|\,.
\end{align}
The first term in~\eqref{HS the forth term} can be bounded using~\eqref{bound G11}: Since $f_k'(x)=0$, if $|x-\mu_k|\le N^{-1/(\b+1)-\epsilon'}$, we can bound $|f'_k(x) g_j^\circ(x+\ii y)|\le C N^{2/(\b+1)+2\epsilon'}$, and we obtain
\begin{align}\label{HS neu 3}
& \left|\int\dd xf_k'(x)\eta\, \re G^{\Delta}_{jj}(x+\ii \eta)  \right| \nonumber \\
& \quad \le C(\varphi_N)^{2\xi}{N^{1/(\b+1)+\epsilon'}}\eta\left(\frac{1}{N^{2/(\b+1)}\eta}+\frac{N^{3\epsilon}}{N^{1/2}\eta}+\frac{(N^{-1/(\b+1)+\epsilon/2})^{\min\{\b,2\}})}{\eta}\right)\nonumber\\
& \quad \le C(\varphi_N)^{2\xi}\left(N^{-1/(\b+1)+\epsilon'+\epsilon}+N^{-\fb+\epsilon'+3\epsilon}\right)\,,
\end{align}
with high probability on $\Omega_V$.

Similarly for the second term in~\eqref{HS the forth term}, using~\eqref{bound G11} we obtain
\begin{align}\label{bound HS3}
\bigg|\int\dd x&\int _{\eta}^{\infty}\dd y f_k'(x)\chi'(y) { y}\, \re G^{\Delta}_{jj} (x+\ii y)\bigg|
\le C(\varphi_N)^{2\xi}\left(N^{-1/(\b+1)+\epsilon'+\epsilon}+N^{-\fb+\epsilon'+3\epsilon}\right)\,.
\end{align}
The third term in~\eqref{HS the forth term} can be bounded using~\eqref{bound G11} as
\begin{align}
 \bigg|\int\dd x& \int_{\eta}^{\infty}\dd y f_k'(x)\chi(y)\re G^{\Delta}_{jj}(x+\ii y)\bigg| \nonumber \\
 &\le C(\varphi_N)^{2\xi}{N^{1/(\b+1)+\epsilon'}} \int_{\eta}^{2\caE}\dd y\left( \frac{1}{N^{2/(\b+1)}y}+\frac{N^{3\epsilon}}{N^{1/2}y}+\frac{(N^{-1/(\b+1)+\epsilon/2})^{\min\{\b,2\}})}{y}\right)\nonumber\\
&\le C|\log\eta|(\varphi_N)^{2\xi}\left( N^{-1/(\b+1)+\epsilon'+\epsilon}+N^{-\fb+\epsilon'+3\epsilon}\right)\,.\label{HS neu 5}
\end{align}
Thus combining~\eqref{bound HS1},~\eqref{bound HS2},~\eqref{HS neu 3},~\eqref{bound HS3} and~\eqref{HS neu 5}, we obtain, upon choosing $\epsilon'=2\epsilon$,
\begin{align}
 \left|\int\dd x\,\dd y\, \widetilde{f_k}(x+\ii y)G^{\Delta}_{ii}(x+\ii y)\right|&\le C(\varphi_N)^{c\xi}\left( N^{-1/(\b+1)+3\epsilon}+N^{  -\fb+5\epsilon  }\right)\,,
\end{align}
with high probability on $\Omega_V$.
\end{proof}
\begin{proof}[Proof of Theorem~\ref{thm:main}: Equation~\eqref{mass_k}]
 Combining \eqref{representation for uki}, \eqref{almost  finished proof of mass_k} and~\eqref{bound on HSterm}, we obtain, for $j,k\le n_0-1$, 
\begin{align}
 |u_k(j)|^2=\frac{\lambda^2-\lambda_+^2}{\lambda^2}\delta_{jk}+\caO\left((\varphi_N)^{c\xi}N^{-1/(\b+1)+3\epsilon}+(\varphi_N)^{c\xi}N^{  -\fb+5\epsilon  } \right)\,,
\end{align}
with high probability on $\Omega_V$. Choosing $k=j$, Equation~\eqref{mass_k} follows.
\end{proof}

\subsection{Proof of Theorem~\ref{thm:local}: Inequality~\eqref{mass_j}}\label{subsection mass_j}
In this subsection we prove the second part of Theorem~\ref{thm:local}. To prove~\eqref{mass_j}, we follow the traditional path of~\cite{ESY1,ESY2,ESY3}. Presumably, the same result can be obtained  by a more detailed analysis of $G_{jj}^{\Delta}$ than the one carried out in the previous subsection.

Recall that we have set $\eta_0\deq N^{-1/2-\epsilon}$.
\begin{lem}\label{tapiridae}
There is a constant $C$, such that for $k\in\llbracket 1,n_0-1\rrbracket$, $j\in\llbracket 1,N\rrbracket$, ($k\not=j$),
\begin{align}\label{eq tapiridae}
 \im G_{jj}(\mu_k+\ii\eta_0)\le\frac{C}{(\lambda v_j-\lambda v_{k})^2}N^{-1/2+3\epsilon}\,,\end{align}
with $(\xi,\nu)$-high probability, on $\Omega_V$.
\end{lem}
\begin{proof}
For $z\in\caD_{\epsilon}'$, we have
\begin{align}\label{intermediate step G_jj}
  (G_{jj}(z))^{-1}={\lambda v_j-z-\widehat{m}_{fc}(z)+\caO(N^{-1/2+2\epsilon})}\,,
\end{align}
with high probability on $\Omega_V$; see Lemma~\ref{prop:step 2_4}. Next, recall from Proposition~\ref{prop:mu_k}, we have with high probability on~$\Omega_V$ that
\begin{equation*}
\mu_k + \re \widehat m_{fc} (\mu_k + \ii \eta_0) = \lambda v_k + \caO (N^{-1/2 + 3\epsilon})\,.
\end{equation*}
Also recall the high probability estimate $\im m(z) \leq \frac{N^{2\epsilon}}{\sqrt N}$, for $z\in\caD_{\epsilon}'$, on $\Omega_V$; see Lemma~\ref{lem:step 2_1}. Thus, choosing $z=\mu_k+\ii \eta_0$, with $k\in\llbracket 1,n_0-1\rrbracket$ and $k\not=j$, in~\eqref{intermediate step G_jj} we obtain
\begin{align*}
 ({G_{jj}(\mu_k+\ii\eta_0)})^{-1}&={\lambda v_j-\mu_k-\ii\eta_0-\widehat{m}_{fc}(\mu_k+\ii\eta_o)+\caO(N^{-1/2+2\epsilon})}\\
&={\lambda v_j-\lambda v_k+\caO(N^{-1/2+3\epsilon})}\,,
\end{align*}
with high probability on $\Omega_V$. Since $|v_j-v_k|\ge C N^{-1/(\b+1)-\epsilon}$, $j\not=k$, on $\Omega_V$, we obtain
\begin{align}
 \im G_{jj}(\mu_k+\ii\eta_0)\le \frac{C}{(\lambda v_j-\lambda v_k)^2}N^{-1/2+3\epsilon}\,,
\end{align}
with high probability on $\Omega_V$, for some constant $C$.
\end{proof}

\begin{proof}[Proof of Theorem~\ref{thm:local}: Inequality~\eqref{mass_j}]
 From the spectral decomposition of the resolvent of $H$, we obtain
\begin{align*}
 \im G_{jj}(\mu_k+\ii\eta_0)&=\sum_{\alpha=1}^N\frac{|u_\alpha(j)|^2\eta_0}{(\mu_{\alpha}-\mu_k)^2+\eta_0^2}\ge \frac{1}{\eta_0}|u_k(j)|^2\,,
\end{align*}
thus, using the bound~\eqref{eq tapiridae}, we conclude that
\begin{align*}
 |u_k(j)|^2&\le C\frac{\eta_0}{(\lambda v_k-\lambda v_j)^2}N^{-1/2+3\epsilon}\le C\frac{1}{(\lambda v_k-\lambda v_j)^2}N^{-1+2\epsilon}\,,
\end{align*}
with high probability on $\Omega_V$, for $k\in\llbracket 1, n_0-1\rrbracket$, $j\in\llbracket 1,N\rrbracket$, $(j\not=k)$. This completes the proof of Theorem~\ref{thm:local}.
\end{proof}

\begin{appendices}
\section{}

In this appendix, we estimate the probabilities for the events $1.$-$3.$ in the definition of $\Omega_V$; see Definition \ref{v assumptions}. Recall the definition of the constants $\epsilon$ in~\eqref{epsilon condition} and $\kappa_0$ in~\eqref{definition of kappa0}. In the following, we denote, unlike in the rest of the paper, by $(v_i)_{i=1}^N$ (unordered) sample points distributed according to the measure $\mu$, (with $\b>1$). We denote by $(v_{(i)})$ the order statistics of $(v_i)$ with the convention $v_{(1)}\ge v_{(2)}\ge\ldots\ge v_{(N)}$. 

\begin{lem} \label{lem:app_1}
Let $(v_{(i)})$ be the order statistics of sample points $(v_i)$ under the probability distribution $\mu$ with $\b>1$. Let $n_0 > 10$ be a fixed positive integer independent of $N$. Then, for any $k \in \llbracket 1, n_0-1 \rrbracket$ and for any sufficiently small $\epsilon > 0$, we have
\begin{align}
\p \left( N^{-\epsilon} \kappa_0 < |v_{(k)} - v_{(j)}| < (\log N) \kappa_0\,, \forall j \neq k \right) \geq 1 - C (\log N)^{1+2\b} N^{-\epsilon}\,.
\end{align}
In addition, for $k=1$, we have
\begin{align}
\p \left( N^{-\epsilon} \kappa_0 < |1 - v_{(1)}| < (\log N) \kappa_0 \right) \geq 1 - C N^{-\epsilon (\b+1)}\,.
\end{align}
\end{lem}

\begin{proof}
Consider the following claims:
\begin{enumerate}

\item[$(1)$] There exists a constant $C > 0$ such that
$$
\p \left( |1 - v_{(1)}| > N^{-\epsilon} \kappa_0 \right) \geq 1 - C N^{-\epsilon (\b+1)}\,.
$$

\item[$(2)$] There exists a constant $c > 0$ such that 
$$
\p \left( |1 - v_{(n_0)}| > (\log N) \kappa_0 \right) \leq e^{-c (\log N)^{\b+1}}\,.
$$

\item[$(3)$] There exists a constant $C > 0$ such that, for any $k \in \llbracket 1, n_0-1 \rrbracket$,
$$
\p \left( |v_{(k)} - v_{(k+1)}| \leq N^{-\epsilon} \kappa_0 \right) \leq C (\log N)^{1+2\b} N^{-\epsilon}\,.
$$

\end{enumerate}

Assuming the claims $(1)$-$(3)$, it is easy to see that the desired lemma holds.

For a random variables $v$ with law $\mu$ as in~\eqref{jacobi measure}, we have for any $x \geq 0$,  
\begin{align} \label{mu tail}
C^{-1} x^{\b+1} \leq \p (1 - v \leq x) \leq C x^{\b+1}\,,
\end{align}
for some constant $C > 1$. Thus, we obtain for the first order statistics of $(v_i)$ that
\begin{align}
\p ( |1 - v_{(1)}| > N^{-\epsilon} \kappa_0 ) = \left( 1 - \p (1 - v \leq N^{-\epsilon} \kappa_0) \right)^N \geq \left( 1 - C N^{-\epsilon (\b+1)} N^{-1} \right)^N \geq 1 - C N^{-\epsilon (\b+1)}\,,
\end{align}
proving claim $(1)$.

Similarly, we have
\begin{align} \begin{split}
&\p \left( |1 - v_{(n_0)}| > (\log N) \kappa_0\right) \leq \binom{N}{n_0} \left( 1 - \p \left( 1 - v \leq (\log N) \kappa_0 \right) \right)^{N-n_0} \\
&\qquad\leq N^{n_0} \left( 1 - C^{-1} (\log N)^{\b+1} N^{-1} \right)^{N-n_0} \leq C N^{n_0} e^{-c (\log N)^{\b+1}} \leq e^{-c'(\log N)^{\b+1}}\,,
\end{split} \end{align}
for some constant $c, c' > 0$, proving claim $(2)$.

To prove claim $(3)$, we assume that $N^{-\epsilon} \kappa_0 < |1 - v_{(1)}| \leq |1 - v_{(n_0)}| \leq (\log N) \kappa_0$, which holds with probability higher than $1 - C N^{-\epsilon (\b+1)}$. Let
$$
I_j \deq \left[ 1 - (j+1) N^{-\epsilon} \kappa_0, 1 - (j-1) N^{-\epsilon} \kappa_0 \right]\,,\quad \qquad (j\in\llbracket 1, (\log N) N^{\epsilon}\rrbracket)\,.
$$
Then, it can easily be seen that if $|v_{(k)} - v_{(k+1)}| > N^{-\epsilon} \kappa_0$, then $v_{(k)}, v_{(k+1)} \in I_j$ for some $j\in\llbracket 1, (\log N) N^{\epsilon}\rrbracket$. Letting $p_j \deq\p ( v \in I_j )$, we have that
$$
\p ( |\{ i \in \llbracket 1, N \rrbracket : v_i \in I_j \}| = 0 ) = (1- p_j)^N, \qquad \p ( |\{ i \in \llbracket 1, N \rrbracket : v_i \in I_j \}| = 1 ) = N p_j (1- p_j)^{N-1}\,,
$$
hence,
\begin{align}
\p (v_{(k)}, v_{(k+1)} \in I_j) \leq \p ( |\{ i \in \llbracket 1, N \rrbracket : v_i \in I_j \}| = 2 ) \leq 1 - (1-p_j)^N - Np_j (1-p_j)^{N-1} \leq N^2 p_j^2\,.
\end{align}
Since
$$
p_j \leq C N^{-\epsilon} \kappa_0 \left( (\log N) \kappa_0 \right)^{\b} = C (\log N)^{\b} N^{-1 -\epsilon}\,,
$$
we have
\begin{align}
\p( |v_{(k)} - v_{(k+1)}| \leq N^{-\epsilon} \kappa_0 ) \leq \sum_{j\in\llbracket 1,(\log N) N^{\epsilon}\rrbracket}^{} N^2 p_j^2\le C (\log N)^{1+2\b} N^{-\epsilon}\,.
\end{align}
This proves the third part of the claim and completes the proof of the lemma.
\end{proof}
Next, we estimate the probability of condition $(2)$ in Definition~\ref{v assumptions} to hold.

\begin{lem} \label{lem:app_2}
Assume the conditions in Lemma \ref{lem:app_1}. Recall the definition of $\caD_{\epsilon}$ in \eqref{domain}. Then, for any fixed $\ell>0$, there exists a constant $C_{\ell}$ (independent of $N$) such that
\begin{align} \label{eq:CLT_fixed}
\p \left( \bigcup_{z \in \caD_{\epsilon}} \ \left \{ \left| \frac{1}{N} \sum_{i=1}^N \frac{1}{\lambda v_i - z - m_{fc}(z)} - \int \frac{\dd \mu(v)}{\lambda v - z - m_{fc}(z)} \right| > \frac{N^{3\epsilon /2}}{\sqrt N} \right \} \right) \leq C_{\ell} N^{-\ell}.
\end{align}
\end{lem}

\begin{proof}
Fix $z \in \caD_{\epsilon}$. For $i\in\llbracket 1,N\rrbracket$, let $X_i \equiv X_i(z)$ be the random variable
$$
X_i \deq \frac{1}{\lambda v_i - z - m_{fc}(z)} - m_{fc}(z) = \frac{1}{\lambda v_i - z - m_{fc}(z)} - \int \frac{\dd \mu(v)}{\lambda v - z - m_{fc}(z)}\,.
$$
By definition, $\E X_i = 0$. Moreover, we have
$$
\E |X_i|^2 \leq \int \frac{\dd \mu(v)}{|\lambda v - z - m_{fc}(z)|^2} = \frac{\im m_{fc}((z)}{\eta + \im m_{fc}(z)} < 1\,,\quad\qquad (z\in\C^+) \,,
$$
and, for any positive integer $p \geq 2$,
$$
\E |X_i|^{p} \leq \frac{1}{\eta^{p-2}} \E |X_i|^2 \leq N^{(1/2 +\epsilon) (p-2)}\,,\qquad\quad (z\in\caD_{\epsilon})\,.
$$
We now consider
\begin{equation} \label{2p expand}
\E \left| \frac{1}{N} \sum_{i=1}^N X_i \right|^{2p} = \frac{1}{N^{2p}} \sum_{i_1, i_2, \ldots, i_{2p}=1}^N \E \prod_{j=1}^p X_{i_j} \prod_{k=p+1}^{2p} \overline{X}_{i_k} \,.
\end{equation}
For fixed $i_1, i_2, \ldots, i_{2p}$, define 
$$
d_i \deq \sum_{j=1}^{2p} \mathbbm{1} (i_j = i)\,,
$$
for $i\in\llbracket 1,N\rrbracket$. 

 We first estimate the summand in the right side of \eqref{2p expand} when $d_1 \geq d_2 \geq \ldots \geq d_r \geq 1$ and $d_{r+1} = d_{r+2} = \ldots = d_N = 0$, for some $r\in\llbracket 1, 2p \rrbracket$. Since $(X_i)$ are independent and centered, we have
$$
\E \prod_{j=1}^p X_{i_j} \prod_{k=p+1}^{2p} \overline{X}_{i_k} = 0\,,
$$
if $d_i = 1$, for some $i$, which shows that we may assume $r \leq p$ and $d_r \geq 2$. When $d_r \geq 2$, we obtain that
$$
\left| \E \prod_{j=1}^p X_{i_j} \prod_{k=p+1}^{2p} \overline{X}_{i_k} \right| \leq \E \prod_{k=1}^r |X_k|^{d_k} \leq \prod_{k=1}^r N^{(1/2 + \epsilon)(d_k -2)} \leq N^{(1/2 + \epsilon)(2p -2r)}.
$$

Next, using the estimate obtained above, we can bound
\begin{align} \begin{split}
&\sum_{i_1, i_2, \ldots, i_{2p}} \E \prod_{j=1}^p X_{i_j} \prod_{k=p+1}^{2p} \overline{X}_{i_k} \\
&\qquad\leq \sum_{r=1}^p r! \binom{N}{r} N^{(1/2 + \epsilon)(2p -2r)} \leq p! \sum_{r=1}^p N^r N^{(1/2 +\epsilon) (2p-2r)} \leq (p+1)! N^p N^{2 \epsilon p}\,.
\end{split} \end{align}

Hence, by Markov's inequality, we obtain that
\begin{align}
\p \left( \left| \frac{1}{N} \sum_{i=1}^N X_i \right| \geq \frac{N^{3\epsilon /2}}{2 \sqrt N} \right) \leq \left( \frac{1}{2} N^{3\epsilon /2} \right)^{-2p} N^p \E \left| \frac{1}{N} \sum_{i=1}^N X_i \right|^{2p} \leq 4^p (p+1)! N^{-\epsilon p}\,.
\end{align}

To obtain the uniform bound on $\caD_{\epsilon}$, we choose a lattice $\caL$ in $\caD_{\epsilon}$ such that for any $z \in \caD_{\epsilon}$ there exists $z' \in \caL$ satisfying $|z - z'| \leq N^{-2}$. Since, for $z \in \caD_{\epsilon}$ and $z' \in \caL$,
$$
\left| \frac{1}{\lambda v - z - m_{fc}(z)} - \frac{1}{\lambda v - z' - m_{fc}(z')} \right| \leq C \frac{1}{\eta_0^2} |z-z'| \leq C N^{-1 + 2\epsilon}\,,
$$
we find that
\begin{align} \begin{split}
\p \left( \bigcup_{z \in \caD_{\epsilon}} \ \left \{ \left| \frac{1}{N} \sum_{i=1}^N \frac{1}{\lambda v_i - z - m_{fc}(z)} -m_{fc}(z) \right| > \frac{N^{3\epsilon /2}}{\sqrt N} \right \} \right) 
&\leq |\caL| 4^p (p+1)! N^{-\epsilon p}\nonumber\\ &\leq C N^4 4^p (p+1)! N^{-\epsilon p}\,.
\end{split} \end{align}
Setting $p = (\ell + 4) / \epsilon$, we obtain the desired lemma. 
\end{proof}

To estimate the probability for the third condition in Definition~\ref{v assumptions}, we need the following two auxiliary lemmas. Recall the definition of $R_2$ in~\eqref{definition of R2 without hat}.

\begin{lem} \label{R2 estimate}
If $0 < C^{-1} \eta \leq \im m_{fc}(z) \leq C\,\eta$, $z=E+\ii\eta$, for some constant $C \geq 1$, then we have
\begin{align}
\frac{1}{1+ C} \leq R_2 (z) \leq \frac{C}{1+ C}\,.
\end{align}
\end{lem}

\begin{proof}
Since
$$
\im m_{fc} (z) = \int \frac{\im z + \im m_{fc}(z)}{|\lambda v - z - m_{fc}(z)|^2} \dd \mu (v)\,,
$$
we have that
$$
\frac{C^{-1}}{1+ C^{-1}} \leq R_2 (z) = \frac{\im m_{fc}(z)}{\im z + \im m_{fc}(z)} \leq \frac{C}{1+ C}\,.
$$
\end{proof}

The imaginary part of $m_{fc}(z)$ can be estimated using the following lemma.

\begin{lem} \label{im mfc bound}
Assume that $\mu_{fc}$ has support $[L_-, L_+]$ and there exists a constant $C>1$ such that
\begin{equation}
C^{-1} \kappa^{\b} \leq \mu_{fc} (z) \leq C \kappa^{\b}\,,
\end{equation}
for any $0 \leq \kappa \leq L_+$. Then,
\begin{itemize}
\item[$(1)$] for $z = L_+ - \kappa + i \eta$ with $0 \leq \kappa \leq L_+$ and $0 < \eta \leq 3$, there exists a constant $C>1$ such that
\begin{equation}
C^{-1} (\kappa^{\b} + \eta) \leq \im m_{fc} (z) \leq C (\kappa^{\b} + \eta)\,;
\end{equation}
\item[$(2)$] for $z = L_+ + \kappa + i \eta$ with $0 \leq \kappa \leq 1$ and $0 < \eta \leq 3$, there exists a constant $C>1$ such that
 \begin{equation}
C^{-1} \eta \leq \im m_{fc} (z) \leq C \eta\,.
\end{equation}
 \end{itemize}
\end{lem}

\begin{rem}
Lemma \ref{im mfc bound} shows that there exists a constant $C_{\b} > 1$ such that
\begin{align}
C_{\b}^{-1} \eta \leq \im m_{fc} (z) \leq C_{\b} \eta\,,
\end{align}
for all $z \in \caD_{\epsilon}$ satisfying $L_+ - \re z \leq N^{\epsilon} \kappa_0$.
\end{rem}

Assuming Lemma \ref{im mfc bound}, we have the following estimate. Recall that $\caD_{\epsilon}$ is defined in \eqref{domain}.

\begin{lem} \label{lem:app_3}
Assume the conditions in Lemma \ref{lem:app_1}. Then, there exist constants $\mathfrak{c} < 1$ and $C > 0$, independent of~$N$, such that, for any $z = E + \ii \eta \in \caD_{\epsilon}$ satisfying
\begin{align} \label{condition app3}
\min_{i \in \llbracket 1, N \rrbracket} |\re (z + m_{fc}(z)) - \lambda v_{(i)}| = |\re (z + m_{fc}(z)) - \lambda v_{(k)}|\,,
\end{align}
for some $k \in \llbracket 1, n_0 -1 \rrbracket$, we have
\begin{align}
\p \left( \frac{1}{N} \sum_{i:i \neq k}^N \frac{1}{|\lambda v_{(i)} - z - m_{fc}(z)|^2} < \mathfrak{c} \right) \geq 1 - C (\log N)^{1+2\b} N^{-\epsilon}.
\end{align}
\end{lem}

\begin{proof}
Without loss of generality, we only prove the case $k=1$;  the general case can easily be shown by using the same argument. In the following, we assume that $N^{-\epsilon} \kappa_0 < |1 - v_{(1)}| < (\log N) \kappa_0$, and $|v_{(1)} - v_{(2)}| > N^{-\epsilon} \kappa_0$. 

Recall the definition of $R_2$ in~\eqref{definition of R2 without hat}. For $i\in\llbracket 1,N\rrbracket$, let $Y_i \equiv Y_i(z)$ be the random variable
$$
Y_i(z) \deq \frac{1}{|\lambda v_i - z - m_{fc}(z)|^2}\,, \qquad\quad(z\in\C^+)\,,
$$
and observe that $\E Y_i = R_2 < 1$, ($z\in\C^+$). Moreover, combining Lemma \ref{R2 estimate} and Lemma \ref{im mfc bound}, we find that there exists a constant $c < 1$, independent of $N$, such that that $R_2 (z) < c$ uniformly for all $z \in \caD_{\epsilon}$ satisfying \eqref{condition app3}. Since $\im (z + m_{fc}(z)) \ge \eta$, we also have that $Y_i(z) \leq \eta^{-2}$.

We first consider the special choice $E = L_+$. Let $\wt Y_i$ be the truncated random variable defined by
$$
\wt Y_i \deq
	\begin{cases}
	Y_i\,, & \text{ if  } Y_i < N^{2\epsilon} \kappa_0^{-2}\,, \\
	N^{2\epsilon} \kappa_0^{-2}\,, & \text{ if  } Y_i \geq N^{2\epsilon} \kappa_0^{-2}\,.
	\end{cases}
$$
Notice that, using the estimate \eqref{mu tail}, we have for $z = L_+ + \ii \eta \in \caD_{\epsilon}$ that
$$
\p ( Y_i \neq \wt Y_i ) \le C N^{-1 - (\b+1)\epsilon}\,.
$$
Let
$$
S_N \deq \sum_{i=1}^N Y_i\,, \qquad \wt S_N \deq \sum_{i=1}^N \wt Y_i\,.
$$
Then,
\begin{align}
\p (S_N \neq \wt S_N) \leq C N^{- (\b+1)\epsilon}.
\end{align}

We now estimate the mean and variance of $\wt Y_i$. From the trivial estimate $\p (Y_i \geq x) \leq \p (Y_i \neq \wt Y_i)$ for $x \geq N^{2\epsilon} \kappa_0^{-2}$, we find that 
\begin{align}
\E Y_i - \E \wt Y_i \leq \int_{N^{2\epsilon} \kappa_0^{-2}}^{\eta^{-2}} \p ( Y_i \neq \wt Y_i ) \dd x \leq C' N^{-(\b-1)\epsilon}\,,
\end{align}
for some $C' > 0$. Hence, we get
\begin{align}
\E \wt Y_i^2 \leq N^{2\epsilon} \kappa_0^{-2} \E \wt Y_i \leq N^{2\epsilon} \kappa_0^{-2} \E Y_i \leq N^{2\epsilon} \kappa_0^{-2}.
\end{align}
We thus obtain that
\begin{align} \begin{split}
&\p \left( \left| \frac{S_N}{N} - \E Y_i \right| \geq C' N^{-(\b-1)\epsilon} + N^{-\epsilon} \right) \leq \p \left( \left| \frac{\wt S_N}{N} - \E \wt Y_i \right| \geq N^{-\epsilon} \right) + \p (S_N \neq \wt S_N) \\
&\qquad\leq \frac{N^{2\epsilon} \E \wt Y_i^2}{N} + C N^{- (\b+1)\epsilon} \leq C N^{-\frac{\b-1}{\b+1} + 4\epsilon} = C N^{-2\fb + 4\epsilon},
\end{split} \end{align}
hence, for a constant $c$ satisfying $R_2 + C' N^{-(\b-1)\epsilon} + N^{-\epsilon} < c < 1$,
$$
\p \left( \frac{1}{N} \sum_{i=1}^N \frac{1}{|\lambda v_i - z - m_{fc}(z)|^2} < c \right) \geq 1 - \p \left( \left| \frac{S_N}{N} - \E Y_i \right| \geq C' N^{-(\b-1)\epsilon} + N^{-\epsilon} \right) \geq 1 - C N^{-2\fb + 4\epsilon}.
$$
This proves the desired lemma for $E = L_+$.

Before we extend the result to general $z \in \caD_{\epsilon}$, we estimate the probabilities for some typical events we want to assume. Consider the set
$$
\Sigma_{\epsilon} \deq \{ v_i : |1 - v_i| \leq N^{3\epsilon} \kappa_0 \}\,,
$$
and the event
$$
\Omega_{\epsilon} \deq \{ |\Sigma_{\epsilon}| < N^{3\epsilon (\b+2)} \}\,.
$$
Since we have from the estimate \eqref{mu tail} that
$$
\p (|1 - v_i| \geq N^{3\epsilon} \kappa_0) \leq C N^{-1 + 3(\b+1)\epsilon}\,,
$$
we find, using a Chernoff bound, that
$$
\p (\Omega_{\epsilon}^c) \leq \exp \left( -C \epsilon (\log N) N^{3\epsilon} N^{3(\b+1) \epsilon} \right)\,,
$$
for some constant $C$. Notice that we have, for $v_i \notin \Sigma_{\epsilon}$,
\begin{align} \label{eq:v_i notin Sigma}
L_++\re m_{fc}(L_++\ii\eta)-\lambda v_i > N^{3\epsilon} \kappa_0 \gg \eta + \im m_{fc}(L_+ + \ii \eta)\,,
\end{align}
where we used Lemma \ref{mfc estimate}, i.e., $|L_+ + \ii \eta + m_{fc}(L_+ +\ii\eta) - \lambda| = \caO (\eta)$, and that $\lambda>1$. We now assume that~$\Omega_{\epsilon}$ holds and that
$$
\frac{1}{N} \sum_{ i=1}^N \frac{1}{|\lambda v_i - (L_+ + \ii \eta) - m_{fc}(L_+ + \ii \eta) |^2} < c < 1\,.
$$
Further, we recall that the condition \eqref{condition app3} implies
$$
\re (z + m_{fc}(z)) \geq \lambda v_{(n_0)},
$$
which yields, together with Lemma \ref{mfc estimate} and Lemma \ref{lem:app_1}, that $E \geq L_+ - N^{\epsilon} \kappa_0$ with probability higher than $1 - C (\log N)^{1+2\b} N^{-\epsilon}$. We therefore assume in the following that $E \geq L_+-N^{\epsilon}\kappa_0$.

Consider the following two choices for such energies $E$:
\begin{enumerate}
\item[$(1)$] When $L_+ - N^{\epsilon} \kappa_0 \leq E \leq L_+ + N^{2\epsilon} \kappa_0$, we have that
$$
|\lambda v_i - z - m_{fc}(z)| = |\lambda v_i - (L_+ + \ii \eta) - m_{fc}(L_+ + \ii \eta)| + \caO(N^{2\epsilon} \kappa_0)\,,
$$
where we used Lemma \ref{mfc estimate}. Hence, using \eqref{eq:v_i notin Sigma}, we obtain for $v_i \notin \Sigma_{\epsilon}$ that
\begin{align} \begin{split}
\frac{1}{|\lambda v_i - z - m_{fc}(z)|^2} &\leq \frac{1}{|\lambda v_i - (L_+ + \ii \eta) - m_{fc}(L_+ + \ii \eta)|^2} + \frac{N^{2\epsilon} \kappa_0}{|\lambda v_i - (L_+ + \ii \eta) - m_{fc}(L_+ + \ii \eta)|^3} \\
&\leq \frac{1 + C N^{-\epsilon}}{|\lambda v_i - (L_+ + \ii \eta) - m_{fc}(L_+ + \ii \eta)|^2}\,.
\end{split} \end{align}
We thus have that
\begin{align} \begin{split}
&\frac{1}{N} \sum_{i=2}^N \frac{1}{|\lambda v_{(i)} - z - m_{fc}(z)|^2} \\
&\quad\leq \frac{N^{3\epsilon (\b+2)}}{N} \frac{1}{(N^{-\epsilon} \kappa_0)^2} + \frac{1}{N} \sum_{i: v_i \notin \Sigma_{\epsilon}} \frac{1 + C N^{-\epsilon}}{|\lambda v_i - (L_+ + \ii \eta) - m_{fc}(L_+ + \ii \eta)|^2} \\
&\quad\leq N^{-\epsilon} + \frac{1}{N} \sum_{i=1}^N \frac{1 + C N^{-\epsilon}}{|\lambda v_i - (L_+ + \ii \eta) - m_{fc}(L_+ + \ii \eta)|^2} < c < 1\,,
\end{split} \end{align}
where we also used the assumption that $|v_{(2)}-v_{(1)}|\ge N^{-\epsilon} \kappa_0$. 

\item[$(2)$] When $E > L_+ + N^{2\epsilon} \kappa_0$, we have
$$
(E - L_+) + \left( \re m_{fc}(E + \ii \eta) - \re m_{fc}(L_+ + \ii \eta) \right) \gg \eta + \im m_{fc}(E + \ii \eta)\,,
$$
where we again used Lemma \ref{mfc estimate}, hence, from \eqref{eq:v_i notin Sigma} we obtain that
$$
|\lambda v_i - z - m_{fc}(z)| \geq |\lambda v_i - (L_+ + \ii \eta) - m_{fc}(L_+ + \ii \eta)|\,.
$$
We may now proceed as in $(1)$ to find that
\begin{align} \begin{split}
&\frac{1}{N} \sum_{i=1}^N \frac{1}{|\lambda v_i - z - m_{fc}(z)|^2} \\
&\quad\leq N^{-\epsilon} + \frac{1}{N} \sum_{i=1}^N \frac{1}{|\lambda v_i - (L_+ + \ii \eta) - m_{fc}(L_+ + \ii \eta)|^2} < c < 1\,.
\end{split} \end{align}

\end{enumerate}

Since we proved in Lemma \ref{lem:app_1} that the assumptions $N^{-\epsilon} \kappa_0 < |1 - v_{(1)}| < (\log N) \kappa_0$ and $|v_{(1)} - v_{(2)}| > N^{-\epsilon} \kappa_0$ hold with probability higher than $1 - C (\log N)^{1+2\b} N^{-\epsilon}$, we find that the desired lemma holds for any $z\in\caD_{\epsilon}'$.
\end{proof}

To conclude this appendix, we prove Lemma~\ref{im mfc bound}.
 
\begin{proof}[Proof of Lemma~\ref{im mfc bound}]
We start with the claim $(1)$: Notice that
\begin{equation} \label{im mfc}
\im m_{fc}(z) = \im \int \frac{\dd\mu_{fc}(x)}{x-z} = \eta \int \frac{\dd \mu_{fc}(x) }{(x-L_++\kappa)^2 + \eta^2}\,.
\end{equation}
When $\eta \geq 1/2$, we may use the trivial bound
$$
\im m_{fc}(z) \sim \eta \int \frac{\dd \mu_{fc}(x) }{\eta^2} = \eta^{-1} \sim 1 \sim (\kappa^{\beta} + \eta)\,.
$$
When $\kappa \geq 1/2$, we can easily see from \eqref{im mfc} that $\im m_{fc}(z) \sim 1 \sim (\kappa^{\beta} + \eta)$. Thus, in the following, we only consider the case $\kappa, \eta < 1/2$.

To prove the lower bound, we notice that
$$
\im m_{fc}(z) \geq \eta \int_{L_+ - \kappa - 2\eta}^{L_+ - \kappa - \eta} \frac{\dd \mu_{fc}(x) }{(x-L_++\kappa)^2 + \eta^2} \geq C \eta \int_{L_+ - \kappa - 2\eta}^{L_+ - \kappa - \eta} \frac{(\kappa+\eta)^{\beta} \dd x }{\eta^2} \geq C (\kappa+\eta)^{\beta} \geq C \kappa^{\beta}\,.
$$
We also have that
$$
\im m_{fc}(z) \geq \eta \int_{0}^{1} \frac{\dd \mu_{fc}(x) }{(x-L_++\kappa)^2 + \eta^2} \geq C \eta \int_{0}^{1} \dd x \geq C \eta\,.
$$
Thus, we find that $\im m_{fc}(z) \geq C(\kappa^{\beta} + \eta)$.

For the upper bound, we first consider the case $\kappa \geq \eta$, where we have
\begin{equation*} \begin{split}
\im m_{fc}(z)& \leq C \eta \int_{L_-}^{L_+ - \kappa - \eta} \frac{(L_+ -x)^{\beta} \dd x}{(x-L_++\kappa)^2} + C \eta \int_{L_+ - \kappa - \eta}^{L_+ - \kappa + \eta} \frac{(\kappa+\eta)^{\beta} \dd x }{\eta^2} + C \eta \int_{L_+-\kappa+\eta}^{L_+ } \frac{\kappa^{\beta} \dd x}{(x-L_++\kappa)^2} \\
&\leq C \eta \int_{\eta}^{L_+ - L_- - \kappa} \frac{(y+\kappa)^{\beta} \dd y}{y^2} + C (\kappa+\eta)^{\beta} + C \eta \int_{\eta}^{\kappa} \frac{\kappa^{\beta}}{y^2} \dd y\\
& \leq C (\kappa+\eta)^{\beta} + C \eta\\& \leq C(\kappa^{\beta} + \eta)\,.
\end{split} \end{equation*}
Here, we used that $(y+\kappa)^{\beta} \leq C(y^{\beta} + \kappa^{\beta})$. The calculation for the case $\kappa < \eta$ is similar, in fact, easier.

We now prove the claim $(2)$. As in the proof of the statement $(1)$, we only consider the case $\kappa, \eta < 1/2$. We have, for the lower bound, that
$$
\im m_{fc}(z) \geq \eta \int_{0}^{1} \frac{\dd \mu_{fc}(x) }{(x-L_+-\kappa)^2 + \eta^2} \geq C \eta \int_{0}^{1} \dd x \geq C \eta\,,
$$
and, for the upper bound,
$$
\im m_{fc}(z) \leq C \eta \int_{L_-}^{L_+} \frac{(L_+ -x)^{\beta} \dd x}{(x-L_+-\kappa)^2} \leq C \eta \int_{\kappa}^{L_+ -L_- +\kappa} \frac{y^{\beta}}{y^2} \dd y \leq C \eta\,.
$$
This completes proof of the desired lemma.
\end{proof}

\section{}\label{appendix to fluctuation average lemma}
In this appendix, we prove Lemma~\ref{Zlemma 1} and Lemma~\ref{Jensen lemma}.
\begin{proof}[Proof of Lemma~\ref{Zlemma 1}]
Set for $l,l'\in\llbracket 0, p\rrbracket$ and for $A=\llbracket n_0,N\rrbracket$,
\begin{align}
 \Gamma_{l,l'}(z)\equiv \Gamma_{l,l'}\deq\max\{|F_{ab}^{(\T,\T')}(z)|\,:\, a,b\in A,\, a\not=b, \T,\T'\subset A,\,|\T|\le l,\,|\T'|\le l'  \}\,.
\end{align}
For simplicity we drop the $z$-dependence from the notation and always work on $\Omega_V$. We first consider $\Gamma_{l,0}$, i.e., we set \mbox{$\T'
=\emptyset$}. Recalling that $\lone(\Xi)|F_{ab}(z)|\le (\varphi_N)^{c\xi}N^{-\fb}N^{\epsilon}$, we obtain $p\,\Gamma_{0,0}\le p(\varphi_N)^{c\xi}N^{-\fb+\epsilon}\ll 1$ on $\Xi$. From~\eqref{Zlemma expand 1}, we get, for $c\in A$,
\begin{align}
 F_{ab}^{(\T c,\emptyset)}=F_{ab}^{(\T,\emptyset)}-F_{ac}^{(\T,\emptyset)}F_{cb}^{(\T,\emptyset)}\,,
\end{align}
 and we obtain on $\Xi$,
\begin{align}\label{rhinozeros}
 \Gamma_{l+1,0}\le \Gamma_{l,0}+\Gamma_{l,0}^2\,.
\end{align}
Iterating~\eqref{rhinozeros}, we obtain
\begin{align}
 \Gamma_{l+1,0}\le \Gamma_{0,0}+\sum_{i=0}^l { \Gamma_{i,0}^2} \le \Gamma_{0,0}+(l \Gamma_{l,0})\Gamma_{l+1,0}\,.
\end{align}
Thus as long as $l\Gamma_{0,0}\le 1/4$, we obtain by induction on $l$, $\Gamma_{l+1,0}\le 2\Gamma_{0,0}$, on $\Xi$, proving~\eqref{Zlemma bound 1} for the special case $\T'=\emptyset$. To prove the claim for $\T'\not=\emptyset$, we fix $l=|\T|$ and observe that { \eqref{Zlemma expand 1}}, together with the assumption $\Gamma_{l,l'}\ll 1$, implies
\begin{align}
 \Gamma_{l,l'+1}\le \Gamma_{l,l'}+C\Gamma_{l,l'}^2\,,
\end{align}
for some numerical constant $C$. Iterating, we find
\begin{align}
 \Gamma_{l,l'+1}\le \Gamma_{l,0}+C\sum_{i'=0}^{l'}\Gamma_{l,i'}^2 { \leq \Gamma_{l,l'} + C l \Gamma_{l, l'} \Gamma_{l, l'+1}\,.}
\end{align}
Thus as long as {$Cl'\Gamma_{l,l'}\le 1/4$}, we obtain on $\Xi$ that { $\Gamma_{l,l'+1} \leq 2\Gamma_{l,0} \leq 4\Gamma_{0,0}$}, where we used that $\Gamma_{l,0}\le 2\Gamma_{0,0} (\le N^{-\fb})$. This proves~\eqref{Zlemma bound 1}.

To prove~\eqref{Zlemma bound 2}, we define, for $l\in\llbracket 1,p\rrbracket$,
\begin{align}
 \widetilde\Gamma_l\deq\max\left\{\left|\frac{F_{ab}^{(\emptyset,a)}}{G_{aa}^{(\T)}}\right|\,:\, a,b\in A,\,a\not=b,\,\T\subset A,\,a\not\in\T,\,|\T|\le l\right\}\,.
\end{align}
Note that $\widetilde\Gamma_0\le (\varphi_N)^{c\xi}N^{-1/2+2\epsilon}$, on $\Xi$. From~\eqref{Zlemma expand basic}, we have on $\Xi$
\begin{align}
 \widetilde\Gamma_{l+1}\le \widetilde\Gamma_l+C\widetilde{\Gamma}_l\Gamma_{l,l}^2\,,
\end{align}
for a numerical constant $C$. Iterating as above, we find
\begin{align}
 \widetilde\Gamma_{l+1}\le\widetilde\Gamma_0+C\sum_{i=0}^l\widetilde\Gamma_i\Gamma_{i,i}^2\le \widetilde\Gamma_0+ { 8} C\Gamma_{0,0}^2\sum_{i=0}^l\widetilde\Gamma_i\,.
\end{align}
Since ${ 8}Cl\Gamma_{0,0}^2\le 1/4$, on $\Xi$, we obtain $\widetilde\Gamma_{l+1}\le 2\widetilde\Gamma_0$, on $\Xi$. Upon using~\eqref{Zlemma bound 1} to bound $|F_{ab}^{(\emptyset,a)}-F_{ab}^{(\T',\T'')}|$, this proves~\eqref{Zlemma bound 2}.

The proof of~\eqref{initial estimate on Q_a} is similar to the proof of~\eqref{Zlemma bound 2} but easier.
\end{proof}

\begin{proof}[Proof of Lemma~\ref{Jensen lemma}]
First, we observe that for a random variables $\caX\equiv \caX(H)$,
\begin{align}
 \E^W|Q_b \caX|^p= \E^W|\caX-\E_b \caX|^p\le 2^{p-1}\E^W|\caX|^p+2^{p-1}\E^W|\E_b \caX|^p\,.
\end{align}
From Jensen's inequality for the partial expectation $\E_b$, we hence obtain
\begin{align}\label{jensen in proof}
 \E^W|Q_b\caX|^p\le 2^p\E^W|\caX|^p\,.
\end{align}
Next, let $h_i\deq 2\lceil\frac{2+p}{1+d_{i}} \rceil$, $i\in\llbracket 1,p\rrbracket$. One checks that $\sum_{i=1}^ph_i^{-1}\le 1$, $2\le h_{i}\le 2p+4$ and $h_i(d_{i}+1)\le 2p+4$. Thus H\"older's inequality gives
\begin{align}
 \left|\E^W\prod_{i=1}^q Q_{i}\caX_{i}\prod_{i=q+1}^p Q_{i}\caY_{i}\right|\le  2^q\prod_{i=1}^q\left(\E^W |\caX_{i}|^{h_{i}}| \right)^{1/{h_{i}}}\prod_{i=q+1}^p\left(\E^W|Q_{i}\caY_{i}|^{h_{i}} \right)^{1/{h_{i}}}\,,
\end{align}
where we used~\eqref{jensen in proof}. By assumption, we have
\begin{align}
 \E^W\left[ |\caX_{i}|^{h_{i}}\lone(\Xi^c)\right]\le (\E^W |\caX_{i}|^{2h_{i}})^{1/2}\mathbb{P}(\Xi^c)^{1/2}\le N^{2K h_i (d_{i}+1) }\mathbb{P}(\Xi^c)^{1/2}
&\ll \mathrm{e}^{-c(\log N)^{\xi}}\,.
\end{align}
Observing that $\frac{1}{h_{i}}\ge \frac{1+d_{i}}{4p}$, we thus get
\begin{align*}
 (\E^W |\caX_{i}|^{h_{i}}\lone(\Xi^c))^{1/h_{i}}
&\le\mathrm{e}^{-c\frac{1+d_{i}}{4p}(\log N)^{\xi}}\\
&\le\left(\mathrm{e}^{-c (\log N)^{3/2}}\right)^{1+d_{i}} \\
&\ll N^{-(1+d_{i})}\,,
\end{align*}
where we used that $p^{-1}\ge (\log N)^{-\xi+3/2}$.
In a similar way, one establishes
\begin{align*}
 (\E^W |\caY_{i}|^{h_{i}}\lone(\Xi^c))^{1/h_{i}} \ll N^{-1}\,.
\end{align*}
Together with the estimates~\eqref{estimates on good event}, valid on the event $\Xi$, we obtain
\begin{align}
 \left|\E^W\prod_{i=1}^q Q_{i}\caX_{i}\prod_{i=q+1}^p Q_{i}\caY_{i}\right|\le (\varphi_N)^{c_0\xi p} N^{-\sum_{i=1}^q (d_{i}-1)(\fb/2-\epsilon)}N^{2p\epsilon}\,.
\end{align}
Recalling that we have set $\sum_{i=1}^qd_{i}-1= s$, the claim follows.
\end{proof}

\section{} \label{Gaussian fluctuation}

In this appendix, we consider the setup $\lambda < \lambda_+$. Recall that $\widehat \mu_{fc} = \widehat \mu \boxplus \mu_{sc}$ denotes the free convolution measure of the empirical measure $\widehat \mu$, which is defined in \eqref{def hat mu}, and the semicircular measure~$\mu_{sc}$. Also recall that we denote by $L_{\pm}$ the endpoints of the support of the measure $\mu_{fc}$.

\begin{lem} \label{lem:gaussian}
Let $\mu$ be a centered Jacobi measure defined in \eqref{jacobi measure} with $\b > 1$. Let $\supp \widehat\mu_{fc} = [ \widehat L_-, \widehat L_+]$, where~$\widehat L_-$ and~$\widehat L_+$ are random variables depending on $(v_i)$. Then, if $\lambda < \lambda_+$, the rescaled fluctuation $N^{1/2} (\widehat L_+ - L_+)$ converges to a Gaussian random variable with mean $0$ and variance $(1 - [m_{fc} (L_+)]^2)$ in distribution, as $N \to \infty$.
\end{lem}

\begin{rem}
When $\a > 1$, the analogous statement to Lemma \ref{lem:gaussian} holds at the lower edge. See also Remark \ref{remark about lower edge}.
\end{rem}

\begin{proof}
Following the proof in \cite{S2, LS}, we find that $\widehat L_+$ is the solution to the equations
\begin{align}
\widehat m_{fc} (\widehat L_+) = \frac{1}{N} \sum_{j=1}^N \frac{1}{\lambda v_j - \widehat L_+ - \widehat m_{fc}(\widehat L_+)}\,, \qquad \frac{1}{N} \sum_{j=1}^N \frac{1}{(\lambda v_j - \widehat L_+ - \widehat m_{fc}(\widehat L_+))^2} = 1\,.
\end{align}
Similarly, we find that $L_+$ is the solution to the equations
\begin{align}
 m_{fc} ( L_+) =\int\frac{\dd\mu(v)}{\lambda v -  L_+ -  m_{fc}( L_+)}\,, \qquad\int\frac{\dd\mu(v)}{(\lambda v -  L_+ -  m_{fc}( L_+))^2} = 1\,.
 \end{align}
Let
$$
\tau \deq L_+ + m_{fc}(L_+)\,, \qquad \widehat \tau \deq \widehat L_+ + \widehat m_{fc}(\widehat L_+)\,.
$$
Since $\lambda < \lambda_+$, we may assume that
\begin{align}
\int \frac{\dd \mu(v)}{(\lambda v - \lambda)^2} > 1 + \delta\,, \qquad \frac{1}{N} \sum_{j=1}^N \frac{1}{(\lambda v_j - \lambda)^2} > 1 + \delta\,,
\end{align}
for some $\delta > 0$. Notice that the second inequality holds with high probability on $\Omega_V$. From the assumption, we also find that $\tau, \widehat \tau > \lambda$. Thus we get
\begin{align}\label{equation C3}
0 &= \frac{1}{N} \sum_{j=1}^N \frac{1}{(\lambda v_j - \widehat \tau)^2} - 1 = \frac{1}{N} \sum_{j=1}^N \frac{1}{(\lambda v_j - \widehat \tau)^2} - \frac{1}{N} \sum_{j=1}^N \frac{1}{(\lambda v_j - \tau)^2} + \caO((\varphi_N)^{\xi} N^{-1/2}) \nonumber \\
&= \frac{1}{N} \sum_{j=1}^N \frac{(-2 \lambda v_j + \tau + \widehat \tau)(\tau - \widehat \tau)}{(\lambda v_j - \tau)^2 (\lambda v_j - \widehat \tau)^2} + \caO((\varphi_N)^{\xi} N^{-1/2})\,, 
\end{align}
which holds with high probability. Since $\tau, \widehat \tau > \lambda$, we have
$$
-2 \lambda v_j + \tau + \widehat \tau \geq 0\,.
$$
Moreover, with high probability, $|\{ v_j : v_j < 0 \}| > cN$ for some constant $c > 0$, independent of $N$. In particular,
$$
\frac{1}{N} \sum_{j=1}^N \frac{-2 \lambda v_j + \tau + \widehat \tau}{(\lambda v_j - \tau)^2 (\lambda v_j - \widehat \tau)^2} > c' > 0\,,
$$
for some constant $c'$ independent of $N$. This shows together with~\eqref{equation C3} that
$$
\tau - \widehat \tau = \caO((\varphi_N)^{\xi} N^{-1/2})\,,
$$
with high probability on $\Omega_V$. 

We can thus write
\begin{align}
\widehat m_{fc}(L_+) = \widehat \tau - \widehat L_+ &= \frac{1}{N} \sum_{j=1}^N \frac{1}{\lambda v_j - \widehat \tau} = \frac{1}{N} \sum_{j=1}^N \frac{1}{\lambda v_j - \tau} + \frac{1}{N} \sum_{j=1}^N \frac{\widehat \tau - \tau}{(\lambda v_j - \tau)^2} + \caO((\varphi_N)^{2\xi} N^{-1}) \nonumber \\
&= m_{fc}(L_+) + X + (\widehat \tau - \tau) + \caO((\varphi_N)^{2\xi} N^{-1})\,,
\end{align}
with high probability, where we define the random variable $X$ by
\begin{align}
X \deq \frac{1}{N} \sum_{j=1}^N \frac{1}{\lambda v_j - \tau} - \int \frac{\dd \mu(v)}{\lambda v - \tau} = \frac{1}{N} \sum_{j=1}^N \left( \frac{1}{\lambda v_j - \tau} - \E \left[ \frac{1}{\lambda v_j - \tau} \right] \right)\,.
\end{align}
Notice that, by the central limit theorem, $X$ converges to a Gaussian random variable with mean~$0$ and variance~$N^{-1} (1 - (m_{fc} (L_+))^2)$. Since,
\begin{align}
L_+ - \widehat L_+ = X + \caO((\varphi_N)^{2\xi} N^{-1})\,,
\end{align}
with high probability, the desired lemma follows.
\end{proof}

When $(v_i)$ are fixed, we may follow the proof of Theorem 2.21 in \cite{LS} to get
\begin{align}
|L_+ - \mu_1| \leq (\varphi_N)^{C \xi} N^{-2/3}
\end{align}
with high probability. Since $|\widehat L_+ - L_+| \sim N^{-1/2}$, we find that the leading fluctuation of the largest eigenvalue comes from the Gaussian fluctuation obtained in Lemma \ref{lem:gaussian}. This also shows that there is a sharp transition in the distribution of the largest eigenvalue from a Gaussian law to a Weibull distribution as $\lambda$ crosses $\lambda_+$.
\end{appendices}

\end{document}